\documentclass[11pt]{article}

\usepackage[letterpaper,margin=1in]{geometry}
\usepackage{microtype}
\usepackage{array}
\usepackage{graphicx}
\usepackage{subcaption}
\usepackage{booktabs}
\usepackage{hyperref}
\usepackage{url}
\usepackage{amsmath}
\usepackage{amssymb}
\usepackage{mathtools}
\usepackage{amsthm}
\usepackage[capitalize,noabbrev]{cleveref}
\usepackage[round,authoryear]{natbib}
\usepackage{algorithm}
\usepackage{algorithmic}
\usepackage{authblk}
\usepackage{xcolor}

\allowdisplaybreaks

\theoremstyle{plain}
\newtheorem{theorem}{Theorem}[section]

\theoremstyle{definition}
\newtheorem{definition}[theorem]{Definition}
\newtheorem{assumption}[theorem]{Assumption}
\theoremstyle{remark}
\newtheorem{remark}[theorem]{Remark}

\newcommand{\RR}{\mathbb R}

\newcommand{\PP}{\mathbb P}
\newcommand{\KL}{\mathrm{KL}}

\title{Duality and Policy Evaluation in Distributionally Robust Bayesian Diffusion Control}
\author[1]{Jose Blanchet}
\author[2]{Jiayi Cheng}
\author[1]{Yuewei Ling}
\author[1]{Hao Liu}
\author[3]{Yang Liu}
\affil[1]{Stanford University, CA 94305, USA}
\affil[2]{New York University, NY 10003, USA}
\affil[3]{The Chinese University of Hong Kong, Shenzhen, Guangdong 518172, China}
\date{}

\hypersetup{
  pdftitle={Duality and Policy Evaluation in Distributionally Robust Bayesian Diffusion Control},
  pdfauthor={Jose Blanchet, Jiayi Cheng, Yuewei Ling, Hao Liu, Yang Liu},
  colorlinks=true,
  linkcolor=blue,
  citecolor=blue,
  urlcolor=blue
}

\begin{document}
\maketitle
\vspace{-1.2em}
\begin{center}
\small \texttt{\{jose.blanchet,yueweiling,haoliu20\}@stanford.edu}\\[-0.1em]
\small \texttt{jiayicheng@nyu.edu; yangliu16@cuhk.edu.cn}
\end{center}
\vspace{0.5em}
\begin{abstract}
We study diffusion control problems under parameter uncertainty. Controllers based on plug-in estimation can be brittle due to potential distribution shifts. Bayesian control with a prior on the parameters offers a formulation with beliefs about such shifts. However, as with any Bayesian model, the prior may be misspecified. To mitigate misspecification and reduce over-pessimism compared to classical robust control approaches (e.g. \citet{hansen2008robustness}), we propose a distributionally robust Bayesian control (DRBC) formulation in which an adversary perturbs the prior within a divergence neighborhood of a baseline prior. We develop a strong duality result that reduces the distributionally robust prior evaluation to a low-dimensional optimization and yields a practical simulation-based policy evaluation and learning procedure with structured policy parameterizations. We validate the efficiency of the algorithm on a synthetic linear-quadratic control example and real-data portfolio selection.
\end{abstract}


\section{Introduction}
\label{sec:intro}
Decision-making under uncertainty is a core challenge in reinforcement learning and control. In many practical
settings, agents must act without knowing key environment parameters (e.g., transition dynamics, reward biases).
We consider a diffusion control problem for which a controller aims at maximizing running and terminal objectives by
making decisions informed by observations.
Since the controller cannot directly observe model parameters, the unknown parameter or factor is either estimated and then plugged into the closed-form solutions or it is
modeled as an unobservable random element with a prior distribution. These two approaches, known as plug-in methods and Bayesian control,
are well studied in the control literature and give rise to sophisticated policies that naturally work well if the estimation is accurate or 
the full Bayesian model is well specified (and is explicitly solvable in some canonical cases).

However, if the model is not well specified, these approaches will often deliver suboptimal results.
Adversarial approaches have been used to mitigate the impact of model misspecification. The controller interacts
with a fictitious adversary to maximize the value function, while the adversary selects a worst-case probability to
ensure policy robustness. Distributionally robust control (DRC) \citep{HansenSargent2001, hansen2008robustness}
is one such method, and it is built on finding robust formulations that are tractable in the sense of leading
to a dynamic programming principle. In exchange for this type of tractability, the approach leads to very
pessimistic policies, because the adversary's power is replenished at every point in time. This also makes
calibrating the size of the distributional uncertainty difficult, because small variations on this parameter
have a significant impact on performance.

To address over-conservatism, we consider a Distributionally Robust Bayesian Control (DRBC) formulation in which
we only build distributional robustness around the prior distribution in the Bayesian control formulation. This
allows us to combat pessimistic policies at the expense of losing the dynamic programming principle, and hence
we develop alternative methods for computing the optimal policy.
In the main text, we focus on the important case where the adversary chooses a prior from a
Kullback--Leibler uncertainty set, which leads to a particularly simple dual representation and an efficient
algorithm. Our formulation and analysis extend to general $\phi$-divergence neighborhoods; we defer these
extensions (including the Cressie--Read family as a concrete example) to the appendix.
Our motivating application is continuous-time control with unknown dynamics, but DRBC applies broadly to problems
where robustness to prior misspecification is critical. In particular, our DRBC formulation applies to many machine learning examples including classical control, reinforcement learning \citep{fleming2006controlled, yong1999stochastic, borkar2008stochastic}, finance \citep{ghosn1996multitask, borodin2003beatbeststock, qiu2015robustportfolio, luo2018efficientonlineportfolio, tsai2023datadependentops, lin2024msparsesr}, and operations research models \citep{harrison2004dynamic}.

\subsection{Our Contributions}
Our contributions are summarized as follows. 
\begin{itemize}
    \item We formulate a DRBC problem for misspecified priors in the context of diffusion control using a $\phi$-divergence uncertainty set around the prior.
    \item We prove a strong duality result that reformulates DRBC into a tractable optimization problem. We provide sample complexity results showing that policy evaluation is possible with the canonical rate $\mathcal{O}_p\left(n^{-1/2}\right)$ via a novel randomized multi-level Monte Carlo (rMLMC) estimator. 
    \item We develop a policy learning approach that designs the policy parameterization case-by-case using problem
structure. We illustrate this with a synthetic linear--quadratic (LQ) example using structured
belief-feature feedback and the Bayesian Merton problem using a neural-network parameterization. 

    \item We present simulation results to demonstrate the accuracy of our deep learning method and confirm the $\mathcal{O}_p\left(n^{-1/2}\right)$ convergence rate. Numerical results show the robustness of DRBC and it overcomes overpessimism in both synthetic and real-data experiments for LQ and Merton.
\end{itemize}

\subsection{Related Work}
\label{relatedwork}

Distributionally Robust Methods have been well-studied in a wide range of areas. For example, Distributionally Robust Optimization (DRO) (i.e. the supervised learning case) can be shown to recover a wide range of successful statistical estimators (including sqrt-Lasso, AdaBoost, group Lasso etc.), by carefully choosing the uncertainty set, often in terms of $\phi$-divergence or Wasserstein sets (see \citet{Blanchet2021WassersteinDRO, Blanchet2024DRO}). Refer also to \citet{rahimian2019distributionally} and \citet{Bayraksan2015} for comprehensive reviews on DRO. 

%\citep{JohnDuchi} and Wasserstein balls \citep{Abadeh2018, Blanchet2019, Gao2016}.  

Motivated by problems in areas such as economics and finance, among others, \citep{HansenSargent2001, hansen2008robustness, RB1, RB2} , DRO has been generalized to the setting of dynamic decision-making with model uncertainty. This situation is significantly more complicated and the literature has focused mostly on developing formulations that are amenable to dynamic programming (DP), giving rise to Distributionally Robust Control (DRC) and Distributionally Robust Markov Decision Processes (DRMDP). The availability of a dynamic programming principle facilitates the development of Distributionally Robust Reinforcement Learning (DRRL) and related settings \citep{NianSi, wang2022policy,Wang2023, Liu2022,Zhou2021, lu2024drrl}. 

However, to develop a DP in the DRC, DRMDP, and DRRL settings, the adversary gets its power replenished at every point in time, making the formulations pessimistic. That occurs at every point in time, thus making these formulations overconservative. 

\textit{In contrast, our formulation combines Bayesian stochastic control with DRO by introducing a single
distributional uncertainty set around the prior distribution. This combats overconservative solutions, but at
the expense of the dynamic programming principle. We develop formulations and techniques that make the optimal
solution learnable by exploiting the underlying control-problem structure.}

Computations of static DRO problems has been studied by \citet{levy2020large, blanchet2020semi, wang2021sinkhorn}. 
 We also mention the literature on RL in finance \citep{Hambly2023}, Bayesian Optimization \citep{Daulton22} and Distributionally Robust Bayesian Optimization (DRBO, \citep{pmlr-v108-kirschner20a}). Our setting is different from DRBO. We work in continuous time, which allows us to use stochastic analysis (via the martingale method and other techniques) to obtain convenient expressions to define a suitable loss for the optimal strategy. Our formulation is also offline and not online as in DRBO. We wish to efficiently evaluate policies for the game corresponding to our formulation. In addition, the work on Bayesian Distributionally Robust Optimization (BDRO, \citet{doi:10.1137/21M1465548}) is more related to our setting here, but it focuses on a static setting. \citet{blanchet_2025_bdrobust_merton} applies the DRBC idea to a specific financial setting and its method is only limited to that setting; In contrast, the DRBC solution method developed in this paper applies to more broad settings.

\section{Bayesian Control and Robust Priors Setup}
\label{sec:setup}
% Our setup has an unknown latent parameter
% independent of the driving Brownian motion, and is only partially revealed through an observed process. This randomizes the hard-to-estimate parameter in standard control problems and can be viewed as a Bayesian idea. 
\subsection{Bayesian Diffusion Control}
\label{sec:bayes}
We use \emph{control}, \emph{policy}, and \emph{strategy} interchangeably.
Let $(\Omega,\tilde{\mathcal F},P)$ be a complete probability space supporting:
(i) an $\RR^{m'}$-valued Brownian motion $W=(W_t)_{t\ge 0}$, and
(ii) the hard-to-estimate latent randomized parameter $B\in\RR^m$ (this can be a general random element, where for simplicity, we restrict to random vectors here) such that $B\perp\!\!\!\perp W$ under $P$.
We write $B\sim \mu$, where $\mu\in\mathcal P(\RR^m)$ is a baseline prior.
Let the \emph{full} filtration be the augmentation of $\mathcal G_t := \sigma(B, W_s: 0\le s\le t).$
We assume the controller observes a filtration $\mathcal F=(\mathcal F_t)_{t\ge 0}$ satisfying the usual
conditions and $\mathcal F_t\subseteq \mathcal G_t$ (the inclusion can be strict).
An admissible control $\pi=(\pi_t)_{t\ge 0}$ takes values in $U\subset\RR^k$ and is
$\mathcal F$-progressively measurable with $\int_0^T \|\pi_t\|_2^2\,dt < \infty, P\text{-a.s.}$

Denote the class of admissible controls by $\mathcal A(x_0)$.
Given $\pi\in\mathcal A(x_0)$, the controlled state process $X^\pi=(X_t^\pi)_{t\in[0,T]}$ (taking values in
$\RR^d$) satisfies the stochastic differential equation (SDE) with $X_0=x_0$ and
\begin{equation}
\label{eq:sde}
dX_t = a(t,B,X_t,\pi_t)\,dt + b(t,B,X_t,\pi_t)^\top dW_t,
\end{equation}
where $a,b$ are such that \eqref{eq:sde} is well-posed and the solution is $\mathcal F$-adapted.
We assume all integrability conditions needed for expectations below.
With running reward $g$ and terminal reward $h$, the Bayesian diffusion control objective is
\begin{equation}
\label{eq:bayes}
V^\mu(x_0)
:=\sup_{\pi\in\mathcal A(x_0)}
E^{P}\!\left[
\int_0^T g(t,X_t^\pi,\pi_t)\,dt + h(X_T^\pi)
\right].
\end{equation}
Let $\{P^b\}_{b\in\RR^m}$ denote a regular conditional probability of $P$ given $B=b$.
For a fixed policy $\pi$, define the conditional value for each $ b\in\RR^m$,
\begin{equation}
\label{eq:Zb}
Z_\pi(b)
:= E^{P^b}\!\left[
\int_0^T g(t,X_t^\pi,\pi_t)\,dt + h(X_T^\pi)
\right].
\end{equation}
Therefore,
\begin{equation}
\label{eq:disintegrate}
E^{P}\!\left[
\int_0^T g(t,X_t^\pi,\pi_t)\,dt + h(X_T^\pi)
\right]
=\int_{\RR^m} Z_\pi(b)\,\mu(db).
\end{equation}

\subsection{Distributionally Robust Bayesian Control via KL Prior Uncertainty}
\label{sec:drbc}

A baseline prior $\mu$ is rarely correct in practice.
We consider \emph{distributionally robust Bayesian control} (DRBC), where an adversary perturbs the prior on $B$
\emph{once at time $0$} while leaving the conditional dynamics given $B=b$ unchanged.

In the main paper we focus on a KL ambiguity set around $\mu$ with $\KL(\nu\|\mu):=\int \log\!\Big(\frac{d\nu}{d\mu}\Big)\,d\nu$:
\begin{equation}
\label{eq:klball}
\mathcal U_{\mathrm{KL}}(\mu,\delta)
:=\Big\{\nu\in\mathcal P(\RR^m): \KL(\nu\|\mu)\le \delta\Big\},
\end{equation}
with the convention $\KL(\nu\|\mu)=+\infty$ if $\nu\not\ll \mu$.

For a fixed policy $\pi$, define the robust value $\mathcal R_\delta(\pi)
:= \inf_{\nu\in \mathcal U_{\mathrm{KL}}(\mu,\delta)} \int_{\RR^m} Z_\pi(b)\,\nu(db),$
and the DRBC problem $V^{\mathrm{DRBC}}(x_0)=\sup_{\pi\in\mathcal A(x_0)} \mathcal R_\delta(\pi).$
The analysis extends to general $\phi$-divergence balls by replacing the
log-exponential transform in the dual by a convex conjugate term.
We treat the Cressie--Read family as a concrete non-KL example in appendix.

Another useful viewpoint is to define the uncertainty set at the level of probability measures on the underlying
space while ensuring that only the law of the latent factor $B$ is perturbed and the Brownian structure of $W$
is preserved. Let $\mathcal{P}(\Omega,\tilde{\mathcal F})$ be the set of probability measures on
$(\Omega,\tilde{\mathcal F})$ and define the subset of $\mathcal{P}(\Omega,\tilde{\mathcal F})$ such that 
\[
\mathcal Q_{\mathrm{KL}}(\mu,\delta)
=\left\{
\begin{array}{@{}l@{}}
Q(B\in A)=\nu(A),\ \forall A\in\mathcal{B}(\RR^{m}),\\
\text{for some }\nu\in\mathcal{P}(\RR^{m}), Q\ll P,\\
\KL(\nu\|\mu)\le\delta,\quad B\perp\!\!\!\perp W. \text{ Under }Q,\\ 
W\ \text{is a standard Brownian motion}.
\end{array}
\right\}.
\]
Endow $\mathcal Q_{\mathrm{KL}}(\mu,\delta)$ with the weak topology \citep{billingsley1986probability} and identify measures with the
same $B$-marginal:
\begin{equation}\label{relation}
Q_1\sim Q_2 \quad \Longleftrightarrow \quad Q_1\circ B^{-1}=Q_2\circ B^{-1}.
\end{equation}
The DRBC uncertainty set can then be taken as the quotient space $\mathcal U_{\mathrm{KL}}(\mu,\delta):=\mathcal Q_{\mathrm{KL}}(\mu,\delta)/\sim$,
which makes precise that DRBC is a reweighting of the prior on $B$ while keeping the conditional diffusion
dynamics (given $B$) unchanged.

\begin{theorem}\label{uncertaintyset}
$\mathcal U_{\mathrm{KL}}(\mu,\delta)$ is well-defined. For any $Q\in\mathcal U_{\mathrm{KL}}(\mu,\delta)$, there exists $\nu\ll\mu$ such that $\frac{dQ}{dP}=\frac{d\nu}{d\mu}(B)$ $P\text{-a.s.}$ and $\text{KL}(Q\|P)=\text{KL}(\nu\|\mu).$
\end{theorem}

In certain DRBC settings, one can justify a minimax interchange, yielding
\begin{align*}
\sup_{\pi\in\mathcal A(x_0)}\inf_{\nu\in\mathcal U(\mu,\delta)}
\int Z_\pi(b)\,\nu(db)=
\inf_{\nu\in\mathcal U(\mu,\delta)}\sup_{\pi\in\mathcal A(x_0)}
\int Z_\pi(b)\,\nu(db).
\end{align*}
When the inner Bayesian control problem admits a closed-form value for each candidate prior, this swap can further reduce DRBC to a constrained \emph{distributional optimization} over \(\nu\) \citep{blanchet_2025_bdrobust_merton}.
However, the interchange alone does \emph{not} ensure tractability for general diffusion control if Bayesian control value is not available in closed form (e.g. linear quadratic controls ensures a swap but no closed form).
Therefore, we develop a \emph{strong duality} theory that directly reformulates the inner robust prior evaluation \(\inf_{\nu\in\mathcal U(\mu,\delta)} \int Z_\pi(b)\,\nu(db)\) for a fixed policy \(\pi\). One can formulate DRBC using Wasserstein ambiguity around the prior, which is often a natural geometric choice; however, this induces significantly heavier inner optimization \citep{Blanchet2019}. In contrast, divergence ambiguity admits a representation that is particularly amenable to Monte Carlo estimators, motivating our focus on divergence balls.

\section{Strong Duality and Policy Evaluation}
\label{sec:duality}

This section first derives a strong duality representation for the inner robust prior problem $\mathcal R_\delta(\pi)$.
This dual form will be the starting point for policy evaluation and the Monte Carlo method.

\begin{theorem}
\label{thm:kl-dual}
Fix $\pi\in\mathcal A(x_0)$ and let $Z_\pi(b)$ be defined in \eqref{eq:Zb}.
Assume that for every $\lambda>0$ the integral
$\int_{\RR^m}\exp\!\left(-\frac{Z_\pi(b)}{\lambda}\right)\mu(db)$ is well-defined in $(0,\infty]$. Then
\begin{align}
\label{eq:kl-dual-lam0}
\mathcal{R}_{\delta}(\pi)=
\sup_{\lambda \ge 0}
\left\{
-\lambda\,\delta
-\lambda\log\!\left(\int_{\RR^m}\exp\!\left(-\frac{Z_\pi(b)}{\lambda}\right)\mu(db)\right)
\right\},
\end{align}
where the map in \eqref{eq:kl-dual-lam0} is extended at $\lambda=0$ by $\operatorname*{ess\,inf}_{b\sim\mu} Z_\pi(b),$ the essential infimum of the random variable $Z_\pi(b)$.
\end{theorem}

Therefore, the DRBC problem becomes a nested maximization problem
\[
\sup_{\pi\in\mathcal A(x_0)}\sup_{\lambda \ge 0}
\left\{
-\lambda\,\delta
-\lambda\log\!\left(\int_{\RR^m}\exp\!\left(-\frac{Z_\pi(b)}{\lambda}\right)\mu(db)\right)
\right\},
\]
which is generally intractable in practice.
In particular, the difficulties are twofold: the nested maximization is highly nonconvex and even
doing policy learning with fixed $\lambda$ is hard.
Motivated by alternating optimization schemes commonly used in distributionally robust learning \citep{NianSi},
we adopt a two-step procedure that alternates between \emph{policy evaluation} and \emph{policy learning}.
Concretely, given a candidate policy $\pi$, we maximize the dual objective in
\eqref{eq:kl-dual-lam0} over $\lambda \ge 0$ (policy evaluation) to obtain an approximate optimizer
$\hat\lambda(\pi)$ and an estimate of $\mathcal R_\delta(\pi)$.
Then, holding $\lambda=\hat\lambda(\pi)$ fixed, we update $\pi$ by ascending an empirical estimate of the
corresponding dual objective (policy learning), using problem-structured parameterizations.
We repeat these alternating updates until convergence of the dual variable and the robust value, yielding
a well-approximated robust policy $\hat\pi_{\mathrm{DRBC}}\approx \pi_{\mathrm{DRBC}}$.

To ensure the alternative maximization is well-posed and numerically stable, we impose the following mild regularity conditions for each fixed $\pi$.

\begin{assumption}
\label{asmp:Z-finite-lb}
Fix $\pi\in\mathcal A(x_0)$. Assume that $Z_\pi(b)$ defined in \eqref{eq:Zb} is finite for $\mu$-a.e.\ $b$ and
integrable under $\mu$.
Moreover, define the essential infimum $m_\pi:=\operatorname*{ess\,inf}_{b\sim\mu} Z_\pi(b),$
and assume $m_\pi>-\infty$.
\end{assumption}

\begin{assumption}
\label{asmp:Z-moment}
There exist constants $w>0$ and $M_\pi\in(0,\infty)$ such that
$\int_{\RR^m} |Z_\pi(b)|^{3(1+w)}\,\mu(db)\le M_\pi.$
\end{assumption}

\begin{assumption}
\label{asmp:no-atom}
Let 
$p_\pi:=\mu\big(\{b\in\RR^m:\ Z_\pi(b)=m_\pi\}\big).$
Assume $p_\pi<e^{-\delta}$.
\end{assumption}

Now we give some intuitions.
Assumption~\ref{asmp:Z-finite-lb} is a basic well-posedness condition for classical control problems (e.g., when the prior is degenerate).
Assumption~\ref{asmp:Z-moment} is a mild light-tail requirement under the baseline prior; it is tailored to the
moment conditions needed to apply our Monte Carlo estimator later.
Assumption~\ref{asmp:no-atom} rules out the boundary optimizer being $0$. In many continuous-prior settings, this condition
holds automatically.

Under Assumptions~\ref{asmp:Z-finite-lb}--\ref{asmp:no-atom}, the optimizer $\lambda^*(\pi)$ lies in a compact interval
$[\underline\lambda_\pi,\overline\lambda_\pi]\subset(0,\infty)$ (see Appendix~\ref{app:lambda-bounds}).
Moreover, for fixed $\pi$, the KL dual objective in \eqref{eq:kl-dual-lam0} is strictly concave in $\lambda>0$
whenever $Z_\pi(b)$ is not $\mu$-a.s.\ constant, hence $\lambda^*(\pi)$ is unique.

For policy evaluation, we assume access to a simulator $\mathcal S$ that (i) draws $b\sim\mu$, and (ii) given
$b$ and $\pi$, generates i.i.d.\ trajectory-level samples whose expectation equals $Z_\pi(b)$.
We denote these i.i.d.\ samples by $\{\widehat Z^{\,b}_j\}_{j\ge 1}$, so that
$E[\widehat Z^{\,b}_j]=Z_\pi(b)$. This can be achieved using standard methods \citep{YG23, RheeGlynn2015}.

Define for $\lambda>0$ the inner transform
\(
\Phi_\lambda(z):=\exp(-z/\lambda).
\)
The quantity inside the logarithm in $\mathcal R_\delta(\pi)$ is
\begin{equation}
\label{eq:Mpi}
M_\pi(\lambda):=\int_{\RR^m}\Phi_\lambda\!\left(Z_\pi(b)\right)\mu(db)
=E_{b\sim\mu}\!\left[\exp\!\left(-\frac{Z_\pi(b)}{\lambda}\right)\right].
\end{equation}
We estimate $M_\pi(\lambda)$ using a randomized multilevel Monte Carlo (rMLMC) construction
\citep{Giles2008,RheeGlynn2015, Yanan}.
For each outer draw $b\sim\mu$, sample an independent level
$N^b=\widetilde N+n_0$ where $n_0\in\mathbb{N}_{\ge 0}$ and $\widetilde N\sim\mathrm{Geo}(R)$ with
$R\in(1/2,3/4)$, and generate $2^{N^b+1}$ i.i.d.\ copies $\{\widehat Z^{\,b}_j\}_{1\le j\le 2^{N^b+1}}$.
Let $S^b_\ell:=\sum_{j=1}^{\ell}\widehat Z^{\,b}_j$ and define the odd/even partial sums
$S^{O,b}_\ell:=\sum_{j=1}^{\ell}\widehat Z^{\,b}_{2j-1}$ and $S^{E,b}_\ell:=\sum_{j=1}^{\ell}\widehat Z^{\,b}_{2j}$.
For $N^b\ge n_0$, set
\[
\Delta_{N^b}^b
:=
\Phi_\lambda\!\left(\frac{S_{2^{N^b+1}}^b}{2^{N^b+1}}\right)
-\frac{1}{2}\left(
\Phi_\lambda\!\left(\frac{S_{2^{N^b}}^{O,b}}{2^{N^b}}\right)
+
\Phi_\lambda\!\left(\frac{S_{2^{N^b}}^{E,b}}{2^{N^b}}\right)
\right),
\]
and define the single-sample rMLMC estimator
\begin{equation}
\label{eq:rmlmc-single}
\mathcal E^{b}_\lambda
:=
\Phi_\lambda\!\left(\frac{S^b_{2^{n_0}}}{2^{n_0}}\right)
+\frac{\Delta_{N^b}^b}{p(N^b)},
\end{equation}
where $p(\cdot)$ denotes the probability mass function of $N^b$.
Given $n$ i.i.d.\ outer samples $b_1,\dots,b_n\sim\mu$, we estimate \eqref{eq:Mpi} by
\begin{equation}
\label{eq:rmlmc-outer}
\widehat M_{n,\pi}(\lambda):=\frac{1}{n}\sum_{i=1}^n \mathcal E^{b_i}_\lambda.
\end{equation}
We then use the plug-in estimator of the KL dual objective,
\begin{equation}
\label{eq:kl-plugin}
\widehat f_{n,\pi}(\lambda):=-\lambda\delta-\lambda\log\!\big(\widehat M_{n,\pi}(\lambda)\big),
\end{equation}
and define
\begin{equation}
\widehat{\mathcal R}_\delta(\pi):=\sup_{\lambda>0}\widehat f_{n,\pi}(\lambda).
\end{equation}
While \eqref{eq:kl-plugin} is generally biased at finite $n$ due to the logarithm, it is consistent and admits a
$\sqrt n$--central limit theorem (CLT) under our assumptions; details are deferred to the appendix. We also remark that the parameters $R$ and $n_0$ are fixed small numbers once chosen, thus there is no scaling issues with our rMLMC method.

\begin{assumption}
\label{asmp:finitevar}
The prior $\mu$ is compactly supported with a continuous density, and the map
$b\mapsto Z_\pi(b)$ is continuously differentiable and injective on the support of $\mu$.
\end{assumption}

\begin{remark}
Assumption~\ref{asmp:finitevar} is a convenient sufficient condition to ensure that the rMLMC estimator
$\mathcal E^{B}_\lambda$ has finite variance.
This is satisfied by common examples with continuous functional parameters including the Merton and linear quadratic control problems. In many cases, the prior can be selected in some belief range, thus compact support is reasonable.
\end{remark}

\begin{algorithm}[t]
\caption{rMLMC policy evaluation for the KL case}
\label{alg:kl-eval}
\textbf{Input:} policy $\pi$, radius $\delta$, prior $\mu$, rMLMC parameters $(R,n_0)$,
sample size $n$, initial $\lambda_0>0$, step sizes $\{\alpha_k\}_{k\ge 0}$.\\
\textbf{Output:} $\widehat{\mathcal R}_\delta(\pi)$ and $\widehat\lambda(\pi)$.
\begin{itemize}
\item Initialize $\lambda\leftarrow\lambda_0$.
\item \textbf{repeat} for $k=0,1,2,\dots$
\begin{itemize}
\item Draw i.i.d.\ $b_1,\dots,b_n\sim\mu$ and compute $\widehat M_{n,\pi}(\lambda)$ via
\eqref{eq:rmlmc-single}--\eqref{eq:rmlmc-outer}.
\item Form the plug-in objective $\widehat f_{n,\pi}(\lambda)$ in \eqref{eq:kl-plugin}.
\item Update $\lambda \leftarrow \lambda+\alpha_k\,\widehat g_k(\lambda)$,
where $\widehat g_k(\lambda)$ is a (stochastic) ascent direction for $\widehat f_{n,\pi}(\lambda)$.
\end{itemize}
\item \textbf{until} $\lambda$ converges.
\item Set $\widehat\lambda(\pi)\leftarrow\lambda$ and return
$\widehat{\mathcal R}_\delta(\pi)=\widehat f_{n,\pi}(\widehat\lambda(\pi))$.
\end{itemize}
\end{algorithm}

The following CLT provides a theoretical guarantee for our rMLMC-based policy evaluation:
it shows that $\widehat{\mathcal R}_\delta(\pi)$ achieves the canonical $\mathcal O_p(n^{-1/2})$ accuracy.

\begin{theorem}
\label{thm:kl-clt-main}
Fix $\pi\in\mathcal A(x_0)$ and suppose Assumptions~\ref{asmp:Z-finite-lb}--\ref{asmp:finitevar} hold.
Assume further that there exists a constant
$\underline z\in\RR$ such that $\widehat Z^{\,b}_j\ge \underline z$ almost surely for all $b$ and $j$.
Then, for the estimator \eqref{eq:kl-plugin},
\[
\sqrt n\left(\widehat{\mathcal R}_\delta(\pi)-\mathcal R_\delta(\pi)\right)
\Rightarrow
\mathcal N\!\left(0,\ \sigma^2_{\pi}\right),
\]
where $\sigma^2_{\pi}
=
\frac{(\lambda^*(\pi))^2\,\text{Var}\!\left(\mathcal E^{B}_{\lambda^*(\pi)}\right)}
{\left(M_{\pi}(\lambda^*(\pi))\right)^2} < \infty,$ $M_\pi(\lambda)$ is defined in \eqref{eq:Mpi}, $\mathcal E^{B}_\lambda$ is defined in
\eqref{eq:rmlmc-single}, $\Rightarrow$ denotes convergence in distribution, and $\mathcal{N}(0, \sigma^2)$ denotes a Gaussian distribution with mean 0 and variance $\sigma^2$.
\end{theorem}

The additional lower-bound condition is natural in our setting.
In many diffusion control applications the simulator returns a realized performance measure such as a terminal
utility, a (negative) cost, or a reward shifted by a baseline, which is typically bounded below on the relevant
state/action domain. Further details and generalizations of this section are deferred to the appendix.

\section{Policy Learning}
\label{sec:policy-learning}

Our policy learning approach follows a simple principle: whenever classical or Bayesian control structure is
available (e.g., feedback form, dependence on a low-dimensional belief state), we use it to design a compact
parametric policy class and then optimize its parameters by backpropagation through a simulator. We consider two
instantiations: a linear--quadratic (LQ) example and the Bayesian Merton portfolio problem.

\subsection{Linear--Quadratic Example (DRBC--LQ)}
\label{sec:lq}
In the classical continuous-time LQ regulator, the drift matrix is known and the optimal control is a linear
state feedback derived from a Riccati equation. Concretely, for known $A\in\RR^{d\times d}$ and matrices
$G\in\RR^{d\times k}$, $\Sigma\in\RR^{d\times d}$, the controlled diffusion is
\begin{equation}
\label{eq:lq-classical}
dX_t=(A X_t+G u_t)\,dt+\Sigma\,dW_t,\qquad X_0=x_0,
\end{equation}
and one typically minimizes the quadratic cost
$E\!\big[\int_0^T (X_t^\top QX_t+u_t^\top Ru_t)\,dt + X_T^\top Q_T X_T\big]$ with $Q,Q_T\succeq 0$ and $R\succ 0$.

In our DRBC--LQ instance, the hard-to-estimate component is the drift. We can view it as a random matrix, but for
simplicity, we randomize it by introducing a latent parameter $\theta\in\RR^m$ with baseline prior $\theta\sim\mu$,
independent of $W$, and define the low-dimensional drift parameterization
\begin{equation}
\label{eq:lq-A-theta}
A(\theta):=A_0+\sum_{i=1}^{m}\theta_i A_i,
\end{equation}
for fixed matrices $A_0,\dots,A_m\in\RR^{d\times d}$. The dynamics become
\begin{equation}
\label{eq:lq-dyn}
dX_t=\big(A(\theta)X_t+G u_t\big)\,dt+\Sigma\,dW_t,\qquad X_0=x_0.
\end{equation}
We assume full-state observation with filtration $\mathcal F_t:=\sigma(X_s:0\le s\le t)$ (augmented), and
admissible controls are $\mathcal F$-progressively measurable $u$ with $\int_0^T\|u_t\|_2^2dt<\infty$ a.s.
We use the reward convention
\begin{align}
\label{eq:lq-J}
&J(u;\theta)\notag\\
&=E\!\left[
-\int_0^T \big( X_t^\top Q X_t + u_t^\top R u_t \big)\,dt
- X_T^\top Q_T X_T
\;\middle|\;\theta
\right],
\end{align}
so that for fixed $u$ the conditional value is $Z_u(\theta):=J(u;\theta)$.
The DRBC--LQ objective robustifies \emph{only} the prior of $\theta$ via the KL ball
$\mathcal U_{\mathrm{KL}}(\mu,\delta)$ and solves
\begin{equation}
\label{eq:lq-drbc}
\sup_{u}\ \inf_{\nu\in\mathcal U_{\mathrm{KL}}(\mu,\delta)}\ E_{\theta\sim\nu}\big[Z_u(\theta)\big].
\end{equation}

By Theorem~\ref{thm:kl-dual}, for any fixed $u$ the inner robust value admits the KL dual form, so
\begin{equation}
\label{eq:lq-dual-obj}
\sup_{u}\ \sup_{\lambda>0}\ 
\Big\{
-\lambda\,\delta
-\lambda\log\!\Big(E_{\theta\sim\mu}\big[\exp(-Z_u(\theta)/\lambda)\big]\Big)
\Big\}.
\end{equation}
We optimize \eqref{eq:lq-dual-obj} by alternating updates in $\lambda$ (policy evaluation) and in $u$
(policy learning), as in \cref{sec:duality}.

\subsubsection{Policy class and policy learning}
\label{sec:lq-learn}
While the classical LQ problem admits a Riccati solution when $A$ is known, the Bayesian--LQ policy is not available
in closed form. We therefore learn a parametric feedback policy, but we incorporate low-dimensional
\emph{belief features} inspired by Bayesian control. Specifically, we use a linear feedback form
\begin{equation}
\label{eq:lq-policy-class}
u_t=\pi_\psi\!\big(t,X_t;\,\widehat\theta,S\big)
:= -K_\psi\!\big(t,\widehat\theta,S\big)\,X_t,
\end{equation}
where $(\widehat\theta,S)$ are belief features (e.g., a least-squares estimate and a covariance/information
proxy) computed from rollouts of the current policy.
During the inner-loop optimization of $\psi$, we freeze $(\widehat\theta,S)$ and treat them as constants (we do
not backpropagate through their computation). In our experiments, $K_\psi$ is implemented as a neural
network that takes $(\widehat\theta,\mathrm{vec}(S))$ as input.

We summarize the policy learning step (with $\lambda$ and $(\widehat\theta,S)$ treated as fixed inputs) in
Algorithm~\ref{alg:drbc-lq-learn}.

\begin{algorithm}[t]
\caption{DRBC--LQ policy learning step}
\label{alg:drbc-lq-learn}
\begin{algorithmic}[1]
\STATE \textbf{Input:} fixed $\lambda>0$; prior $\mu$; frozen belief features $(\widehat\theta,S)$; inner steps
$S_{\mathrm{in}}$; step size $\eta$; initialization $\psi$.
\FOR{$s=1,\dots,S_{\mathrm{in}}$}
  \STATE Sample $\theta_1,\dots,\theta_N\sim \mu$.
  \STATE For each $\theta_i$, simulate \eqref{eq:lq-dyn} with control
  $u_t=\pi_{\psi}(t,X_t;\widehat\theta,S)$ to estimate $Z_{\pi_\psi}(\theta_i)$.
  \STATE Compute a stochastic gradient $\widehat{\nabla}_\psi \mathcal L_{\mathrm{LQ}}(\pi_\psi,\lambda)$
  by backpropagating through the simulator \emph{while treating $(\widehat\theta,S)$ as constants}.
  \STATE Update $\psi \leftarrow \psi + \eta\, \widehat{\nabla}_\psi \mathcal L_{\mathrm{LQ}}(\pi_\psi,\lambda)$.
\ENDFOR
\STATE \textbf{Output:} updated policy parameters $\psi$.
\end{algorithmic}
\end{algorithm}

\subsection{Bayesian Merton Problem (DRBC--Merton)}
\label{sec:merton}
The purpose of this subsection is to illustrate how problem-specific structure can guide the design of a
practical parametric policy class for DRBC policy learning. We therefore focus on parameterization and learning
mechanics, and we will not rely on minimax swap arguments in this section.

Let $W=(W_1,\ldots,W_d)^\top$ be an $\RR^d$-valued Brownian motion on a complete filtered probability space
$(\Omega,\mathcal F,P)$.
Let $B:\Omega\to\RR^d$ be a latent random vector, independent of $W$ under $P$, with baseline prior
$B\sim\mu\in\mathcal P(\RR^d)$.
The risk-free asset satisfies $S_{0,0}=s_0>0$ and
\[
dS_{0,t}=rS_{0,t}\,dt,\qquad 0\le t\le T,
\]
with interest rate $r>0$.
The $d$ risky assets (stocks) $S=(S_1,\ldots,S_d)^\top$ satisfy for $i=1,\ldots,d,\ \ 0\le t\le T,$
\begin{equation}
\label{eq:merton-stock}
dS_{i,t}=S_{i,t}\Big(B_i\,dt+\sum_{j=1}^d\sigma_{ij}\,dW_{j,t}\Big),
\end{equation}
where $\sigma=(\sigma_{ij})_{1\le i,j\le d}$ is invertible and $S_{i,0}>0$.

The controller observes only stock prices, i.e., controls are adapted to the augmentation of 
$\mathcal F^S_t:=\sigma(S_s:0\le s\le t)$.
A control $\pi=\{\pi_t\}_{t\in[0,T]}$ is an $\mathcal F^S$-progressively measurable process with
$\int_0^T\|\pi_t\|_2^2dt<\infty$ $P$-a.s., where $\pi_{i,t}$ denotes the dollar amount invested in asset $i$ at
time $t$.
The wealth process $X$ with $X_0=x_0$ satisfies
\begin{equation}
\label{eq:merton-wealth}
dX_t=rX_t\,dt+\pi_t^\top\big(B-r\mathbf 1\big)\,dt+\pi_t^\top\sigma\,dW_t,
\qquad 0\le t\le T.
\end{equation}

Given a strictly increasing, strictly concave utility $u:(0,\infty)\to\RR$, the Bayesian control problem is
\begin{equation}
\label{eq:merton-bayes}
V^\mu(x_0)=\sup_{\pi\in\mathcal A(x_0)} E^P\!\big[u(X_T)\big].
\end{equation}
In our experiments we take the power utility $u(x)=\frac{1}{\alpha}x^\alpha$ with $\alpha\in(0,1)$.
We robustify only the prior on $B$ by allowing an adversary to choose $\nu$ in a KL ball
$\mathcal U_{\mathrm{KL}}(\mu,\delta)$, yielding the DRBC value with $Z_\pi(b):=E^{P^b}\!\big[u(X_T)\big]$
\[
V^{\mathrm{DRBC}}(x_0)=\sup_{\pi\in\mathcal A(x_0)}\ \inf_{\nu\in\mathcal U_{\mathrm{KL}}(\mu,\delta)}
E_{B\sim\nu}\!\big[Z_\pi(B)\big]
.
\]More details of the Bayesian Merton problem can be found in Appendix \ref{sec:finance}.
For the general $\phi$-divergence dual, policy learning for fixed
$(\lambda,\beta)$ amounts to solving
\begin{equation}
\label{eq:merton-inner}
\sup_{\pi\in\mathcal A(x_0)}\ \int_{\RR^d}\Phi_{\lambda,\beta}\!\left(E^{P^b}[u(X_T)]\right)\mu(db),
\end{equation}
where $\Phi_{\lambda,\beta}(x)=-(\lambda\phi)^*(\beta-x)$ is increasing and concave under mild conditions (see Theorem \ref{thm:phi-dual}).
This is a continuous-time smooth ambiguity objective in the spirit of \citet{Klibanoff2005}.
Closed forms are available only in special cases, so we use \eqref{eq:merton-inner} mainly as a guide for
constructing a good parameterization.
\subsubsection{Designing a parametrization via problem structure}
\label{sec:merton-param}
To keep notation simple, we present the resulting loss in the one-dimensional case ($d=1$), but the same
construction extends to the multi-asset case verbatim.

We utilize tools from stochastic calculus and results from \citet{GLS22} to design a deep learning method to solve Problem (\ref{eq:merton-inner}). This motivates an alternate definition of the set of admissible controls $\tilde{\mathcal{A}}(x_0)$ that contains the optimal policy for a large class of Bayesian control problems (see Appendix \ref{alterdef}): the collection of all alternate admissible controls $\pi$, which is a subset of $\mathcal{A}(x_0)$ such that 
there exists a function $v \sim u$ such that there exists corresponding function $h: \mathbb{R} \to \mathbb{R}$ and a functional $\rho: L^1 \to \mathbb{R}$ such that for the controlled terminal wealth $X_T^{\pi}$,
$
\int_{\mathbb{R}}E^{P^b}\left[v(X_T^{\pi})\right]\lambda(b)d\mu(b) = h(x_0)\rho(\lambda)    
$
and $
    h(1) = v(e^{rT}).
$
For a fixed $b \in \mathbb{R}$, define $\eta^{b*}_t = \exp\left(-\frac{B-b}{\sigma}W_t - \frac{(B-b)^2}{2\sigma^2}t\right)$, $\frac{dQ^b}{dP} \Big|_{\mathcal{F}^S_t} = \eta^{b*}_t$, and $W_t^{b} = W_t + \frac{B-b}{\sigma}t$ for $t \in [0,T]$, then under $Q^b$, the process $W^b$ is an $\mathcal{F}^S$-Brownian motion that is independent from $B$, hence under $Q^b$, the stock price $S$ evolves as
$dS_t = S_t\left(bdt + \sigma dW^b_t\right)$
. Moreover,
$d X_t = (X_t - \pi_t) r \, d t + \pi_t \left(b d t + \sigma d W^b_t\right).$
For a fixed $b \in \mathbb{R}$,
$E^{P^b}\left[u(X_T)\right] = E^{Q^b}\left[u(X_T)\right],$
therefore it is equivalent to study the problem
$\sup_{\pi \in \tilde{\mathcal{A}}(x_0)} \int_{\mathbb{R}}\Phi_{\lambda,\beta}\left(E^{Q^b}\left[u(X_T)\right]\right)d\mu(b),$
which is solved in \citet{GLS22}.
Let $b_0 = E^P[B]$ and if $b = b_0$, then we denote $W^{b_0}$ as $\hat{W}$. We also define for $t \in [0,T]$,
$\eta_t = \exp\left(-\nu\hat{W}_t - \frac{1}{2}\nu^2t\right),$
where $\nu = \frac{b_0-r}{\sigma}.$ For a fixed $b \in \mathbb{R}$, we define $t \in [0,T]$,
$\eta^{b}_t = \exp\left(-\nu_b\hat{W}_t - \frac{1}{2}\nu_b^2t\right),$
where $\nu_b = \frac{b_0-b}{\sigma}.$
The families of measures $\{P^b\}_{b \in \mathbb{R}}$ and $\{Q^b\}_{b \in \mathbb{R}}$ are quite different. 
For example, the distribution of $B$ under $Q^b$ is still $\mu$ (see Appendix \ref{transformderive}), where under $P^b$, $B$ is a constant $b$.
In the rest of this section, we assume that $\mu \sim \mathcal{N}\left(\mu_0, \sigma_0^2\right)$. Now, we are able to derive sufficient conditions that an optimal terminal wealth satisfies. Theorem \ref{system} motivates a loss function that also ensures numerical stability.
\begin{theorem}\label{system}
Suppose $u$ and $\Phi_{\lambda, \beta}$ are both strictly increasing and strictly concave, $X_T$ is a terminal wealth such that there exists a constant $\kappa \in \mathbb{R}$ with $E^{Q^{r}} \left[ X_T  \right] = x_0 e^{r T}$ and $L_{\kappa}(X_T) = 0$
with
\begin{align*}
L_{\kappa}(X_T) &=
\kappa - \log\left(u'(X_T)\right)  - \frac{K_2^2}{4K_1} +\log\left(\frac{\sigma_0}{\sigma_Q}\right)+ \frac{r-b_0}{\sigma}\left(W_T + \frac{B-b_0}{\sigma}T\right) - \frac{(b_0-r)^2}{2\sigma^2}T\notag \\
& +K_3+\log\left(\int_{\mathbb{R}}\Phi_{\lambda, \beta}' \left( E^{Q^b} \left[ u(X_T) \right] \right) d\mu(b)\right)-\log\left(\int_{\mathbb{R}}\Phi_{\lambda, \beta}'\left(E^{Q^a}[u(X_T)]\right)d\mu_A(a)\right)\notag 
\end{align*}
$D = W_T + \frac{B-b_0}{\sigma}T$, $K_1 = \dfrac{T}{2 \sigma^2} + \dfrac{1}{2 \sigma_0^2}$, $K_2 = \dfrac{T b_0}{\sigma^2} + \dfrac{D}{\sigma} + \dfrac{\mu_0}{\sigma_0^2}$, $K_3 = \dfrac{T b_0^2}{2 \sigma^2} + \dfrac{b_0 D}{\sigma} + \dfrac{\mu_0^2}{2 \sigma_0^2}$, $\mu_A \sim \mathcal{N}\left( \mu_Q, \sigma_Q^2 \right )$, $\sigma_Q^2 = \left( \dfrac{T}{\sigma^2} + \dfrac{1}{\sigma_0^2} \right )^{-1}$, and $\mu_Q = \sigma_Q^2 \left( \dfrac{W_T}{\sigma} + \dfrac{B T}{\sigma^2} + \dfrac{\mu_0}{\sigma_0^2} \right )$,
    then $X_T$ is the optimal terminal wealth for Problem (\ref{eq:merton-inner}).
\end{theorem}

From Theorem \ref{system}, we guess that the optimal terminal wealth has the form $X^*_T = h(W_T,B)$, where $h$ is a function that we plan to use neural network $h_{\theta}$ to approximate ($\theta \in \mathbb{R}^d$). $\kappa \in \mathbb{R}$ is a learnable scalar. We replace $X_T$ in Theorem \ref{system} by $h_{\theta}(W_T,B)$, denote all the learnable parameters as $\boldsymbol{\theta} = (\theta, \kappa)^T \in \mathbb{R}^{d+1}$, and design the loss function as for a choice of $b_1 \in \mathbb{R}$,
\begin{align}\label{loss}
\mathcal{L}(\boldsymbol{\theta}) = E^{Q^{b_1}} \bigg[\bigg\|L_{\kappa}(h_{\theta}(W_T,B)))\bigg\|_2^2\bigg]+ \Bigg( E^{Q^{r}}[h_{\theta}(W_T,B)] - x_0 e^{rT} \Bigg)^2.   
\end{align}
If $X^*_T = h_{\theta}(W_T,B)$, then $\mathcal{L}(\boldsymbol{\theta}) = 0$. The Algorithm \ref{BPG} can be used to minimize $\mathcal{L}(\boldsymbol{\theta})$ to find the optimal numerical solution $\boldsymbol{\theta}^*$. Discussion of DRBC algorithm in general case is in Appendix \ref{algo4}. Usage of $h_{\theta}$ enables scaling with dimension $n$, for simplicity we don't show it here.

\section{Numerical Experiments}
\label{sec:experiments}

We evaluate DRBC along two axes. In Section~\ref{sec:exp-lq}, we run a synthetic misspecified LQ experiment and
solve DRBC via alternating optimization, comparing against plug-in baselines. In
Section~\ref{sec:exp-merton}, we study the Bayesian Merton problem: Section~\ref{ablation} varies the radius
$\delta$ and compares DRBC to DRC, Section~\ref{smps} benchmarks the neural network method against a closed-form
solution and confirms the $\mathcal{O}_p(n^{-1/2})$ policy-evaluation scaling, and Section~\ref{realdata} reports
out-of-sample Sharpe ratios on real stock data. Implementation details for all experiments are in Appendix~\ref{ED}, with additional
simulated and high-dimensional results in Appendix~\ref{add}.

\subsection{Linear Quadratic Control Example}
\label{sec:exp-lq}
We consider a continuous-time LQ problem with an unknown drift parameter $\theta^\star$.
We generate controlled trajectories under $\theta^\star$ and use the resulting state--control data to form a
plug-in estimate $\widehat\theta$.
To test robustness to prior misspecification, the DRBC is computed using a \emph{misspecified} prior $\mu$ (distinct from the ground-truth distribution of $\theta^\star$).
We also compute an \emph{oracle} benchmark that knows $\theta^\star$ and solves the corresponding LQ problem exactly.
We compare plug-in and DRBC, and quantify their performance by the gap to the oracle evaluated under the ground-truth parameter $\theta^\star$, as reported in Table~\ref{tab:lq_main}.
For completeness, the corresponding achieved utility values are reported in Table~\ref{tab:lq_utility_appendix} in appendix.

\begin{table}[t]
\centering
\caption{LQ control performance gap (mean $\pm$ std.\ over 100 runs).}
\label{tab:lq_main}
\small
\setlength{\tabcolsep}{4pt}
\begin{tabular}{lccc}
\toprule
Method & $\delta=0.01$ & $\delta=0.05$ & $\delta=0.10$ \\
\midrule
Plug-in gap to oracle
& \multicolumn{3}{c}{$3.6 \pm 13.0$} \\
DRBC gap to oracle
& \textbf{0.8 $\pm$ 1.3} & \textbf{1.2 $\pm$ 1.3} & \textbf{1.3 $\pm$ 1.3} \\
\bottomrule
\end{tabular}
\vspace{-0.05in}
\end{table}

\subsection{Bayesian Merton Example}\label{sec:exp-merton}
\subsubsection{Ablation Study with the Radius}\label{ablation}
We consider a controlled synthetic environment in which the ground-truth drift is time-varying and known, so that
the true optimal policy can be computed. Concretely, the state evolves as $dX_t = X_t\Big(r\,dt+\pi_t(B_t-r)\,dt+\sigma\,dW_t\Big)$ with $B_t=\frac{B_0}{2}\bigl(1+\cos(\kappa t)\bigr)$.
We solve for the optimal policy $\pi^\star$ under this true model and then compute the resulting evaluated policies under DRC and
DRBC for different uncertainty radii $\delta$. The utility gap $\mathrm{Gap}
=
\mathrm{Utility}(\text{true model},\pi^\star)
-
\mathrm{Utility}(\text{true model},\widehat{\pi}_{\delta})$ is the performance measure,
where smaller values indicate a policy closer to the true optimum. As shown by Figure \ref{fig:gap_vs_delta}, DRBC outperforms DRC across the range of radii we consider.

\begin{figure}[t]
  \centering
  \includegraphics[width=\linewidth]{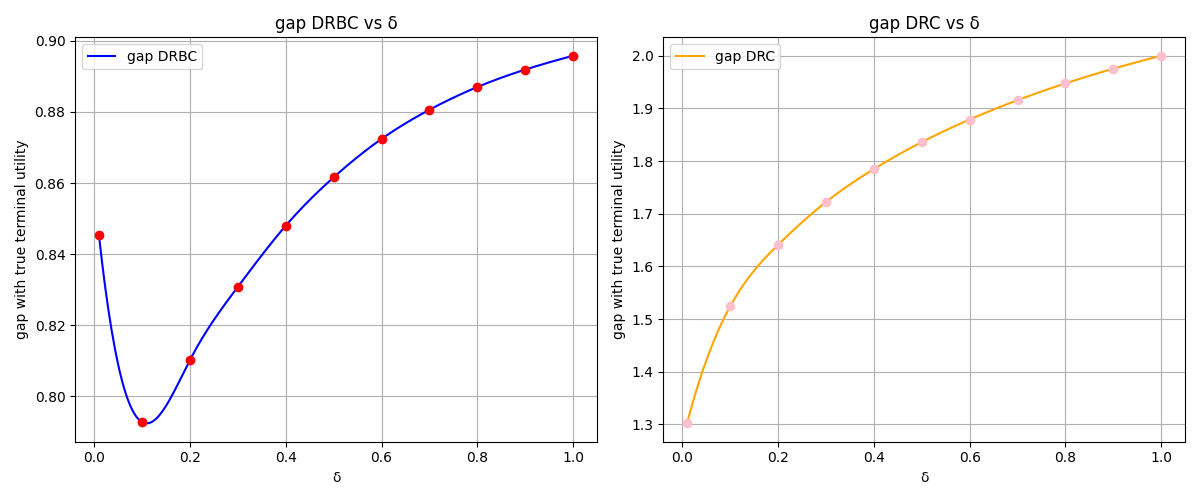}
  \caption{Utility gap versus uncertainty radius $\delta$.}
  \label{fig:gap_vs_delta}
\end{figure}

\subsubsection{Compare with Closed-form Solutions and Rate of Convergence}\label{smps}
If we replace $\Phi_{\lambda, \beta}$ with a power function $\Phi$, Problem (\ref{eq:merton-inner}) admits a closed-form solution \citep{GLS22}. This allows us to explicitly evaluate the performance of Algorithm \ref{BPG} in this specific scenario. In Table \ref{result6.1}, we compare the closed-form optimal value $E^{Q^{b_1}}\left[u(X_T^*)\right]$ with the learned optimal value $E^{Q^{b_1}}\left[u\left(h_{\theta^*}(W_T,B)\right)\right]$ across various market parameters $r$ and different values of $b_1$. To estimate the learned optimal value \( E^{Q^{b_1}}\left[ u\left( h_{\theta^*}(W_T, B) \right) \right] \), we employ Monte Carlo approach by conducting 100 independent experiments, each utilizing 2000 samples of the pair \( (W_T, B) \).

\begin{table}[t]
\centering
\caption{Comparisons of learning results and closed-form solutions. Here $b_1=0.1$, $b_2=0.3$; $r_1=0.05$, $r_2=0.1$. We run evaluations 100 times.}
\label{result6.1}
\small
\begin{tabular}{lcc}
\toprule
Comparing term & Learning result & Closed form\\
\midrule
$E^{Q^{b_1}}\!\left[u(X_T^*)\right],\ r_1$ & $3.174\pm0.013$ & $3.226$\\
$E^{Q^{b_2}}\!\left[u(X_T^*)\right],\ r_1$ & $3.460\pm0.030$ & $3.380$\\
$E^{Q^{b_1}}\!\left[u(X_T^*)\right],\ r_2$ & $3.179\pm0.014$ & $3.267$\\
\bottomrule
\end{tabular}
\end{table}

For the CLT rate experiment, for each fixed $\delta = 0.01,0.05, 0.10$, we sample and compare three different numbers of independent and identically distributed copies of $B$: $n = 10^2$, $10^3$, and $10^4$. Table \ref{result6.3} presents the means and standard deviations of the estimator  with a fixed policy $\pi$, computed from 100 independent experiments. The numerical results demonstrate that the estimator $\hat{Q}_{\text{DRBCKL}}(\pi)$ converges, and both the scaling rates of the standard deviation and the difference $\left|\hat{Q}_{\text{DRBCKL}}(\pi) - Q_{\text{DRBCKL}}(\pi)\right|$ are consistent with the $\mathcal{O}_p\left(n^{-1/2}\right)$ rate predicted by theory.

\begin{table}[t]
\centering
\caption{Policy evaluation vs.\ sample size $n$ of $B$.}
\label{result6.3}
\begin{tabular}{lccc}
\toprule
$\delta$ & $10^{2}$ & $10^{3}$ & $10^{4}$ \\
\midrule
0.01 & $3.764\!\pm\!0.040$ & $3.766\!\pm\!0.012$ & $3.765\!\pm\!0.004$ \\
0.05 & $3.695\!\pm\!0.041$ & $3.697\!\pm\!0.013$ & $3.696\!\pm\!0.004$ \\
0.10 & $3.643\!\pm\!0.042$ & $3.645\!\pm\!0.014$ & $3.645\!\pm\!0.004$ \\
\bottomrule
\end{tabular}
\end{table}

\subsubsection{Real Data Experiments} \label{realdata}

This experiment is motivated by \citet{BlanchetChenZhou2021}. We use $S\&P$ 500 constituents data from 2015
to 2024 and evaluate different methods using average annualized Sharpe Ratio for all stocks. We use a rolling window of one year to get the required parameters for all methods like interest rate $r$ and an estimation of $\sigma$. For the ease of computation, we choose two fixed priors across time based on \citet{wang2020continuous}, prior 1 is more deviated and prior 2 is less deviated. The uncertainty set radius $\delta$ is chosen following the cross-validation type in \citet{NianSi}. Results in Table \ref{tab:mean_sharpe_ratios} show that the DRBC is substantially better than benchmarks and it reduces the overpessimism in real data. 

\begin{table}[t]
  \caption{Comparison of mean Sharpe Ratios across methods for part of $S\&P$ 500 data.}
  \label{tab:mean_sharpe_ratios}
  \vskip 0.1in
  \begin{center}
    \begin{small}
      \begin{sc}
        \begin{tabular}{c c}
          \toprule
          Method & Mean of Sharpe Ratio ($\uparrow$)\\
          \midrule
          Merton                                  & $0.015 \pm 0.301$ \\
          Bayesian (prior 1)    & $0.493 \pm 0.281$ \\
          Bayesian (prior 2)    & $0.655 \pm 0.282$ \\
          DRC (prior 1)                           & $-0.220 \pm 0.292$ \\
          DRC (prior 2)                           & $-0.237 \pm 0.293$ \\
          DRBC (prior 1)                          & $\mathbf{0.818 \pm 0.308}$ \\
          DRBC (prior 2)                          & $\mathbf{1.147 \pm 0.311}$ \\
          \bottomrule
        \end{tabular}
      \end{sc}
    \end{small}
  \end{center}
  \vskip -0.15in
\end{table}

\section{Conclusion and Future Work}

We proposed the formulation and an efficient numerical algorithm for distributionally robust Bayesian control (DRBC) for diffusion control under prior misspecification,
where robustness is imposed only on the prior of the latent parameter, which is also hard to estimate. 
Two future directions are especially promising. First, extending DRBC to optimal-transport
uncertainty sets around the prior calls for new scalable numerical methods. Second, deeper structural results for
Bayesian diffusion control could guide improved parameterizations
and accelerate DRBC policy learning in partially observed and high-dimensional settings.

\section*{Acknowledgments}
J. Blanchet gratefully acknowledges support from DoD through ONR N00014-24-1-2655, and support from NSF via grants 2312204 and 2403007. Y. Liu acknowledges financial support from the National Natural Science Foundation of China (Grant No. 12401624), Guangdong Science and Technology Program (Grant No. 2024QN11X076), Shenzhen Science and Technology Program (Grant Nos. RCBS20231211090814028, JCYJ20250604141203005, 2025TC0010) and The Chinese University of Hong Kong (Shenzhen) University Development Fund (Grant No. UDF01003336) and is partly supported by the Guangdong Provincial Key Laboratory of Mathematical Foundations for Artificial Intelligence (Grant No. 2023B1212010001).

\bibliographystyle{plainnat}
\bibliography{example_paper}

\clearpage
\appendix
\section{Preliminary Definitions}\label{prelim}
\begin{definition}
    Given a convex function $\phi:[0,\infty)\to\mathbb{R}$ with $\phi(1)=0$,
$\phi$-divergence of $Q$ from $P$ is
\[
D_{\phi}(P\|Q)=\int_{\Omega}\phi\!\left(\frac{dP}{dQ}\right)\,dQ,
\]
where $\frac{dP}{dQ}$ is the Radon--Nikodym derivative of $P$ with respect to $Q$.
If $\phi(x)=x\log(x)-x+1$, then $D_{\phi}$ is the Kullback--Leibler (KL) divergence.
\end{definition}

\begin{definition}\label{C-R-def}
The \textit{Cressie-Read divergence} is a $\phi$- divergence where the convex function is taken by for a fixed $k > 1$,
$$f_k(t) = \frac{t^k - kt + k-1}{k(k-1)}.$$
\end{definition}

\begin{definition}\label{lineargrowth}
    Suppose $g: \mathbb{R}^d \to \mathbb{R}$ is a function, then we say $g$ satisfies the \textit{linear growth condition} if there exists a constant $c > 0$ such that $$\left|g(x)\right| \leq c\left(1 + \left\Vert x\right\Vert\right),$$ where $\left\Vert .\right\Vert$ denotes a norm in the Euclidean space.
\end{definition}

\begin{definition}
    Suppose $f: D \subset \mathbb{R}^d \to \mathbb{R}$ is a function, then we say $f$ is \textit{H\"{o}lder continuous} with parameter $w$ and bounding constant $K$ if there exists $K > 0$ and $w > 0$ such that for any $x, y \in D$, 
    \begin{equation}\label{Holdercont}
       \left|f(x) - f(y)\right| \leq K\left\Vert x-y\right\Vert^w. 
    \end{equation}
    We say $f$ is \textit{locally H\"{o}lder continuous} with parameter $w$ and bounding constant $K$ if Equation (\ref{Holdercont}) holds inside each compact neighborhood.
\end{definition}

\section{Proof of Theorem \ref{uncertaintyset}}\label{sec:2.1}
In this section, we give a proof of Theorem \ref{uncertaintyset} when the uncertainty ball is characterized by a general $\phi$-divergence. Then the statement in the main text will be obtained as a corollary. This also justifies that our formulation can be generalized to any $\phi$-divergence uncertainty set. We use $\mathcal{U}_{\delta}$ to denote the general uncertainty set.
\begin{proof}
    \begin{itemize}
    \item (1) It is easy to see by checking the definitions, Equation (\ref{relation}) defines an equivalence relation, then the set of all equivalence classes under the quotient topology defines $\mathcal{U}_{\delta}$.
        \item (2) Fix $Q \in \mathcal{U}_{\delta}$, then there are 4 cases to check.
        \begin{itemize}
            \item (a) $\frac{dQ}{dP} = f(B)$, where $f: \mathbb{R}^d \to \mathbb{R}$ is bounded and (Borel) measurable. Then from definition of the uncertainty set, for any $A \in \mathcal{B}\left(\mathbb{R} ^d\right)$, there exists $\nu \ll \mu$ such that
            \begin{align*}
                Q(B \in A) &= E^{Q}\left[\mathbf{1}_{\{B \in A\}}\right] = E^P\left[\frac{dQ}{dP}\mathbf{1}_{\{B \in A\}}\right] = E^P\left[f(B)\mathbf{1}_{\{B \in A\}}\right] \\ &= \nu(A) = \int_A \frac{d\nu}{d\mu}d\mu =  E^P\left[\frac{d\nu}{d\mu}(B)\mathbf{1}_{\{B \in A\}}\right],
            \end{align*}
            which implies that $\frac{dQ}{dP} = \frac{d\nu}{d\mu}(B)$ $P$-almost surely from the uniqueness of Radon-Nikodym derivative.
            \item (b) $\frac{dQ}{dP} = f(W)$, where $f: \Omega \to \mathbb{R}$ is bounded and measurable. In this case, it is convenient to assume $\Omega = C\left([0,T];\mathbb{R}^d\right)$. Thus the standard Brownion motion $W: \Omega \to \Omega$ can be viewed as a function-valued random element. Therefore, for $A \in \mathcal{B}(\Omega)$, 
            \begin{align*}
               Q(W \in A) &= E^{Q}\left[\mathbf{1}_{\{W \in A\}}\right] = E^P\left[\frac{dQ}{dP}\mathbf{1}_{\{W \in A\}}\right] = E^P\left[f(W)\mathbf{1}_{\{W \in A\}}\right] \neq P(W \in A)
            \end{align*}unless $f = 1$, which is equivalent to the case when $\frac{dQ}{dP} = \frac{d\nu}{d\mu}(B)$ with $\nu = \mu$.
            \item (c) $\frac{dQ}{dP} = f(B, W)$, where $f: \mathbb{R}^d \times \Omega \to \mathbb{R}$ is bounded and jointly measurable. From part (b) and independence between $B$ and $W$, it is easy to construct a contradiction. 
            \item (d) $\frac{dQ}{dP} = f(Y)$, where $f: \Omega \to \mathbb{R}$ is bounded and  measurable, and $Y$ is a stochastic process that is independent of both $B$ and $W$. From the definition of the equivalence classes, this case is equivalent to the case when $\frac{dQ}{dP} = \frac{d\nu}{d\mu}(B)$ with $\nu = \mu$.
        \end{itemize}
        Moreover, \begin{align*}
           D_{\phi}(Q \parallel P) = \int \phi\left(\frac{dQ}{dP}\right) \, dP &= \int \phi\left(\frac{d\nu}{d\mu}(B)\right) \, dP= \int \phi\left(\frac{d\nu}{d\mu}\right) \, d\mu = D_{\phi}(\nu \parallel \mu).
        \end{align*}
        % \item (3) 
    \end{itemize}
\end{proof}

\section{Proof of Theorem \ref{thm:kl-dual}}

In the main text we focus on the KL uncertainty set. In this appendix we first derive the dual
representation for a general $\phi$--divergence ball around the baseline prior $\mu$, and then obtain the KL and
Cressie--Read cases as corollaries.

For $\delta\ge 0$, define the ambiguity set
\[
\mathcal U_{\phi}(\mu,\delta):=\{\nu\in\mathcal P(\RR^m): D_\phi(\nu\|\mu)\le \delta\}.
\]
As shown in Section \ref{sec:2.1}, this uncertainty set can be viewed as generated by a quotient topology. Therefore, in the following proof, we will use the equivalence relation formulation to conduct the proof. 
Fix a policy $\pi\in\mathcal A(x_0)$ and recall that $Z_\pi(b)$ is defined in \eqref{eq:Zb}.
The robust prior evaluation for fixed $\pi$ is
\[
\mathcal R_{\phi,\delta}(\pi):=\inf_{\nu\in\mathcal U_\phi(\mu,\delta)}\int_{\RR^m} Z_\pi(b)\,\nu(db).
\]

\begin{theorem}
\label{thm:phi-dual}
Fix $\pi\in\mathcal A(x_0)$ and assume $Z_\pi(B)\in L^1(\mu)$ and $\mathcal U_\phi(\mu,\delta)\neq\emptyset$.
Let $\phi^*$ be the convex conjugate of $\phi$ and define
\[
\Phi_{\lambda,\beta}(z):= -(\lambda \phi)^*(\beta-z),\qquad \lambda\ge 0,\ \beta\in\RR,
\]
with the convention $(\lambda\phi)^*(\cdot)=\lambda\,\phi^*(\cdot/\lambda)$ for $\lambda>0$.
Then
\begin{equation}
\label{eq:phi-dual}
\mathcal R_{\phi,\delta}(\pi)
=
\sup_{\lambda \ge 0,\ \beta\in\RR}
\left\{
\beta-\lambda\,\delta+\int_{\RR^m}\Phi_{\lambda,\beta}\!\left(Z_\pi(b)\right)\mu(db)
\right\}.
\end{equation}
\end{theorem}

\begin{proof}
The proof uses the law invariance theory developed in \citet{Shapiro17} for $\phi$--divergences in our
$\sigma(B)$--measurable uncertainty set.

Recall that we begin with the complete probability space $(\Omega, \mathcal{F}, P)$.
Let $\hat{\mathcal{F}} = \sigma(B) \subset \mathcal{F}$, and define $\hat{P} = P|_{\hat{\mathcal{F}}}$.
Then the triple $(\Omega,\hat{\mathcal{F}},\hat{P})$ is a probability space (may not be complete) such that for
any $\hat{\mathcal{F}}$--measurable random variable $Z$,
\begin{equation}\label{equal}
    E^P\!\left[Z\right]= E^{\hat{P}}\!\left[Z\right].
\end{equation}

For a fixed $Q$ in the $\phi$--divergence uncertainty set on $(\Omega,\mathcal F)$ that only reweights $B$
(i.e., $dQ/dP$ is $\hat{\mathcal F}$--measurable; equivalently $dQ/dP=f(B)$), define the restriction
$\hat{Q} = Q|_{\hat{\mathcal{F}}}$.
We define the space $L^1\!\left(\Omega, \hat{\mathcal{F}}, \hat{P}\right)\subset L^1\!\left(\Omega, \mathcal{F}, P\right)$.
Following \citet{Shapiro17}, we define an equivalence relation $\sim_{\phi}$ between two functions with mean 1
$X,Y \in L^1\!\left(\Omega, \hat{\mathcal{F}}, \hat{P}\right)$ by
\[
X \sim_{\phi} Y
\quad \Longleftrightarrow \quad
\int_{\Omega}\phi\!\left(X\right)d\hat{P} = \int_{\Omega}\phi\!\left(Y\right)d\hat{P}.
\]
We then define a quotient space of $L^1\!\left(\Omega, \hat{\mathcal{F}}, \hat{P}\right)$ with respect to
$\sim_{\phi}$ by
\[
\hat{\mathcal{A}}
=
\left\{\left[\hat{X}\right]:\ \hat{X}=\frac{d\hat{Q}}{d\hat{P}},\ \hat{Q}=Q|_{\hat{\mathcal F}}
\text{ for some admissible }Q\right\}.
\]

Let $\hat{\mathcal{U}}_{\phi,\delta}$ be the collection of restrictions of all admissible $Q$ on
$\hat{\mathcal{F}}$.
For any admissible $Q$, there is a unique $\hat{Q}\in\hat{\mathcal{U}}_{\phi,\delta}$ such that
$\hat{Q}=Q|_{\hat{\mathcal{F}}}$ from the definition.
Conversely, for a fixed $\hat{Q}\in\hat{\mathcal{U}}_{\phi,\delta}$, there is also a unique admissible $Q$ such
that $\hat{Q}=Q|_{\hat{\mathcal{F}}}$; this follows from the fact that two admissible measures that coincide on
$\sigma(B)$ are identified (equivalently, they induce the same $B$--marginal).

Now fix $\pi$ and consider the $\hat{\mathcal F}$--measurable random variable $Z:=Z_\pi(B)$.
Because our uncertainty set only reweights $B$, for any admissible $Q$ we have
\[
E^Q[Z_\pi(B)] = E^{\hat Q}[Z_\pi(B)],
\]
and therefore
\[
\mathcal R_{\phi,\delta}(\pi)
=
\inf_{\hat{Q} \in \hat{\mathcal{U}}_{\phi,\delta}}E^{\hat{Q}}\!\left[Z_\pi(B)\right].
\]

Next, define the one-dimensional linear subspace
\[
D=\{\alpha Z_\pi(B):\ \alpha\in\RR\}\subset L^1(\Omega,\hat{\mathcal F},\hat P).
\]
From the law invariance theory in \citet{Shapiro17}, the functional
$\rho: D\to\RR$ given by
\[
\rho(Z)=\inf_{\hat{Q} \in \hat{\mathcal{U}}_{\phi,\delta}}E^{\hat{Q}}[Z]
\]
is law invariant with respect to $(\Omega,\hat{\mathcal{F}},\hat{P})$, and strong duality holds.
As in Section 3.2 of \citet{Shapiro17}, the Lagrangian for
$\inf_{\hat{Q} \in \hat{\mathcal{U}}_{\phi,\delta}} E^{\hat{Q}}[Z]$ can be written in terms of
$\hat{X}=d\hat Q/d\hat P$ as
\begin{align*}
\mathcal{L}_Z\!\left(\hat{X},\lambda,\beta\right)
&= \int_{\Omega}Z\hat{X}\,d\hat{P}
+ \lambda\left(\int_{\Omega}\phi\!\left(\hat{X}\right)d\hat{P}-\delta\right)
+ \beta\left(1-\int_{\Omega}\hat{X}\,d\hat{P}\right)\\
&=\beta - \lambda \delta + \int_{\Omega}\left(Z\hat{X} + \lambda \phi\!\left(\hat{X}\right)-\beta\hat{X}\right)d\hat{P}.
\end{align*}
The Lagrangian dual problem is
\[
\sup_{\lambda \geq 0,\ \beta \in \mathbb{R}}\inf_{\hat{X} \geq 0}\mathcal{L}_Z\!\left(\hat{X},\lambda,\beta\right).
\]
Since $L^1(\Omega,\hat{\mathcal F},\hat P)$ is decomposable, again as in \citet{Shapiro17},
\begin{align*}
\inf_{\hat{X} \geq 0}\mathcal{L}_Z\!\left(\hat{X},\lambda,\beta\right)
&=\beta - \lambda \delta
+ \int_{\Omega}\inf_{x \geq 0}\left(Zx + \lambda \phi(x) - \beta x\right)d\hat{P}\\
&=\beta - \lambda \delta
+ \int_{\Omega}-\left(\lambda \phi\right)^*(\beta - Z)\,d\hat{P}\\
&=\beta - \lambda \delta
+ E^{\hat{P}}\!\left[\Phi_{\lambda,\beta}\!\left(Z_\pi(B)\right)\right]\\
&=\beta - \lambda \delta
+ \int_{\RR^m}\Phi_{\lambda,\beta}\!\left(Z_\pi(b)\right)\mu(db),
\end{align*}
where the last equality uses \eqref{equal} and the fact that $Z_\pi(B)$ is $\sigma(B)$--measurable with law
induced by $\mu$.
Taking the supremum over $(\lambda,\beta)$ yields \eqref{eq:phi-dual}.
\end{proof}

\begin{remark}
A common (but incorrect) shortcut would be to treat the uncertainty set as if it were law invariant with respect
to the full filtration $\mathcal F$ and to apply $\phi$--divergence duality directly to a generic
$\mathcal F$--measurable payoff. This would lead to an expression of the form
\[
\inf_{Q} E^Q[H]
\stackrel{\text{(wrong)}}{=}
\sup_{\lambda\ge 0,\ \beta\in\RR}\left\{\beta-\lambda\delta+E^P\!\left[-(\lambda\phi)^*(\beta-H)\right]\right\},
\]
which is not valid here because admissible Radon--Nikodym derivatives are constrained to be
$\sigma(B)$--measurable (i.e., of the form $f(B)$), hence the uncertainty set is \emph{not} law invariant with
respect to $\mathcal F$.
In the language of \citet{Shapiro17}, suppose $X=f(B)$ with $B\sim\mathcal N(0,1)$, and define $Y=f(W_1)$.
Then $X$ and $Y$ are distributionally equivalent under $P$, but $Y$ cannot induce an admissible change of
measure (since it is not $\sigma(B)$--measurable), even though $X$ can.
\end{remark}

%%%%%%%%%%%%%%%%%%%%%%%%%%%%%%%%%%%%%%%%%%%%%%%%%%%%%%%%%%%%%%%%%%%%%%%%%%%%
% Corollary: KL (main text)
%%%%%%%%%%%%%%%%%%%%%%%%%%%%%%%%%%%%%%%%%%%%%%%%%%%%%%%%%%%%%%%%%%%%%%%%%%%%

\subsection{KL Duality as A Corollary of Theorem~\ref{thm:phi-dual}}
\label{app:kl-cor}

\begin{proof}[Proof of Theorem~\ref{thm:kl-dual}]
Apply Theorem~\ref{thm:phi-dual} with $\phi(x)=x\log x-x+1$.
Then $\phi^*(y)=e^y-1$, hence for $\lambda>0$,
\[
\Phi_{\lambda,\beta}(z)=-(\lambda\phi)^*(\beta-z)
= -\lambda\phi^*\!\left(\frac{\beta-z}{\lambda}\right)
= \lambda\left(1-\exp\!\left(\frac{\beta-z}{\lambda}\right)\right).
\]
Therefore,
\begin{align*}
\mathcal R_{\delta}(\pi)
&=\sup_{\lambda \geq 0,\ \beta \in \mathbb{R}}
\left\{\beta - \lambda \delta
+ \int_{\RR^m}\lambda\left(1-\exp\!\left(\frac{\beta-Z_\pi(b)}{\lambda}\right)\right)\mu(db)\right\}.
\end{align*}
If $\lambda>0$, differentiating w.r.t.\ $\beta$ yields the optimizer
\[
\beta^*
=
-\lambda\log\!\left(\int_{\RR^m}\exp\!\left(\frac{-Z_\pi(b)}{\lambda}\right)\mu(db)\right).
\]
Substituting $\beta^*$ gives
\[
\mathcal R_{\delta}(\pi)
=
\sup_{\lambda>0}
\left\{
-\lambda \delta
-\lambda\log\!\left(\int_{\RR^m}\exp\!\left(\frac{-Z_\pi(b)}{\lambda}\right)\mu(db)\right)
\right\},
\]
and the extension at $\lambda=0$ follows from the same argument as the discussion of case 1 after Assumption~1
in \citet{HuHong2013}.
\end{proof}

%%%%%%%%%%%%%%%%%%%%%%%%%%%%%%%%%%%%%%%%%%%%%%%%%%%%%%%%%%%%%%%%%%%%%%%%%%%%
% Corollary: Cressie--Read
%%%%%%%%%%%%%%%%%%%%%%%%%%%%%%%%%%%%%%%%%%%%%%%%%%%%%%%%%%%%%%%%%%%%%%%%%%%%

\subsection{Cressie--Read Divergence as A Corollary of Theorem~\ref{thm:phi-dual}}
\label{app:cr}

We give the dual form for the Cressie--Read divergence.
Let $\phi_k(x) = \frac{x^k-kx+k-1}{k(k-1)}$ with $k\in(1,\infty)$, denote the uncertainty set by
$\mathcal{U}_{k,\delta}:=\mathcal U_{\phi_k}(\mu,\delta)$, and define $k_*=\frac{k}{k-1}$ and
$c_k(\delta) = \left(1+k(k-1)\delta\right)^{\frac{1}{k}}$.

\begin{theorem}\label{C-Rdivergence}
\begin{align*}
\mathcal R_{k,\delta}(\pi)
:=
\inf_{\nu \in \mathcal{U}_{k,\delta}} \int_{\RR^m} Z_\pi(b)\,\nu(db)
=
\sup_{\beta \in \mathbb{R}}\left\{
\beta - c_k(\delta)\left(\int_{\RR^m}\left(\beta - Z_\pi(b)\right)_+^{{k_*}}\mu(db)\right)^{\frac{1}{k_*}}
\right\}.
\end{align*}
\end{theorem}

\begin{proof}
From \citet{JohnDuchi}, we know that
\[
\phi_k^* (x) = \frac{1}{k}\left((k-1)x + 1\right)_+^{k_*} - \frac{1}{k}.
\]
Plugging this into Theorem~\ref{thm:phi-dual} gives
\begin{align*}
\mathcal R_{k,\delta}(\pi)
&=\sup_{\lambda \geq 0,\ \beta \in \mathbb{R}}
\left\{\beta - \lambda \delta + \int_{\RR^m}\Phi_{\lambda,\beta}\!\left(Z_\pi(b)\right)\mu(db)\right\}\\
&=\sup_{\lambda \geq 0,\ \beta \in \mathbb{R}}
\left\{\beta - \lambda \left(\delta - \frac{1}{k}\right)
- \lambda^{1-k_*}\frac{(k-1)^{k_*}}{k}
\int_{\RR^m}\left(\beta - Z_\pi(b)+ \frac{\lambda}{k-1}\right)_+^{{k_*}}\mu(db)\right\}\\
&=\sup_{\lambda \geq 0,\ \tilde{\beta} \in \mathbb{R}}
\left\{\tilde{\beta} - \lambda \left(\delta + \frac{1}{k(k-1)}\right)
- \lambda^{1-k_*}\frac{(k-1)^{k_*}}{k}
\int_{\RR^m}\left(\tilde{\beta} - Z_\pi(b)\right)_+^{{k_*}}\mu(db)\right\},
\end{align*}
where $\tilde{\beta} = \beta + \frac{\lambda}{k-1}$.
Noting that $\frac{k_* - 1}{k_*} = \frac{1}{k}$ and taking derivatives w.r.t.\ $\lambda$ yields
\[
\lambda
= (k-1)\left(\delta k (k-1)+1\right)^{-\frac{1}{k_*}}
\left(\int_{\RR^m}\left(\tilde{\beta} - Z_\pi(b)\right)_+^{{k_*}}\mu(db)\right)^{\frac{1}{k_*}}.
\]
Substituting back, we obtain
\begin{align*}
\mathcal R_{k,\delta}(\pi)
&= \sup_{\tilde{\beta} \in \mathbb{R}}\left\{
\tilde{\beta}
- \left(\delta k (k-1)+1\right)^{\frac{1}{k}}
\left(\int_{\RR^m}\left(\tilde{\beta} - Z_\pi(b)\right)_+^{{k_*}}\mu(db)\right)^{\frac{1}{k_*}}
\right\}\\
&=\sup_{\tilde{\beta} \in \mathbb{R}}\left\{
\tilde{\beta}
- c_k(\delta)\left(\int_{\RR^m}\left(\tilde{\beta} - Z_\pi(b)\right)_+^{{k_*}}\mu(db)\right)^{\frac{1}{k_*}}
\right\},
\end{align*}
which finishes the proof.
\end{proof}

\subsection{Compact Support of Dual Variables}\label{app:lambda-bounds}
\begin{theorem}\label{thm:lambda-compact}
Assume Assumption~\ref{asmp:Z-finite-lb} holds and $\delta>0$.
Fix $\pi\in\mathcal A(x_0)$ and let $\lambda^*(\pi)$ be any maximizer of the KL dual objective in
\eqref{eq:kl-dual-lam0}. Then any maximizer satisfies
\[
\lambda^*(\pi)\in[0,\overline\lambda_\pi],
\qquad
\overline\lambda_\pi:=\frac{\int_{\RR^m} Z_\pi(b)\,\mu(db)-m_\pi}{\delta}\ \in(0,\infty),
\]
where $m_\pi=\operatorname*{ess\,inf}_{b\sim\mu} Z_\pi(b)$.
Moreover, if Assumption~\ref{asmp:no-atom} holds, then $\lambda^*(\pi)>0$ and hence
$\lambda^*(\pi)\in(0,\overline\lambda_\pi]$.
\end{theorem}

\begin{proof}
Define for $\lambda>0$
\[
f_\pi(\lambda):=-\lambda\delta-\lambda\log\!\left(\int_{\RR^m}\exp\!\left(-\frac{Z_\pi(b)}{\lambda}\right)\mu(db)\right),
\]
and extend $f_\pi(0):=m_\pi$ as in Theorem~\ref{thm:kl-dual}.
By Jensen's inequality,
\begin{align*}
\log\!\left(\int_{\RR^m}\exp\!\left(-\frac{Z_\pi(b)}{\lambda}\right)\mu(db)\right)
&\ge
\int_{\RR^m}\log\!\left(\exp\!\left(-\frac{Z_\pi(b)}{\lambda}\right)\right)\mu(db)\\
&=
-\frac{1}{\lambda}\int_{\RR^m} Z_\pi(b)\,\mu(db).
\end{align*}
Therefore, for all $\lambda>0$,
\[
f_\pi(\lambda)\le \int_{\RR^m} Z_\pi(b)\,\mu(db)-\lambda\delta.
\]
On the other hand, since $f_\pi(0)=m_\pi$ and $\lambda^*(\pi)$ maximizes $f_\pi$ over $\lambda\ge 0$, we have
$f_\pi(\lambda^*(\pi))\ge m_\pi$. Combining with the upper bound yields
\[
\int_{\RR^m} Z_\pi(b)\,\mu(db)-\lambda^*(\pi)\delta \ \ge\ m_\pi,
\]
hence
\[
\lambda^*(\pi)\le \overline\lambda_\pi:=\frac{\int Z_\pi\,d\mu - m_\pi}{\delta}.
\]
This proves $\lambda^*(\pi)\in[0,\overline\lambda_\pi]$.

It remains to show that $\lambda^*(\pi)>0$ under Assumption~\ref{asmp:no-atom}.
Write $Z_\pi(b)=m_\pi+\Delta_\pi(b)$ where $\Delta_\pi(b)\ge 0$ $\mu$-a.s.
Then for $\lambda>0$,
\[
f_\pi(\lambda)
=
m_\pi-\lambda\delta-\lambda\log\!\left(\int_{\RR^m}\exp\!\left(-\frac{\Delta_\pi(b)}{\lambda}\right)\mu(db)\right).
\]
As $\lambda\downarrow 0$,
\(
\exp(-\Delta_\pi(b)/\lambda)\to \mathbf 1\{\Delta_\pi(b)=0\}
\)
and by dominated convergence,
\[
\int \exp\!\left(-\frac{\Delta_\pi(b)}{\lambda}\right)\mu(db)\ \longrightarrow\ 
\mu(\Delta_\pi=0)=p_\pi.
\]
Hence the right-derivative at $0$ satisfies
\[
\lim_{\lambda\downarrow 0}\frac{f_\pi(\lambda)-f_\pi(0)}{\lambda}
=
-\delta-\log p_\pi.
\]
Under Assumption~\ref{asmp:no-atom}, $p_\pi<e^{-\delta}$, so $-\delta-\log p_\pi>0$, meaning $f_\pi(\lambda)>f_\pi(0)$
for all sufficiently small $\lambda>0$. Therefore $\lambda=0$ cannot be optimal and $\lambda^*(\pi)>0$.
\end{proof}

\section{Details of Section \ref{sec:duality}}
In this section, we first provide the proof of Theorem \ref{thm:kl-clt-main}. Then we provide the details of Algorithm \ref{alg:kl-eval} and discuss the extension to general $\phi$-divergence.

\subsection{Proof of Theorem~\ref{thm:kl-clt-main}}
\label{app:kl-clt}

\begin{proof}
Fix $\pi\in\mathcal A(x_0)$ throughout and suppress $\pi$ from the notation when there is no ambiguity.
Recall that for $\lambda>0$,
\[
M(\lambda):=M_\pi(\lambda)=\int_{\RR^m}\exp\!\left(-\frac{Z_\pi(b)}{\lambda}\right)\mu(db),
\qquad
\widehat M_n(\lambda):=\widehat M_{n,\pi}(\lambda)=\frac{1}{n}\sum_{i=1}^n \mathcal E^{b_i}_\lambda,
\]
where $\mathcal E^{b}_\lambda$ is defined in \eqref{eq:rmlmc-single}.
We also write
\[
f(\lambda):=-\lambda\delta-\lambda\log M(\lambda),
\qquad
\widehat f_n(\lambda):=-\lambda\delta-\lambda\log\widehat M_n(\lambda),
\]
so that $\mathcal R_\delta(\pi)=\sup_{\lambda>0}f(\lambda)$ and
$\widehat{\mathcal R}_\delta(\pi)=\sup_{\lambda>0}\widehat f_n(\lambda)$.

By Appendix~\ref{app:lambda-bounds}, under Assumptions~\ref{asmp:Z-finite-lb}--\ref{asmp:no-atom} there exists a
compact interval $K:=[\underline\lambda_\pi,\overline\lambda_\pi]\subset(0,\infty)$ such that the (unique) maximizer
$\lambda^*=\lambda^*(\pi)$ of $f$ lies in $K$. Define
\[
\mathcal R_{\delta,K}(\pi):=\sup_{\lambda\in K}f(\lambda),
\qquad
\widehat{\mathcal R}_{\delta,K}(\pi):=\sup_{\lambda\in K}\widehat f_n(\lambda).
\]
Since $\lambda^*\in K$, $\mathcal R_{\delta,K}(\pi)=\mathcal R_\delta(\pi)$.
Moreover, the compact containment argument implies that maximizing over $\lambda>0$ is asymptotically equivalent
to maximizing over $K$ (the difference is $o_\PP(n^{-1/2})$), hence it suffices to prove the CLT for
$\widehat{\mathcal R}_{\delta,K}(\pi)$.

Fix $\lambda\in K$.
We show that
\begin{equation}
\label{eq:finite-var-kl}
\text{Var}\!\left(\mathcal E^B_\lambda\right)<\infty.
\end{equation}

By Theorem~\ref{thm:kl-clt-main}, there exists a constant $\underline z\in\RR$ such that
$\widehat Z^{\,b}_j\ge \underline z$ almost surely for all $b$ and $j$. In particular, almost surely,
\[
\frac{S^b_\ell}{\ell}\ge \underline z,\qquad
\frac{S^{O,b}_\ell}{\ell}\ge \underline z,\qquad
\frac{S^{E,b}_\ell}{\ell}\ge \underline z
\qquad\text{for all }b\text{ and }\ell\ge 1.
\]

Write $m_\pi:=\operatorname*{ess\,inf}_{b\sim\mu}Z_\pi(b)>-\infty$ (Assumption~\ref{asmp:Z-finite-lb}).
Since $\lambda\ge \underline\lambda_\pi>0$ and $\widehat Z^{\,b}_j\ge \underline z$ almost surely, the map
$\Phi_\lambda(x):=e^{-x/\lambda}$ satisfies the uniform derivative bounds on $K\times[\underline z,\infty)$:
\begin{equation}
\label{eq:Phi-der-bnd}
\sup_{\lambda\in K}\sup_{x\ge \underline z}|\Phi'_\lambda(x)|
\le \frac{1}{\underline\lambda_\pi}\exp\!\left(-\frac{\underline z}{\underline\lambda_\pi}\right),
\qquad
\sup_{\lambda\in K}\sup_{x\ge \underline z}|\Phi''_\lambda(x)|
\le \frac{1}{\underline\lambda_\pi^{\,2}}\exp\!\left(-\frac{\underline z}{\underline\lambda_\pi}\right).
\end{equation}
In particular, for each $\lambda\in K$,
\[
0<\Phi_\lambda(x)\le \exp\!\left(-\frac{\underline z}{\underline\lambda_\pi}\right)
\quad\text{for all }x\ge \underline z,
\]
so the base term $\Phi_\lambda(S^B_{2^{n_0}}/2^{n_0})$ in \eqref{eq:rmlmc-single} is bounded and hence has finite
second moment.

It remains to control the multilevel correction term $\Delta^B_{N^B}/p(N^B)$.
For deterministic levels $n\ge n_0$, define
\[
\bar Z^{\,B}_{n+1}:=\frac{S^B_{2^{n+1}}}{2^{n+1}},
\qquad
\bar Z^{\,O,B}_{n}:=\frac{S^{O,B}_{2^{n}}}{2^{n}},
\qquad
\bar Z^{\,E,B}_{n}:=\frac{S^{E,B}_{2^{n}}}{2^{n}},
\]
so that
\[
\Delta^B_n=\Phi_\lambda(\bar Z^{\,B}_{n+1})
-\frac{1}{2}\Big(\Phi_\lambda(\bar Z^{\,O,B}_{n})+\Phi_\lambda(\bar Z^{\,E,B}_{n})\Big).
\]
For each $n\ge n_0$, apply Taylor's theorem to $\Phi_\lambda(\cdot)$ around $Z_\pi(B)$:
there exist random points $\xi^{B}_{n+1}$ between $Z_\pi(B)$ and $\bar Z^{\,B}_{n+1}$,
$\xi^{O,B}_{n}$ between $Z_\pi(B)$ and $\bar Z^{\,O,B}_{n}$, and
$\xi^{E,B}_{n}$ between $Z_\pi(B)$ and $\bar Z^{\,E,B}_{n}$ such that
\begin{align*}
\Phi_\lambda(\bar Z^{\,B}_{n+1})
&=\Phi_\lambda(Z_\pi(B))
+\Phi'_\lambda(Z_\pi(B))\big(\bar Z^{\,B}_{n+1}-Z_\pi(B)\big)
+\frac{1}{2}\Phi''_\lambda(\xi^{B}_{n+1})\big(\bar Z^{\,B}_{n+1}-Z_\pi(B)\big)^2,\\
\Phi_\lambda(\bar Z^{\,O,B}_{n})
&=\Phi_\lambda(Z_\pi(B))
+\Phi'_\lambda(Z_\pi(B))\big(\bar Z^{\,O,B}_{n}-Z_\pi(B)\big)
+\frac{1}{2}\Phi''_\lambda(\xi^{O,B}_{n})\big(\bar Z^{\,O,B}_{n}-Z_\pi(B)\big)^2,\\
\Phi_\lambda(\bar Z^{\,E,B}_{n})
&=\Phi_\lambda(Z_\pi(B))
+\Phi'_\lambda(Z_\pi(B))\big(\bar Z^{\,E,B}_{n}-Z_\pi(B)\big)
+\frac{1}{2}\Phi''_\lambda(\xi^{E,B}_{n})\big(\bar Z^{\,E,B}_{n}-Z_\pi(B)\big)^2.
\end{align*}
Using the identity $\bar Z^{\,B}_{n+1}=\frac12(\bar Z^{\,O,B}_{n}+\bar Z^{\,E,B}_{n})$, the zeroth- and
first-order terms cancel exactly, giving
\begin{align}
\label{eq:Delta-remainder}
\Delta^B_n
&=
\frac{1}{2}\Phi''_\lambda(\xi^{B}_{n+1})\big(\bar Z^{\,B}_{n+1}-Z_\pi(B)\big)^2
-\frac{1}{4}\Phi''_\lambda(\xi^{O,B}_{n})\big(\bar Z^{\,O,B}_{n}-Z_\pi(B)\big)^2
-\frac{1}{4}\Phi''_\lambda(\xi^{E,B}_{n})\big(\bar Z^{\,E,B}_{n}-Z_\pi(B)\big)^2.
\end{align}
By \eqref{eq:Phi-der-bnd} and the fact that $\xi^{B}_{n+1},\xi^{O,B}_{n},\xi^{E,B}_{n}\ge \underline z$ almost surely,
there is a constant $C_K<\infty$ (depending only on $K$ and $\underline z$) such that
$|\Phi''_\lambda(\xi)|\le C_K$ for all $\lambda\in K$ and all $\xi\ge \underline z$.
Therefore, from \eqref{eq:Delta-remainder},
\begin{equation}
\label{eq:Delta-square-bound}
(\Delta^B_n)^2
\ \le\
C_K\Big(
\big|\bar Z^{\,B}_{n+1}-Z_\pi(B)\big|^4
+\big|\bar Z^{\,O,B}_{n}-Z_\pi(B)\big|^4
+\big|\bar Z^{\,E,B}_{n}-Z_\pi(B)\big|^4
\Big).
\end{equation}

Next we bound the fourth moments of the sample-mean errors. Under Assumption~\ref{asmp:Z-moment} (and assuming
$3(1+w)\ge 4$), we have $E[|\widehat Z^{\,B}_1|^4]<\infty$ and $E[|Z_\pi(B)|^4]<\infty$.
By a Marcinkiewicz--Zygmund inequality (or standard moment bounds for averages of i.i.d.\ variables), there is a
constant $C<\infty$ such that for all $n\ge n_0$,
\begin{equation}
\label{eq:mean-4th}
E\Big[\big|\bar Z^{\,B}_{n+1}-Z_\pi(B)\big|^4\Big]
\le
\frac{C}{2^{2(n+1)}}\,E\Big[\big|\widehat Z^{\,B}_1-Z_\pi(B)\big|^4\Big]
\ \lesssim\ 2^{-2n},
\end{equation}
and the same bound holds with $\bar Z^{\,O,B}_{n}$ and $\bar Z^{\,E,B}_{n}$ in place of $\bar Z^{\,B}_{n+1}$
(up to constants), since the odd and even subsamples are i.i.d.\ given $B$.
Combining \eqref{eq:Delta-square-bound} and \eqref{eq:mean-4th}, we obtain
\begin{equation}
\label{eq:Delta-decay}
E\big[(\Delta^B_n)^2\big]\ \lesssim\ 2^{-2n}.
\end{equation}

Finally, since $N^B=\widetilde N+n_0$ with $\widetilde N\sim\mathrm{Geo}(R)$ and $R\in(1/2,3/4)$, its pmf
$p(n)$ decays geometrically. Using \eqref{eq:Delta-decay},
\begin{align*}
E\!\left[\left(\frac{\Delta^B_{N^B}}{p(N^B)}\right)^2\right]
&=
\sum_{n=n_0}^{\infty}\frac{E[(\Delta^B_n)^2]}{p(n)}
\ \lesssim\
\sum_{n=n_0}^{\infty}\frac{2^{-2n}}{p(n)}
\ <\ \infty,
\end{align*}
since $2^{-2n}=4^{-n}$ and $p(n)$ is geometric with ratio $R>1/2>1/4$.
Therefore $E[(\Delta^B_{N^B}/p(N^B))^2]<\infty$, and together with the boundedness of the base term this proves
\eqref{eq:finite-var-kl}.

Now we show the CLT.
Define the $C(K)$-valued random element
\[
\mathbb G_n(\lambda):=\sqrt n\big(\widehat M_n(\lambda)-M(\lambda)\big),\qquad \lambda\in K.
\]
For each fixed $\lambda\in K$, $\widehat M_n(\lambda)$ is an average of i.i.d.\ terms $\mathcal E^B_\lambda$ with
finite variance, hence the pointwise CLT holds:
\[
\mathbb G_n(\lambda)\Rightarrow \mathcal N\big(0,\text{Var}(\mathcal E^B_\lambda)\big).
\]
On $K$, the map
$\lambda\mapsto \mathcal E^B_\lambda$ is Lipschitz with a square-integrable random Lipschitz constant.
Indeed, under Assumptions~\ref{asmp:Z-finite-lb} and~\ref{asmp:Z-moment}, together with the smoothness of
$\Phi_\lambda$ on $K$ and the rMLMC construction, one can bound $\sup_{\lambda\in K}|\partial_\lambda \mathcal
E^B_\lambda|$ in $L^2$, implying
\[
\sup_{\lambda,\lambda'\in K}\frac{|\mathcal E^B_\lambda-\mathcal E^B_{\lambda'}|}{|\lambda-\lambda'|}\in L^2.
\]
Consequently, $\{\mathbb G_n\}_{n\ge 1}$ is tight in $C(K)$ and the finite-dimensional CLTs imply the functional
CLT
\begin{equation}
\label{eq:func-clt-kl}
\mathbb G_n \Rightarrow \mathbb G \quad\text{in }C(K),
\end{equation}
for some mean-zero Gaussian process $\mathbb G$ on $K$.
In particular, $\sup_{\lambda\in K}|\widehat M_n(\lambda)-M(\lambda)|\to 0$ in probability.
Since $M(\lambda)>0$ for all $\lambda>0$ and $K$ is compact, $\inf_{\lambda\in K}M(\lambda)>0$, hence with
probability tending to $1$ we have $\inf_{\lambda\in K}\widehat M_n(\lambda)>0$ and $\log \widehat
M_n(\lambda)$ is well-defined on $K$.

Define $T:C(K)\to C(K)$ by $(T\varphi)(\lambda)=\log(\varphi(\lambda))$.
On $\{\varphi:\inf_{\lambda\in K}\varphi(\lambda)>0\}$, $T$ is Hadamard differentiable with derivative
\[
(T'_M h)(\lambda)=\frac{h(\lambda)}{M(\lambda)}.
\]
Applying the functional delta theorem to \eqref{eq:func-clt-kl} yields
\[
\sqrt n\big(\log\widehat M_n(\cdot)-\log M(\cdot)\big)
\Rightarrow
\frac{\mathbb G(\cdot)}{M(\cdot)}
\quad\text{in }C(K).
\]
Since the map $S:C(K)\to C(K)$ given by $(S\psi)(\lambda):=-\lambda\delta-\lambda\psi(\lambda)$ is continuous
linear, we obtain
\[
\sqrt n\big(\widehat f_n(\cdot)-f(\cdot)\big)
\Rightarrow
-\lambda\,\frac{\mathbb G(\lambda)}{M(\lambda)}
\quad\text{in }C(K).
\]

Finally, define the supremum functional $V:C(K)\to\RR$ by $V(\varphi)=\sup_{\lambda\in K}\varphi(\lambda)$.
By strict concavity of $f$ in the KL case (hence uniqueness of $\lambda^*$ in $K$),
$V$ is Hadamard directionally differentiable at $f$ with $V'_f(h)=h(\lambda^*)$.
Therefore,
\[
\sqrt n\big(\widehat{\mathcal R}_{\delta,K}(\pi)-\mathcal R_{\delta,K}(\pi)\big)
=\sqrt n\big(V(\widehat f_n)-V(f)\big)
\Rightarrow
-\lambda^*\frac{\mathbb G(\lambda^*)}{M(\lambda^*)}.
\]
Since $\mathbb G(\lambda^*)\sim \mathcal N(0,\text{Var}(\mathcal E^B_{\lambda^*}))$, we conclude
\[
-\lambda^*\frac{\mathbb G(\lambda^*)}{M(\lambda^*)}
\sim
\mathcal N\!\left(0,\ \frac{(\lambda^*)^2\,\text{Var}(\mathcal E^B_{\lambda^*})}{(M(\lambda^*))^2}\right).
\]
Removing the restriction to $K$ completes the proof of Theorem~\ref{thm:kl-clt-main}.
\end{proof}

\subsection{Details of Algorithm~\ref{alg:kl-eval}}
\label{app:kl-eval-details}

In this subsection we provide implementation details for Algorithm~\ref{alg:kl-eval}, in particular how to
construct a stochastic ascent direction $\widehat g_k(\lambda)$ for the plug-in objective
$\widehat f_{n,\pi}(\lambda)=-\lambda\delta-\lambda\log(\widehat M_{n,\pi}(\lambda))$.

Fix $\lambda>0$ and recall that $\widehat M_{n,\pi}(\lambda)=\frac1n\sum_{i=1}^n \mathcal E^{b_i}_\lambda$ with
$\mathcal E^{b}_\lambda$ defined in \eqref{eq:rmlmc-single}. Differentiating the plug-in objective gives
\begin{equation}
\label{eq:ghat-kl}
\frac{\partial}{\partial\lambda}\widehat f_{n,\pi}(\lambda)
=
-\delta-\log\big(\widehat M_{n,\pi}(\lambda)\big)
-\lambda\frac{\widehat M'_{n,\pi}(\lambda)}{\widehat M_{n,\pi}(\lambda)},
\end{equation}
where $\widehat M'_{n,\pi}(\lambda)$ is a Monte Carlo estimate of $\frac{\partial}{\partial\lambda}M_\pi(\lambda)$.
We construct $\widehat M'_{n,\pi}(\lambda)$ using the same rMLMC coupling as for $\widehat M_{n,\pi}(\lambda)$.

Define the auxiliary function
\[
A_\lambda(x):=\frac{x}{\lambda^2}\exp\!\left(-\frac{x}{\lambda}\right),\qquad x\in\RR,\ \lambda>0,
\]
so that $\frac{\partial}{\partial\lambda}\Phi_\lambda(x)=A_\lambda(x)$ for $\Phi_\lambda(x)=\exp(-x/\lambda)$.
For each outer draw $b\sim\mu$ and the same inner samples used to form $\mathcal E^b_\lambda$, define
\[
\Delta^{b,\mathrm{der}}_{N^b}
:=
A_\lambda\!\left(\frac{S_{2^{N^b+1}}^b}{2^{N^b+1}}\right)
-\frac{1}{2}\left(
A_\lambda\!\left(\frac{S_{2^{N^b}}^{O,b}}{2^{N^b}}\right)
+
A_\lambda\!\left(\frac{S_{2^{N^b}}^{E,b}}{2^{N^b}}\right)
\right),
\]
and the corresponding single-sample estimator
\begin{equation}
\label{eq:rmlmc-single-der}
\mathcal E^{b,\mathrm{der}}_\lambda
:=
A_\lambda\!\left(\frac{S^b_{2^{n_0}}}{2^{n_0}}\right)
+\frac{\Delta^{b,\mathrm{der}}_{N^b}}{p(N^b)}.
\end{equation}
Given $n$ i.i.d.\ outer samples $b_1,\dots,b_n\sim\mu$, set
\[
\widehat M'_{n,\pi}(\lambda):=\frac{1}{n}\sum_{i=1}^n \mathcal E^{b_i,\mathrm{der}}_\lambda.
\]
Then a natural stochastic ascent direction for Algorithm~\ref{alg:kl-eval} is
\begin{equation}
\label{eq:stoch-ascent}
\widehat g_k(\lambda)
:=
-\delta-\log\big(\widehat M_{n,\pi}(\lambda)\big)
-\lambda\frac{\widehat M'_{n,\pi}(\lambda)}{\widehat M_{n,\pi}(\lambda)},
\end{equation}
which mimics \eqref{eq:ghat-kl}. Under mild regularity (e.g., the same moment and lower-bound conditions used in
Theorem~\ref{thm:kl-clt-main}), $\widehat g_k(\lambda)$ is stable and can be combined with standard step-size
choices $\{\alpha_k\}$ (constant, diminishing, or Adam-style schedules) to perform stochastic ascent.

\begin{remark}
Once an rMLMC estimator is built for a transformed
expectation, one obtains a stochastic gradient by differentiating the transform and applying the same multilevel
coupling to the derivative transform. In the KL case, the \emph{population} objective is strictly concave in
$\lambda$ (whenever $Z_\pi(B)$ is not $\mu$-a.s.\ constant), so gradient ascent is well behaved.
\end{remark}

\subsection{Extension to General \texorpdfstring{$\phi$}{phi}-Divergence}
\label{app:phi-eval}

In the main text we focus on KL ambiguity. Here we briefly describe how the rMLMC policy evaluation approach
extends to a general $\phi$-divergence ball.

Recall the general dual representation:
for fixed $\pi\in\mathcal A(x_0)$,
\begin{equation}
\label{eq:phi-dual-app}
\mathcal R_{\phi,\delta}(\pi)
=
\sup_{\lambda\ge 0,\ \beta\in\RR}
\left\{
\beta-\lambda\delta+\int_{\RR^m}\Phi_{\lambda,\beta}\!\left(Z_\pi(b)\right)\mu(db)
\right\},
\qquad
\Phi_{\lambda,\beta}(z):=-(\lambda\phi)^*(\beta-z).
\end{equation}
Define the inner target
\[
G_{\pi}(\lambda,\beta):=\int_{\RR^m}\Phi_{\lambda,\beta}\!\left(Z_\pi(b)\right)\mu(db).
\]
Fix $(\lambda,\beta)$ and draw $b\sim\mu$.
As in the KL case, sample $N^b=\widetilde N+n_0$ with $\widetilde N\sim\mathrm{Geo}(R)$ and generate
$2^{N^b+1}$ i.i.d.\ inner samples $\{\widehat Z^{\,b}_j\}_{1\le j\le 2^{N^b+1}}$ with
$E[\widehat Z^{\,b}_j]=Z_\pi(b)$.
Let $S^b_\ell, S^{O,b}_\ell,S^{E,b}_\ell$ be the corresponding full/odd/even partial sums as in the main text.
Define the multilevel correction
\[
\Delta^{b}_{N^b,\lambda,\beta}
:=
\Phi_{\lambda,\beta}\!\left(\frac{S_{2^{N^b+1}}^b}{2^{N^b+1}}\right)
-\frac{1}{2}\left(
\Phi_{\lambda,\beta}\!\left(\frac{S_{2^{N^b}}^{O,b}}{2^{N^b}}\right)
+
\Phi_{\lambda,\beta}\!\left(\frac{S_{2^{N^b}}^{E,b}}{2^{N^b}}\right)
\right),
\]
and the single-sample estimator
\begin{equation}
\label{eq:rmlmc-phi-single}
\mathcal E^{b}_{\lambda,\beta}
:=
\Phi_{\lambda,\beta}\!\left(\frac{S^b_{2^{n_0}}}{2^{n_0}}\right)
+\frac{\Delta^{b}_{N^b,\lambda,\beta}}{p(N^b)}.
\end{equation}
With $n$ i.i.d.\ draws $b_1,\dots,b_n\sim\mu$, define
\begin{equation}
\label{eq:rmlmc-phi-outer}
\widehat G_{n,\pi}(\lambda,\beta):=\frac{1}{n}\sum_{i=1}^n \mathcal E^{b_i}_{\lambda,\beta}.
\end{equation}
Under mild integrability conditions (e.g., growth control of $\Phi_{\lambda,\beta}$ on a compact parameter set
together with Assumption~\ref{asmp:Z-moment}), the estimator \eqref{eq:rmlmc-phi-outer} is unbiased for
$G_{\pi}(\lambda,\beta)$ and has finite variance.

We then form the empirical dual objective
\begin{equation}
\label{eq:phi-emp-obj}
\widehat Q_{\phi,\delta}(\pi)
:=
\sup_{\lambda\ge 0,\ \beta\in\RR}
\left\{
\beta-\lambda\delta+\widehat G_{n,\pi}(\lambda,\beta)
\right\}.
\end{equation}
In contrast to the KL reduction \eqref{eq:kl-dual-lam0}, the joint maximization over $(\lambda,\beta)$ in
\eqref{eq:phi-emp-obj} is not guaranteed to be strictly concave for a general $\phi$ (but it is always concave). Consequently, we recommend using stochastic gradient methods for the maximization in
\eqref{eq:phi-emp-obj} in the general case. One can obtain gradient estimators by
differentiating $\Phi_{\lambda,\beta}$ (when differentiable) and applying the same multilevel coupling as in
\eqref{eq:rmlmc-phi-single}. This yields stochastic ascent directions for both $\lambda$ and $\beta$ and allows
alternating optimization between policy evaluation and policy learning.

For the Cressie--Read family $\phi_k(x)=\frac{x^k-kx+k-1}{k(k-1)}$ with $k>1$, the convex conjugate
$\phi_k^*$ admits a closed form (see, e.g., \citet{JohnDuchi}), which leads to a simplified dual expression that
eliminates $\lambda$ and reduces the dual to a one-dimensional maximization over $\beta$ (Appendix~\ref{app:cr}).
Moreover, because the resulting objective has the form ``linear minus an $L_{k_*}$-norm'' with $k_*=\frac{k}{k-1}>1$,
it is concave and typically strictly concave under mild non-degeneracy of $Z_\pi(B)$, implying uniqueness of the
dual optimizer $\beta^*$.

The CLT proof in Theorem~\ref{thm:kl-clt-main} follows a standard path: (i) a uniform CLT for the rMLMC
estimator over a compact parameter set, and (ii) a functional delta theorem for the corresponding maximization
functional. The same approach extends to the general $\phi$-divergence objective \eqref{eq:phi-emp-obj} under
mild conditions ensuring (a) finite variance and stochastic equicontinuity of
$(\lambda,\beta)\mapsto \widehat G_{n,\pi}(\lambda,\beta)$ on compact sets, and (b) suitable regularity of the
argmax mapping (e.g., uniqueness or a directional derivative characterization). 

\section{Details of Section \ref{sec:policy-learning}}\label{sec:finance}
In this section, we first review the solutions of the Merton problem mentioned in Section \ref{sec:merton}. 
\subsection{Optimal Solutions of the Bayesian Merton Problem}
By adapting the methods and results from \citet{KZ98}, the optimal solution is given by the following theorem. 
\begin{theorem}\label{Karatzassolution}
    The optimal value function of Problem (\ref{eq:merton-bayes}) is given by
   $$V(x_0) = \frac{\left(x_0e^{rT}\right)^{\alpha}}{\alpha}\left(\int_{\mathbb{R}^d}\left(F(T,z)\right)^{\frac{1}{1-\alpha}}\varphi_T(z)dz\right)^{1-\alpha},$$
   and the optimal fractions invested in each stock for each time $t \geq 0$ is given by the vector
   $$\frac{\pi^*_t}{X^*_t} = {\left(\sigma^T\right)}^{-1}\frac{\int_{\mathbb{R}^d}\nabla F\left(T,z+Y_t\right)\left(F\left(T,z+Y_t\right)\right)^{\frac{\alpha}{1-\alpha}}\varphi_{T-t}(z)dz}{(1-\alpha) \int_{\mathbb{R}^d}\left(F\left(T,z+Y_t\right)\right)^{\frac{1}{1-\alpha}}\varphi_{T-t}(z)dz},$$
   where for fixed $s > 0$, $\varphi_s$ is the density function of the $d$-dimensional multivariate Gaussian distribution $\mathcal{N}(0,sI_d)$ ($0$ is the zero $d \times 1$ vector and $I_d$ is the $d \times d$ identity matrix), 
   \begin{equation}\label{YS}
    Y_t =\sigma^{-1} \left(B-r\left(1, \ldots, 1\right)^T\right)t + W_t,   
   \end{equation}
 and $F(t,y): =F_{\mu}(t,y) = \int_{\mathbb{R}^d}L_t(z,y)d\mu(z)$ with $L_t(z,y) = 1$ if $t = 0$ and $$L_t(z,y) = \exp\left(\left(\sigma^{-1}\left(z-r\left(1, \ldots, 1\right)^T\right)\right)^Ty - \frac{1}{2}\left\Vert\sigma^{-1}\left(z-r\left(1, \ldots, 1\right)^T\right)\right\Vert^2t\right)$$ if $t > 0$. Moreover, the filtration generated by the process $\{Y_t\}_{t \in [0,T]}$ is the same as $\mathcal{F}^S$ under the probability measure $P$.
\end{theorem}
In practice, the prior distribution is chosen by experts and other available information, and in reality the fraction of investment into risky asset is computed via the formula provided by Theorem \ref{Karatzassolution} with real observations.
\subsection{Alternate Admissible Controls}\label{alterdef}
In this section, we give the definition of the admissible controls that will help the tractability of the general case.

To illustrate our definition of the admissible set, we go back to Problem (\ref{eq:merton-bayes}) with the strictly concave utility function $u$. It is shown that the optimal terminal wealth $X^*_T$ is given by 
$X^*_T = I\left(\frac{\mathcal{K}(x_0)e^{-rT}}{F(T,Y(T))}\right)$ \citep{KZ98}.  
For a density function $\lambda: \mathbb{R} \to \mathbb{R}$, we define the inner product $\left\langle E^{P^.}\left[u(X^*_T)\right], \lambda \right\rangle := \int_{\mathbb{R}}E^{P^b}\left[u(X^*_T)\right]\lambda(b)d\mu(b) = \int_{\mathbb{R}}E^{P^b}\left[\left(u \circ I\right)\left(\frac{\mathcal{K}(x_0)e^{-rT}}{F(T,Y(T))}\right)\right]\lambda(b)d\mu(b).$
% \begin{align*}
% \int_{\mathbb{R}}E^{P^b}\left[u(\hat{X}_T)\right]\lambda(b)d\mu(b) &= \int_{\mathbb{R}}E^{P^b}\left[u\left(I\left(\frac{\mathcal{K}(x_0)e^{-rT}}{F(T,Y(T))}\right)\right)\right]\lambda(b)d\mu(b)\\
% &= \int_{\mathbb{R}}E^{P^b}\left[\left(u \circ I\right)\left(\frac{\mathcal{K}(x_0)e^{-rT}}{F(T,Y(T))}\right)\right]\lambda(b)d\mu(b).
% \end{align*}
% With some regularity assumptions in Assumption 3.1 in \citet{KZ98}, $\mathcal{K}$ is a continuous and invertible function of $x$.
We want to see whether there exists a function $h: \mathbb{R} \to \mathbb{R}$ and a functional $\rho: L^1 \to \mathbb{R}$ such that
\begin{equation}\label{eq1}
\int_{\mathbb{R}}E^{P^b}\left[u(X^*_T)\right]\lambda(b)d\mu(b) = h(x_0)\rho(\lambda)    
\end{equation}
and \begin{equation}\label{eq2}
    h(1) = u(e^{rT}).
\end{equation}

\begin{theorem}\label{decompose}
    For the following cases (1) and (2), the conditions (\ref{eq1}) and (\ref{eq2}) are met, where for case (3), the conditions are not met: 
    \begin{itemize}
        \item (1) $u(x) = \frac{1}{\alpha}x^{\alpha} $, where $\alpha < 1$ and $\alpha \neq 0$.
        \item (2) $u(x) = \frac{-1}{\gamma}e^{-\gamma x} $, where $\gamma > 0$.
        \item (3) $u(x) = \log(x)$.
    \end{itemize}
\end{theorem}

\begin{proof}
    \begin{itemize}
        \item (1) Here, $$I(y) = y^{\frac{1}{1-\alpha}},$$ thus $$(u\circ I)(y) = \frac{1}{\alpha}y^{\frac{\alpha}{1-\alpha}},$$
        which implies that 
        \begin{align*}
\int_{\mathbb{R}}E^{P^b}\left[u(X^*_T)\right]\lambda(b)d\mu(b) &=  \int_{\mathbb{R}}E^{P^b}\left[\left(u \circ I\right)\left(\frac{\mathcal{K}(x_0)e^{-rT}}{F(T,Y(T))}\right)\right]\lambda(b)d\mu(b)\\
&=\int_{\mathbb{R}}E^{P^b}\left[\frac{1}{\alpha}
\left(\frac{\mathcal{K}(x_0)e^{-rT}}{F(T,Y(T))}\right)^{\frac{\alpha}{1-\alpha}}\right]\lambda(b)d\mu(b)\\
&=\frac{1}{\alpha}\left(\mathcal{K}(x_0)e^{-rT}\right)^{\frac{\alpha}{1-\alpha}}\int_{\mathbb{R}}E^{P^b}\left[
\left(F(T,Y(T))\right)^{\frac{\alpha}{\alpha-1}}\right]\lambda(b)d\mu(b)\\
&= h(x_0)\rho(\lambda),
        \end{align*}
        where $$h(x) = \frac{1}{\alpha}\left(\mathcal{K}(x)e^{-rT}\right)^{\frac{\alpha}{1-\alpha}} $$
        and 
        $$\rho(\lambda) = \int_{\mathbb{R}}E^{P^b}\left[
\left(F(T,Y(T))\right)^{\frac{\alpha}{\alpha-1}}\right]\lambda(b)d\mu(b).$$
        \item (2) Here, $$I(y) = \frac{-1}{\gamma}\log(y),$$ thus $$(u\circ I)(y) = \frac{-1}{\gamma}y,$$
        which implies that 
        \begin{align*}
\int_{\mathbb{R}}E^{P^b}\left[u(X^*_T)\right]\lambda(b)d\mu(b) &=  \int_{\mathbb{R}}E^{P^b}\left[\left(u \circ I\right)\left(\frac{\mathcal{K}(x_0)e^{-rT}}{F(T,Y(T))}\right)\right]\lambda(b)d\mu(b)\\
&=\int_{\mathbb{R}}E^{P^b}\left[\frac{-1}{\gamma}\left(\frac{\mathcal{K}(x_0)e^{-rT}}{F(T,Y(T))}\right)\right]\lambda(b)d\mu(b)\\
&= h(x_0)\rho(\lambda),
        \end{align*}
        where $$h(x) = \frac{-1}{\gamma}\mathcal{K}(x_0)e^{-rT} $$
        and 
        $$\rho(\lambda) = \int_{\mathbb{R}}E^{P^b}\left[
\frac{1}{F(T,Y(T))}\right]\lambda(b)d\mu(b).$$
\item (3) Here, $$I(y) = \frac{1}{y} \text{ and }\mathcal{K}(x_0) = \frac{1}{x_0},$$ thus $$(u\circ I)(y) = -\log(y),$$
        which implies that 
        \begin{align*}
\int_{\mathbb{R}}E^{P^b}\left[u(X^*_T)\right]\lambda(b)d\mu(b) &=  \int_{\mathbb{R}}E^{P^b}\left[\left(u \circ I\right)\left(\frac{\mathcal{K}(x_0)e^{-rT}}{F(T,Y(T))}\right)\right]\lambda(b)d\mu(b)\\
&=  \int_{\mathbb{R}}E^{P^b}\left[-\log\left(\frac{\mathcal{K}(x_0)e^{-rT}}{F(T,Y(T))}\right)\right]\lambda(b)d\mu(b)\\
&=x_0 + rT + \int_{\mathbb{R}}E^{P^b}\left[\log\left(F(T,Y(T))\right)\right]\lambda(b)d\mu(b),
        \end{align*}which shows that the decomposition is impossible.
    \end{itemize}
\end{proof}

\begin{remark}
    In general, if $u \circ I$ satisfies $$u \circ I\left(\frac{a}{b}\right) = f(a)g(b),$$
    then the decomposition condition holds. There are other examples satisfying this condition.
\end{remark}

Now we define a subset $ \tilde{\mathcal{A}}(x_0) \subset \mathcal{A}(x_0)$ such that for all $\pi \in \tilde{\mathcal{A}}(x_0)$, there exists corresponding function $h: \mathbb{R} \to \mathbb{R}$ and a functional $\rho: L^1 \to \mathbb{R}$ such that for the controlled terminal wealth $X_T^{\pi}$,
$\int_{\mathbb{R}}E^{P^b}\left[u(X_T^{\pi})\right]\lambda(b)d\mu(b) = h(x_0)\rho(\lambda)    
$
and $h(1) = u(e^{rT}).$
Thus for some utility functions, Problem (\ref{eq:merton-bayes}) is equivalent to 
$\sup_{\pi \in \tilde{\mathcal{A}}(x_0)}E^P\left[u(X_T)\right].$
Besides, if we define a function $\tilde{u}(x) = u(x) + C$, where $C$ is a constant that does not depend on $x$, then the problems $\sup_{\pi \in \tilde{\mathcal{A}}(x_0)}E^P\left[u(X_T)\right]$
and $\sup_{\pi \in \tilde{\mathcal{A}}(x_0)}E^P\left[\tilde{u}(X_T)\right]$
should have the same optimal solution (not the same value function), at least in the definition of the classical problem. However this is not true if we follow our current definition of admissible controls since the decomposition cannot be satisfied by these two problems at the same time.

To rescue this, we firstly define an equivalence relation for two real-valued function $f$ and $g$
\begin{equation*}
f \sim g \iff 
\begin{aligned}
    & \text{ there exists a constant }C\text{ such that } f(x) = g(x) + C.
\end{aligned}
\end{equation*}
We call $\tilde{\mathcal{A}}(x_0)$ the collection of all alternate admissible controls $\pi$, which is a subset of $\mathcal{A}(x_0)$ such that 
there exists a function $v \sim u$ ($u$ is the utility function in the objective function) such that there exist corresponding function $h: \mathbb{R} \to \mathbb{R}$ and a functional $\rho: L^1 \to \mathbb{R}$ such that for the controlled terminal wealth $X_T^{\pi}$,
$
\int_{\mathbb{R}}E^{P^b}\left[v(X_T^{\pi})\right]\lambda(b)d\mu(b) = h(x_0)\rho(\lambda)    
$
and $
    h(1) = v(e^{rT}).
$

In the DRBC formulation, since the deviation from the prior distribution and its corresponding underlying probability space is small, it is reasonable to continue to search for optimal solutions in the space of $\tilde{\mathcal{A}}(x_0)$ for those utilities. Thus we call $\tilde{\mathcal{A}}(x_0)$ the collection of all alternate admissible controls and Problem (\ref{eq:merton-inner}) becomes
\begin{equation}\label{quasinewinner}
    \sup_{\pi \in \tilde{\mathcal{A}}(x_0)} \int_{\mathbb{R}}\Phi_{\lambda,\beta}\left(E^{P^b}\left[u(X_T)\right]\right)d\mu(b).
\end{equation}

\subsection{Properties of \texorpdfstring{$Q^b$}{Qb}}\label{transformderive}
\begin{theorem}
    Under $Q^b$ for any $b \in \mathbb
    {R}$, the distribution of $B$ is still $\mu$.
\end{theorem}
\begin{proof}
    Let $A \in \mathcal{F}$, then from independence of $B$ and $W$, we have
    \begin{align*}
        \mu_b(A) &= Q^b(B \in A) = E^{Q^b}\left[\mathbf{1}_{\{B \in A\}}\right] \\
        &= E^P\left[\frac{dQ^b}{dP}\mathbf{1}_{\{B \in A\}}\right] = E^P\left[\exp\left(-\frac{B-b}{\sigma}W_T - \frac{(B-b)^2}{2\sigma^2}T\right)\mathbf{1}_{\{B \in A\}}\right]\\
        &= \int_{A}\int_{\mathbb{R}}\exp\left(-\frac{x-b}{\sigma}y - \frac{(x-b)^2}{2\sigma^2}T\right)\mu(x)f_{W_T}(y)dydx\\&= 
        \int_{A}\int_{\mathbb{R}}\exp\left(-\frac{x-b}{\sigma}y - \frac{(x-b)^2}{2\sigma^2}T\right)\mu(x)\frac{1}{\sqrt{T}}\varphi\left(\frac{y}{\sqrt{T}}\right)(y)dydx.
    \end{align*}
    Therefore, \begin{align*}
        \mu_b(x) &= \mu(x)\int_{\mathbb{R}}\frac{1}{\sqrt{T}}\exp\left(-\frac{y(x-b)}{\sigma} - \frac{(x-b)^2}{2\sigma^2}T\right)\varphi\left(\frac{y}{\sqrt{T}}\right)dy\\
        &= \frac{\mu(x)}{\sqrt{2\pi T}}\int_{\mathbb{R}}\exp\left(\frac{y(b-x)}{\sigma} - \frac{(x-b)^2}{2\sigma^2}T - \frac{y^2}{2T}\right)dy\\
        &= \frac{\mu(x)}{\sqrt{2\pi T}}\int_{\mathbb{R}}\exp\left(-\frac{1}{2T}\left(y^2 - \frac{2Ty(b-x)}{\sigma} + \frac{(x-b)^2}{\sigma^2} T^2\right)\right)dy\\
        &= \frac{\mu(x)}{\sqrt{2\pi T}}\int_{\mathbb{R}}\exp\left(-\frac{1}{2T}\left(y-\frac{b-x}{\sigma}T\right)^2\right)dy\\
        &= \mu(x).
        % &= \frac{\mu(x)}{\sqrt{2\pi T}}\int_{\mathbb{R}}\exp\left(-\frac{y(x-b)}{\sigma} - \frac{(x-b)^2}{2\sigma^2}T - \frac{y^2}{2T}\right)dy\\
        % &= \frac{\mu(x)}{\sqrt{2\pi T}}\int_{\mathbb{R}}\exp\left(-\frac{1}{2T}\left(y^2 + \frac{2T(x-b)}{\sigma}y\right)-\frac{(x-b)^2}{\sigma}T\right)dy\\
        % &= \frac{\mu(x)}{\sqrt{2\pi T}}\exp\left(-\frac{(x-b)^2}{2\sigma^2}T\right)\int_{\mathbb{R}}\exp\left(-\frac{1}{2T}\left(y + T\frac{(x-b)^2}{\sigma}\right)^2\right)dy\\
        % &= \mu(x)\exp\left(-\frac{(x-b)^2}{2\sigma^2}T\right).
    \end{align*}
    % If $\mu \sim \mathcal{N}\left(\mu_0, \sigma_0^2\right)$, then 
    % \begin{align*}
    %    \mu_b(x) &= \frac{1}{\sqrt{2\pi}\sigma_0}\exp\left(-\frac{1}{2\sigma_0}\left(x-\mu_0\right)^2\right)\exp\left(-\frac{(x-b)^2}{2\sigma^2}T\right)\\
    %    &= \frac{1}{\sqrt{2\pi}\sigma_0}\exp\left(-\frac{(x-\mu_0)^2}{2\sigma_0^2}-\frac{(x-b)^2}{2\sigma^2}T\right)\\
    %    &= \frac{1}{\sqrt{2\pi} \sigma_b}\exp\left(-\frac{1}{2\sigma_b^2}\left(x-\mu_b\right)^2\right),
    % \end{align*}where $\mu_b = \frac{T\sigma_0^2 b + \sigma^2 \mu_0}{T\sigma_0^2 + \sigma^2}$ and $\sigma_b = \frac{\sigma_0^2 \sigma^2}{T\sigma_0^2 + \sigma^2}$, which finishes the proof.
\end{proof}

\subsection{Proof of Theorem \ref{system}}
\begin{proof}
From Corollary 3.8 in \citet{GLS22}, if $X_T$ is a terminal wealth such that there exists a constant $\kappa > 0$ with
\begin{equation}\label{explode}
\left\{
\begin{aligned}
& u'(X_T) = \frac{\kappa \eta_T \int_\mathbb{R} \Phi_{\lambda, \beta}' \left( E^{Q^b} \left[ u(X_T) \right] \right) d\mu(b)}{\int_{\mathbb{R}} \Phi_{\lambda, \beta}' \left( E^{Q^b} \left[ u(X_T) \right] \right) \eta_T^b \, d\mu(b)}\\
& E^{Q^{r}} \left[ X_T  \right] = x_0 e^{r T},
\end{aligned}
\right.
\end{equation}
    then $X_T$ is the optimal terminal wealth corresponding to Problem (\ref{eq:merton-inner}).
Since $u$ and $\Phi_{\lambda, \beta}$ are both strictly increasing and strictly concave, then $u' > 0$ and $\Phi_{\lambda, \beta}' > 0$, hence the first equation is equivalent to (we abuse the notation by using $\kappa$ for both the positive multiplier and its logarithm) \begin{align*}
    \log\left(u'(X_T)\right) &= \kappa + \log(\eta_T) + \log\left(\int_{\mathbb{R}}\Phi_{\lambda, \beta}' \left( E^{Q^b} \left[ u(X_T) \right] \right) d\mu(b)\right) \\
    &-\log\left(\int_{\mathbb{R}} \Phi_{\lambda, \beta}' \left( E^{Q^b} \left[ u(X_T) \right] \right) \eta_T^b \, d\mu(b)\right),
\end{align*}
where
\[
\log(\eta_T)
=
\frac{r-b_0}{\sigma}\left(W_T + \frac{B-b_0}{\sigma}T\right)
-
\frac{(b_0-r)^2}{2\sigma^2}T.
\]
The factor $\Phi_{\lambda,\beta}'\!\left(E^{Q^b}[u(X_T)]\right)$ does not itself contain an exponential term that can explode. Therefore, it remains to simplify
\[
\log\left(
\int_{\mathbb{R}}
\Phi_{\lambda,\beta}'\!\left(E^{Q^b}[u(X_T)]\right)
\eta_T^b\,d\mu(b)
\right).
\]

% We notice that $\log\left(\int_{\mathbb{R}} \Phi_{\lambda, \beta}' \left( E^{Q^b} \left[ u(X_T) \right] \right) \eta_T^b \, d\mu(b)\right)$ is indeed a random variable where the randomness is from the $B$ and $W_T$, and the integral can be seen as taken over another random variable $A$ (since $b$ is dummy) which is independent of $B$ and $W_T$, but $A$ has the same distribution as $B$ under $P$. In other words, 
Suppose $A: \Omega \to \mathbb{R}$ is a random variable that is independent from $B$ and $W$ with distribution $\mu$, and if we denote the $\sigma$-algebra generated by $B$ and $W_T$ as $\mathcal{G}$, then
\begin{align*}
   &\int_{\mathbb{R}} \Phi_{\lambda, \beta}' \left( E^{Q^b} \left[ u(X_T) \right] \right) \eta_T^b \, d\mu(b)\\ &= \int_{\mathbb{R}} \Phi_{\lambda, \beta}' \left( E^{Q^a} \left[ u(X_T) \right] \right) \exp\left(-\frac{b_0-a}{\sigma}\left(W_T + \frac{B-b_0}{\sigma}T \right)-\frac{(b_0-a)^2}{2\sigma^2}T\right) \, d\mu(a)\\
   &= E^P\left[\Phi_{\lambda, \beta}' \left( E^{Q^A} \left[ u(X_T) \right] \right)\exp\left(-\frac{b_0-A}{\sigma}\left(W_T + \frac{B-b_0}{\sigma}T \right)-\frac{(b_0-A)^2}{2\sigma^2}T\right) \,\middle\vert\, \mathcal{G}\right]\\
   &= E^Q\left[\Phi_{\lambda, \beta}' \left( E^{Q^A} \left[ u(X_T) \right] \right)\,\middle\vert\, \mathcal{G} \right]
   E^P\left[\exp\left(-\frac{b_0-A}{\sigma}\left(W_T + \frac{B-b_0}{\sigma}T \right)-\frac{(b_0-A)^2}{2\sigma^2}T\right) \,\middle\vert\, \mathcal{G} \right],
\end{align*}
where the last equality is from abstract Bayes' rule \citep{absBayes} with the Radon-Nikodym derivative 
$$\frac{dQ}{dP} = \frac{\exp\left(-\frac{b_0-A}{\sigma}\left(W_T + \frac{B-b_0}{\sigma}T \right)-\frac{(b_0-A)^2}{2\sigma^2}T\right)}{E^P\left[\exp\left(-\frac{b_0-A}{\sigma}\left(W_T + \frac{B-b_0}{\sigma}T \right)-\frac{(b_0-A)^2}{2\sigma^2}T\right) \,\middle\vert\, \mathcal{G} \right]}.$$
% For the second term, since $A \sim \mu \sim \mathcal{N}\left(\mu_0, \sigma_0^2\right)$ and $A$, $B$, $W_T$ are mutually independent under $P$, we denote $D = W_T + \frac{B - b_0}{\sigma} T$ for convenience of notations, then 
% \begin{align*}
% I: &= E^P\left[\exp\left(-\frac{b_0-A}{\sigma}D-\frac{(b_0-A)^2}{2\sigma^2}T\right) \,\middle\vert\, \mathcal{G} \right]\\
% &= \exp\left(
%     -C_0 + \frac{C_1^2 \sigma_0^2}{2} - C_2 \mu_0^2
% \right)
% \cdot
% \frac{1}{\sqrt{1 + 2 C_2 \sigma_0^2}},
% \end{align*}
% where 
% $$C_0 = \frac{b_0}{\sigma} \left( W_T + \frac{B - b_0 T}{\sigma} \right) + \frac{b_0^2 T}{2 \sigma^2},$$
% $$C_1 = -\frac{1}{\sigma} \left( W_T + \frac{B - b_0 T}{\sigma} \right) + \frac{b_0 T}{\sigma^2},$$and
% $$C_2 = \frac{T}{2 \sigma^2}.$$
For the first term, we need to know the distribution of $A$ under $Q$ conditioned on $\mathcal{G}$. That is, for any Borel measurable set $E \subset \mathbb{R}$, from the abstract Bayes' rule, we want to compute
\begin{align*}
    Q\left(A \in E\,\middle|\, \mathcal{G}\right) &= E^Q\left[\mathbf{1}_{\{A \in E\}}\,\middle|\, \mathcal{G}\right]= \frac{E^P\left[ \dfrac{dQ}{dP} \mathbf{1}_{\{A \in E\}} \,\middle|\, \mathcal{G} \right]}{E^P\left[ \dfrac{dQ}{dP} \,\middle|\, \mathcal{G} \right]}\\
     &= \frac{ \displaystyle E^P\left[ \exp\left( -\dfrac{b_0 - A}{\sigma} D - \dfrac{(b_0 - A)^2}{2 \sigma^2} T \right) \mathbf{1}_{\{ A \in E \}} \, \middle| \, \mathcal{G} \right] }{ \displaystyle E^P\left[ \exp\left( -\dfrac{b_0 - A}{\sigma} D - \dfrac{(b_0 - A)^2}{2 \sigma^2} T \right) \, \middle| \, \mathcal{G} \right] }.
\end{align*}
Moreover, if $D = W_T + \frac{B-b_0}{\sigma}T$, then
\begin{align*}
\text{Numerator} &= \int_E \exp\left( -\dfrac{b_0 - a}{\sigma} D - \dfrac{(b_0 - a)^2}{2 \sigma^2} T \right) \cdot \frac{1}{\sqrt{2 \pi \sigma_0^2}} \exp\left( -\dfrac{(a - \mu_0)^2}{2 \sigma_0^2} \right) \, da, \\
\text{Denominator} &= \int_{-\infty}^{\infty} \exp\left( -\dfrac{b_0 - a}{\sigma} D - \dfrac{(b_0 - a)^2}{2 \sigma^2} T \right) \cdot \frac{1}{\sqrt{2 \pi \sigma_0^2}} \exp\left( -\dfrac{(a - \mu_0)^2}{2 \sigma_0^2} \right) \, da.
\end{align*}
If we define $K_1 = \dfrac{T}{2 \sigma^2} + \dfrac{1}{2 \sigma_0^2}, $
$K_2 = \dfrac{T b_0}{\sigma^2} + \dfrac{D}{\sigma} + \dfrac{\mu_0}{\sigma_0^2},$
$K_3 = \dfrac{T b_0^2}{2 \sigma^2} + \dfrac{b_0 D}{\sigma} + \dfrac{\mu_0^2}{2 \sigma_0^2},$ and $ C = \dfrac{K_2^2}{4 K_1} - K_3,$ then
\begin{align*}
  Q\left(A \in E\,\middle|\, \mathcal{G}\right) &= \frac{ \exp(C) \displaystyle \int_E \exp\left( -K_1 \left( a - \dfrac{K_2}{2 K_1} \right )^2 \right ) da }{ \exp(C) \int_{-\infty}^{\infty} \exp\left( -K_1 \left( a - \dfrac{K_2}{2 K_1} \right )^2 \right ) da}\\
  &= \frac{ \exp(C) \displaystyle \int_E \exp\left( -K_1 \left( a - \dfrac{K_2}{2 K_1} \right )^2 \right ) da }{ \exp(C) \sqrt{ \dfrac{\pi}{K_1} } } \\
&= \int_E \dfrac{1}{ \sqrt{ \pi / K_1 } } \exp\left( -K_1 \left( a - \dfrac{K_2}{2 K_1} \right )^2 \right ) da \\
&= \int_E \dfrac{1}{ \sqrt{2 \pi \sigma_Q^2} } \exp\left( -\dfrac{ (a - \mu_Q)^2 }{ 2 \sigma_Q^2 } \right ) da,  
\end{align*}
where
\begin{align*}
\mu_Q
&=
\frac{K_2}{2K_1}
=
\sigma_Q^2\left(
\frac{Tb_0}{\sigma^2}+
\frac{D}{\sigma}+
\frac{\mu_0}{\sigma_0^2}
\right)\\
&=
\sigma_Q^2\left(
\frac{W_T}{\sigma}+
\frac{BT}{\sigma^2}+
\frac{\mu_0}{\sigma_0^2}
\right),\\
\sigma_Q^2
&=
\frac{1}{2K_1}
=
\left(
\frac{T}{\sigma^2}+
\frac{1}{\sigma_0^2}
\right)^{-1}.
\end{align*}
This computation also gives
\begin{align*}
&E^P\left[
\exp\left(
-\frac{b_0-A}{\sigma}D
-\frac{(b_0-A)^2}{2\sigma^2}T
\right)
\,\middle|\,\mathcal G
\right]
=
\exp(C)\frac{\sigma_Q}{\sigma_0}.
\end{align*}
Thus, under \( Q \) conditioned on \( \mathcal{G} \), \( A \) follows a Gaussian distribution $\mathcal{N}\left( \mu_Q, \sigma_Q^2 \right ) \sim \mu_A$.
Therefore,
\begin{align*}
    \text{log term }&=\log\left(\int_{\mathbb{R}} \Phi_{\lambda, \beta}' \left( E^{Q^b} \left[ u(X_T) \right] \right) \eta_T^b \, d\mu(b)\right)\\ 
%     &= \log\left(E^Q\left[\Phi_{\lambda, \beta}' \left( E^{Q^A} \left[ u(X_T) \right] \right)\right]\right) + \log\left(\exp\left(
%     -C_0 + \frac{C_1^2 \sigma_0^2}{2} - C_2 \mu_0^2
% \right)
% \cdot
% \frac{1}{\sqrt{1 + 2 C_2 \sigma_0^2}}\right)\\
% &=\log\left(\int_{\mathbb{R}}\Phi_{\lambda, \beta}'\left(E^{Q^a}[u(X_T)]\right)d\mu_A(a)\right) -C_0 + \frac{C_1^2 \sigma_0^2}{2} - C_2 \mu_0^2- \frac{1}{2}\log\left(1 + 2C_2\sigma_0^2\right).
&= \log\left(\exp(C) \frac{\sigma_Q}{\sigma_0}\right) + \log\left(\int_{\mathbb{R}}\Phi_{\lambda, \beta}'\left(E^{Q^a}[u(X_T)]\right)d\mu_A(a)\right)\\
&= C + \log\left(\sigma_Q\right) - \log\left(\sigma_0\right) + \log\left(\int_{\mathbb{R}}\Phi_{\lambda, \beta}'\left(E^{Q^a}[u(X_T)]\right)d\mu_A(a)\right).
\end{align*}
Hence, if $X_T$ is a terminal wealth such that there exists a constant $\kappa > 0$ with $E^{Q^{r}} \left[ X_T  \right] = x_0 e^{r T}$ and 
\begin{align*}
    \log\left(u'(X_T)\right) &= \kappa + \frac{r-b_0}{\sigma}\left(W_T + \frac{B-b_0}{\sigma}T\right) - \frac{(b_0-r)^2}{2\sigma^2}T \\
    &+ \log\left(\int_{\mathbb{R}}\Phi_{\lambda, \beta}' \left( E^{Q^b} \left[ u(X_T) \right] \right) d\mu(b)\right)\\ &- \log\left(\int_{\mathbb{R}} \Phi_{\lambda, \beta}' \left( E^{Q^b} \left[ u(X_T) \right] \right) \eta_T^b \, d\mu(b)\right)\\
    &=\log(\kappa) + \frac{r-b_0}{\sigma}\left(W_T + \frac{B-b_0}{\sigma}T \right)- \frac{(b_0-r)^2}{2\sigma^2}T \\
    &+ \log\left(\int_{\mathbb{R}}\Phi_{\lambda, \beta}' \left( E^{Q^b} \left[ u(X_T) \right] \right) d\mu(b)\right)\\&-\log\left(\int_{\mathbb{R}}\Phi_{\lambda, \beta}'\left(E^{Q^a}[u(X_T)]\right)d\mu_A(a)\right) -C - \log\left(\sigma_Q\right) +\log\left(\sigma_0\right),
\end{align*}then $X_T$ is the optimal terminal wealth corresponding to Problem (\ref{eq:merton-inner}). If we define \begin{align*}
L_{\kappa}(X_T) &=
\kappa - \log\left(u'(X_T)\right) +\log\left(\int_{\mathbb{R}}\Phi_{\lambda, \beta}' \left( E^{Q^b} \left[ u(X_T) \right] \right) d\mu(b)\right)\\
&-\log\left(\int_{\mathbb{R}}\Phi_{\lambda, \beta}'\left(E^{Q^a}[u(X_T)]\right)d\mu_A(a)\right)\notag\\
&\quad + K_3 - \frac{K_2^2}{4K_1} - \log\left(\sigma_Q\right) +\log\left(\sigma_0\right)+ \frac{r-b_0}{\sigma}\left(W_T + \frac{B-b_0}{\sigma}T\right) - \frac{(b_0-r)^2}{2\sigma^2}T,\notag
\end{align*}then $L_{\kappa}(X_T) = 0$.
\end{proof}

\section{Algorithms}\label{algo}
\subsection{Implementation Details for the Merton Problem}
Here we elaborate on the details of simulating a sample of the optimal terminal wealth $X_T$. Recall that when $d=1$, the controlled SDE is given by
\begin{equation}\label{SDE1}
d X_t = (X_t - \pi_t) r \, d t + \pi_t \left(B d t + \sigma d W_t\right),    
\end{equation}
where $\pi_t$ is the amount of money invested in the stock at time $t \in [0,T]$. Theorem \ref{Karatzassolution} provides the optimal fraction of total wealth invested in stock at time $t$. If we model $\pi_t$ as the fraction of total wealth invested in stock at time $t$, then the controlled SDE becomes 
\begin{equation}\label{SDE2}
  d X_t = X_t \left(rdt + \pi_t\left(B-r\right)dt + \sigma dW_t\right).  
\end{equation}These two formulations are equivalent since they give the same optimal value function (thus the same optimal terminal wealth) and the optimal fraction for the formulation (\ref{SDE1}) indeed gives the optimal control for the formulation (\ref{SDE2}). Thus, in the implementations, when we sample $X_T$ for the KL divergence case, we plug in the optimal fraction for Equation (\ref{SDE2}) and use either Euler's method or rMLMC method \citep{RheeGlynn2015}.

\subsection{Deep Learning Method to Learn Optimal Terminal Wealth for the General Case}\label{algo3}
In this section, we discuss and give the Algorithm \ref{first} to compute the loss function $\mathcal{L}(\boldsymbol{\theta})$ with fixed neural network parameter (get high quality estimates (e.g. unbiased, low variance, and fast computing speed)), and then auto-differentiation (back propagation neural network \citep{rumelhart1986learning, lecun2015deep}) can be used to do the optimizations (Algorithm \ref{BPG}). 

Before discussing the loss function (\ref{loss}), we remark that Equation (\ref{explode}) also motivates a loss function for $b_1 \in \mathbb{R}$,
\begin{align}\label{badloss}
&\mathcal{\tilde{L}}(\boldsymbol{\theta}) =E^{Q^{b_1}} \bigg[\bigg\|
u'(h_{\theta}(W_T,B))\int_{\mathbb{R}} \Phi_{\lambda, \beta}' \left( E^{Q^a} \left[ u(X_T) \right] \right)* \\
&\exp\left(-\frac{b_0-a}{\sigma}\left(W_T + \frac{B-b_0}{\sigma}T \right)-\frac{(b_0-a)^2}{2\sigma^2}T\right) \, d\mu(a) \notag\\ 
&- \kappa \exp\left(\frac{r-b_0}{\sigma}\left(W_T + \frac{B-b_0}{\sigma}T\right) - \frac{(b_0-r)^2}{2\sigma^2}T\right)\int_{\mathbb{R}}\Phi_{\lambda, \beta}' \left( E^{Q^b} \left[ u(h_{\theta}(W_T,B)) \right] \right) d\mu(b)
\bigg\|_2^2\bigg] \notag\\
&\quad + \Bigg( E^{Q^{r}}[h_{\theta}(W_T,B)] - x_0 e^{rT} \Bigg)^2. \notag
\end{align}
In theory, loss (\ref{loss}) and (\ref{badloss}) are equivalent. However, in terms of numerical computation, (\ref{badloss}) is bad since when we initializing neural network parameters and do the optimization steps, we cannot control the flow and with high probability, the terms in (\ref{badloss}) becomes $\infty$ since $\frac{1}{\sigma^2}$ is small. Moreover, the scalar learnable parameter in (\ref{badloss}) is restricted to be positive so we need to do projected gradient descent, while in (\ref{loss}) the scalar learnable parameter can be any real number. Therefore, in order to achieve the numerical stability and convenience for implementation, we choose loss (\ref{loss}).

% The main challenges in this computation are to compute $\int_{\mathbb{R}} \phi\left(E^{Q^b}\left[u(h_{\theta}(W_T, B))\right]\right) d\mu(b)$ and $\int_{\mathbb{R}} \phi\left(E^{Q^a}\left[u(h_{\theta}(W_T, B))\right]\right) d\mu_A(a)$ (with realizations of $W_T$ and $B$), which motivates Algorithm \ref{first}.

Recall that from Theorem \ref{system}, if we replace $\Phi_{\lambda, \beta}$ by $\phi$, then the conditions that the parametrized optimal terminal wealth $h_{\theta}(W_T,B)$ needs to satisfy are
$$E^{Q^{r}}[h_{\theta}(W_T,B)] = x_0 e^{rT}$$ and
\begin{align}
\log\left(u'(h_{\theta}(W_T,B))\right)  &=  \log\left(\int_{\mathbb{R}}\phi' \left( E^{Q^b} \left[ u(h_{\theta}(W_T,B)) \right] \right) d\mu(b)\right) \text{        (constant log term)}\label{1}\\ &-\log\left(\int_{\mathbb{R}}\phi'\left(E^{Q^a}[u(h_{\theta}(W_T,B))]\right)d\mu_A(a)\right)\text{ (random log term)}\label{2}\\
&\quad+ \kappa + K_3 - \frac{K_2^2}{4K_1} - \log\left(\sigma_Q\right) +\log\left(\sigma_0\right) \\
&+\frac{r-b_0}{\sigma}\left(W_T + \frac{B-b_0}{\sigma}T\right) - \frac{(b_0-r)^2}{2\sigma^2}T,\label{3}
\end{align}where the second condition is an equation of random variables, thus for each $b_1$ in the support of $\mu$, we need to sample $W_T$ and $B$ under the probability measure $Q^{b_1}$. From Appendix \ref{transformderive}, under $Q^{b_1}$, the distribution of $B$ is unchanged. Moreover, since under $Q^{b_1}$, $W^{b_1}$ is a standard Brownian motion and $W^{b_1}_t = W_t + \frac{B-b_1}{\sigma}t$, then we get samples of $W_T$ by first sampling $N \sim \mathcal{N}(0,T)$, $B \sim \mu$, and then compute $W_T = N - \frac{B-b_1}{\sigma}T$. The reason to do this can be found in Section \ref{algo4}.

% We choose the probability measure $Q^r$ with the following intuitions and reasons. To begin with, since the other part of the loss function requires samples under $Q^r$, using the same probability space reduces the cost of simulation. On the other hand, the probability measure $Q^r$ represents the risk-neutral measure in financial mathematics \citep{KaratzasShreve1998}, thus is a natural choice. Finally, choosing $Q^r$ is beneficial for potential downstream tasks. For example, $\{e^{-rt} \hat{X}_t \}_{0 \leq t \leq T}$ is a martingale under $Q^{r}$, then we can numerically compute the optimal wealth process with the learned optimal terminal wealth and the conditional expectations. Therefore, the loss function is defined by Equation (\ref{loss}).

Once we sample from $W_T$ and $B$ under $Q^{b_1}$, we need to compute the nested expectations  (\ref{1}) and (\ref{2}) since (\ref{3}) is easy to compute. Essentially, computations of (\ref{1}) and (\ref{2}) are the same, except the outermost distribution in (\ref{1}) is deterministic, where the outermost distribution in (\ref{2}) depends on the sampling of $W_T$ and $B$ under $Q^{b_1}$. We can view the whole loss function as a nested expectation with layers in $x \mapsto x^2$ and $x \mapsto \log(x)$ and then apply the method in \citet{YG23}, but then the regularity conditions are hard to check. Thus, we only consider the rMLMC method for the nonlinear transform $\phi'$. For the rest of the estimator, we use the plug-in method.

\begin{algorithm}
\caption{Loss Estimator with rMLMC}\label{first}
\textbf{Input:} Functions $\phi$ and $u$, scalar parameters $(x_0, b_0, T, \mu_0, \sigma_0, r,\sigma)$, fixed parameters $\kappa$ and $\theta$, data sets $\{W_k^{(1)}\}_{1 \leq k \leq n^{(1)}}$, $\{B_k^{(1)}\}_{1 \leq k \leq n^{(1)}}$, $\{W_{i,j}^{(3)}\}_{1 \leq j \leq n^{(2)}, 1 \leq i \leq 2^{N+1}, }$, $\{B_{i,j}^{(3)}\}_{1 \leq j \leq n^{(2)}, 1 \leq i \leq 2^{N+1}}$, $\{W_{i,j,k}^{(3)}\}_{1 \leq j \leq n^{(2)}, 1 \leq i \leq 2^{N+1}, 1 \leq k \leq n^{(1)}}$, and $\{B_{i,j,k}^{(3)}\}_{1 \leq j \leq n^{(2)}, 1 \leq i \leq 2^{N+1}, 1 \leq k \leq n^{(1)}}$, sample $N$.\\
\textbf{Output:} Estimate of $\mathcal{L}(\boldsymbol{\theta}) = \mathcal{L}\left(\theta, \kappa\right)$.
\begin{itemize}
    \item \textbf{Compute the constant log term:}
 \textbf{For each $j = 1, \ldots, n^{(2)}$}
\begin{itemize}
    % \item Initialize $Z_j = 0$.
    \item For each $i = 1, \ldots, 2^{N + 1}$, define $Y_{i,j} = \left(W^{(3)}_{i,j}, B^{(3)}_{i,j}\right)^T$ and compute $X_{i,j} = u\left(h_{\theta}\left(Y_{i,j}\right)\right)$.
    \item Split the sequence of $X_{i,j}$ into odd and even indices for $i$. Compute
    $S_{j,l} = \sum_{i=1}^l X_{i,j}$, $S^O_{j,l} = \sum_{i=1}^k X^O_{i,j}$, and $S^E_{j,l} = \sum_{i=1}^l X^E_{i,j}$.
    \item Compute
    $
    \Delta_N = \phi'\left(\frac{S_{j,2^{N+1}}}{2^{N+1}}\right) - \frac{1}{2}\left(\phi'\left(\frac{S^O_{j,2^N}}{2^N}\right) + \phi'\left(\frac{S^E_{j,2^N}}{2^N}\right)\right).
    $
    \item Compute $Z_j = \frac{\Delta_N}{p(N)} + \phi'\left(\frac{S_{j, 2^{n_0}}}{2^{n_0}}\right),$
    where $p(N)$ is the probability mass function of $N$.
\end{itemize}\textbf{end for }Compute $\hat{I}_c = \log\left(\frac{1}{n^{(2)}} \sum_{j=1}^{n^{(2)}}Z_j\right).$

\item Compute $D = W^{(1)}_k + \frac{B^{(1)}_k-b_0}{\sigma}T$, $K_1 = \dfrac{T}{2 \sigma^2} + \dfrac{1}{2 \sigma_0^2},$, $K_2 = \dfrac{T b_0}{\sigma^2} + \dfrac{D}{\sigma} + \dfrac{\mu_0}{\sigma_0^2}$, and $K_3 = \dfrac{T b_0^2}{2 \sigma^2} + \dfrac{b_0 D}{\sigma} + \dfrac{\mu_0^2}{2 \sigma_0^2}$.

\item Compute $\hat{I}_{r_1} = \frac{r-b_0}{\sigma}\left(W^{(1)}_k + \frac{B^{(1)}_k-b_0}{\sigma}T\right) - \frac{(b_0-r)^2}{2\sigma^2}T + K_3 - \frac{K_2^2}{4K_1} - \log\left(\sigma_Q\right) +\log\left(\sigma_0\right).$
\item \textbf{Compute the random log term: }\textbf{For each $k = 1, \ldots, n^{(1)}$}\begin{itemize}
    \item 
 \textbf{For each $j = 1, \ldots, n^{(2)}$}
\begin{itemize}
    % \item Initialize $Z_j = 0$.
    \item For each $i = 1, \ldots, 2^{N + 1}$, define $Y_{i,j,k} = \left(W^{(3)}_{i,j,k}, B^{(3)}_{i,j,k}\right)^T$ and compute $X_{i,j,k} = u\left(h_{\theta}\left(Y_{i,j,k}\right)\right)$.
    \item Split the sequence of $X_{i,j,k}$ into odd and even indices for $i$. Compute
    $S_{j,l,k} = \sum_{i=1}^l X_{i,j,k}$, $S^O_{j,l,k} = \sum_{i=1}^k X^O_{i,j,k}$, and $S^E_{j,l,k} = \sum_{i=1}^l X^E_{i,j,k}$.
    \item Compute
    $
    \Delta_{N,k} = \phi'\left(\frac{S_{j,2^{N+1},k}}{2^{N+1}}\right) - \frac{1}{2}\left(\phi'\left(\frac{S^O_{j,2^N,k}}{2^N}\right) + \phi'\left(\frac{S^E_{j,2^N,k}}{2^N}\right)\right)
    $ and $Z_{j,k} = \frac{\Delta_{N,k}}{p(N)} + \phi'\left(\frac{S_{j,2^{n_0},k}}{2^{n_0}}\right).$
\end{itemize}\textbf{end for }Compute $\hat{I}_{r_2} = \log\left(\frac{1}{n^{(2)}} \sum_{j=1}^{n^{(2)}}Z_{j,k}\right)$. \textbf{end for }
\end{itemize}
\item Compute $\mathcal{L}_2 = \left(\frac{1}{n^{(1)}}\sum_{k=1}^{n^{(1)}}h_{\theta}\left(W^{(1)}_k, B^{(1)}_k\right) - x_0e^{rT}\right)^2$.
\end{itemize}
\textbf{Return:} 
% \begin{align*}
%     &\mathcal{L}\left(\boldsymbol{\theta}\right) =\\
%     & = \frac{1}{n^{(1)}}\sum_{k=1}^{n^{(1)}}\left(\kappa - \log\left(u'\left(h_{\theta}\left(W^{(1)}_k, B^{(1)}_k\right)\right)\right) + \hat{I}_{r_1} - \hat{I}_{r_2} + \hat{I}_{c}\right)^2 + .
% \end{align*}
$\mathcal{L}\left(\boldsymbol{\theta}\right) = \frac{1}{n^{(1)}}\sum_{k=1}^{n^{(1)}}\left(\kappa - \log\left(u'\left(h_{\theta}\left(W^{(1)}_k, B^{(1)}_k\right)\right)\right) + \hat{I}_{r_1} - \hat{I}_{r_2} + \hat{I}_{c}\right)^2 + \mathcal{L}_2.$ 
\end{algorithm}

\begin{algorithm}
\caption{rMLMC DR Policy Learning Step for the General Case}\label{BPG}
\textbf{Input:} Functions $\phi$ and $u$, scalar parameters $(x_0, b_0, T, \mu_0, \sigma_0, r,\sigma)$, step-size sequence $\{\alpha_k\}_{k \in \mathbb{Z}_{\geq 0}}$, initilizations $\boldsymbol{\theta} = (\theta, \kappa)^T$ and $k = 0$.\\
\textbf{Output:} DR optimal terminal wealth $X_T$.\\ 
\textbf{Samples for constant log term:} sample $\tilde{N} \sim \text{Geo}(R)$ and compute $N = \tilde{N} + n_0$; sample $n^{(2)}$ i.i.d. values $\{B^{(2)}_j\}_{1 \leq j \leq n^{(2)}}$ from $\mathcal{N}(\mu_0, \sigma_0^2)$.  \textbf{For each $j = 1, \ldots, n^{(2)}$}\begin{itemize}
    % \item Initialize $Z_j = 0$.
    \item Sample $2^{N + 1}$ i.i.d. samples $\{B^{(3)}_{i,j}\}_{1 \leq i \leq 2^{N + 1}}$ from $\mathcal{N}(\mu_0, \sigma_0^2)$. 
    \item Sample $2^{N + 1}$ i.i.d. samples $\{N^{(3)}_{i,j}\}_{1 \leq i \leq 2^{N + 1}}$ from $\mathcal{N}(0,T)$. 
     \item For each $i = 1, \ldots, 2^{N + 1}$, compute $W^{(3)}_{i,j} = N^{(3)}_{i,j} - \frac{B^{(3)}_{i,j}-B^{(2)}_j}{\sigma}T$.
    \end{itemize}
Sample $n^{(1)}$ i.i.d. values $\{B^{(1)}_k\}_{1 \leq k \leq n^{(1)}}$ from $\mathcal{N}(\mu_0, \sigma_0^2)$ and $n^{(1)}$ i.i.d. values $\{N^{(1)}_k\}_{1 \leq k \leq n^{(1)}}$ from $\mathcal{N}(0, T)$. For each $k = 1, \ldots, n^{(1)}$, compute $W^{(1)}_k = N^{(1)}_k - \frac{B^{(1)}_k-r}{\sigma}T$. Compute $\sigma_Q^2 = \left( \dfrac{T}{\sigma^2} + \dfrac{1}{\sigma_0^2} \right )^{-1}$.\\
\textbf{Samples for random log term:}\textbf{ For each $k = 1, \ldots, n^{(1)}$}\begin{itemize}
    \item Compute $\mu_{Q} = \sigma_Q^2 \left( \dfrac{W^{(1)}_k}{\sigma} + \dfrac{B^{(1)}_k T}{\sigma^2} + \dfrac{\mu_0}{\sigma_0^2} \right )$. Sample $n^{(2)}$ i.i.d. values $\{B^{(2)}_{j,k}\}_{1 \leq j \leq n^{(2)}}$ from $\mathcal{N}(\mu_Q, \sigma_Q^2)$.
 \textbf{For each $j = 1, \ldots, n^{(2)}$}
\begin{itemize}
    % \item Initialize $Z_j = 0$.
    \item Sample $2^{N + 1}$ i.i.d. samples $\{B^{(3)}_{i,j,k}\}_{1 \leq i \leq 2^{N + 1}}$ from $\mathcal{N}(\mu_0, \sigma_0^2)$. 
    \item Sample $2^{N + 1}$ i.i.d. samples $\{N^{(3)}_{i,j,k}\}_{1 \leq i \leq 2^{N + 1}}$ from $\mathcal{N}(0,T)$. 
    \item For each $i = 1, \ldots, 2^{N + 1}$, compute $W^{(3)}_{i,j,k} = N^{(3)}_{i,j,k} - \frac{B^{(3)}_{i,j,k}-B^{(2)}_{j,k}}{\sigma}T$.
    \end{itemize}
    \end{itemize}
\textbf{repeat}
\begin{itemize}
\item Compute $\mathcal{L}\left(\boldsymbol{\theta}\right)$ by Algorithm \ref{first} and the above samples.
\item Update $\boldsymbol{\theta} = \boldsymbol{\theta} - \alpha_k \nabla_{\boldsymbol{\theta}} \mathcal{L}\left(\boldsymbol{\theta}\right)$, where the gradient is computed by back propagation; update $k = k+1$.
\end{itemize}
\textbf{until} $\boldsymbol{\theta} = (\theta,\kappa)^T$ converges to $\boldsymbol{\theta}^* = (\theta^*,\kappa^*)^T$.\\
\textbf{Return }$X_T = h_{\theta^*}(W_T,B)$.
\end{algorithm}

\subsection{Deep Learning Method to Learn Optimal Policy for the General Case}\label{algo4}
Suppose we have numerically computed the approximation of the optimal terminal wealth $X^*_T \approx h_{\theta^*}(W_T,B)$, then for the policy evaluation step (Section \ref{sec:duality}), the simulation is simpler than the KL case (since we can directly sample unbiased terminal wealth directly without simulating an SDE).
However, in order to use the rMLMC estimator for $Q_{\text{DRBC}}$, we still need to draw $\{b_1, b_2, \ldots, b_n\}$ from the distribution of $\mu$ and generate an unbiased estimator of $$E^{P^{b_i}}\left[u(X_T)\right] = E^{Q^{b_i}}\left[u(X_T)\right] = E^{Q^{b_i}}\left[u(h_{\theta^*}(W_T,B))\right].$$
Thus, it is better to design a loss function for each $b_i$ rather than choose a "uniformly" best $Q^{b^*}$ for the loss function.

    Suppose the multipliers of the alternative optimization algorithm converge finally, then we get realizations of $\hat{V}(x_0)$ in the worst case. Since the optimal value function of the distributionally robust problem is just plugging the worst case probability $\mu^*$ into the original form of the solution (Theorem \ref{Karatzassolution}), thus we can solve the optimization problem $\inf_{\theta}\mathcal{L}(\theta)$ to derive $\mu^*$ if we parametrize the density by a neural network: 
    \begin{align*}
        \mathcal{L}(\theta)
      &= E^P\left[\left\Vert \hat{V}(x_0) - V(x_0, f_{\theta}) \right\Vert_2^2\right],
    \end{align*}
    where $\theta$ represents the parameter of the neural network $f_{\theta}$ for the approximation of $\rho_{\mu^{*}}$ and $V(x_0, \mu) = \frac{\left(x_0e^{rT}\right)^{\alpha}}{\alpha}\left(\int_{\mathbb{R}}\left(F_{\mu}(T,z)\right)^{\frac{1}{1-\alpha}}\varphi_T(z)dz\right)^{1-\alpha}$ for a fixed distribution $\mu \in \mathcal{P}(\mathbb{R})$. To simplify the trackability issue, we may parametrize $\rho_{\mu^{*}}$ as exponential family or Gaussian mixtures \citep{goodfellow2016deep}.  Finally, to get the DRBC optimal control with real observations of stock prices, we plug in Theorem \ref{Karatzassolution} with the learned worst-case probability and the observations.

\section{Additional Experiments}\label{add}

%%%%%%%%%%%%%%%%%%%%%%%%%%%%%%%%%%%%%%%%%%%%%%%%%%%%%%%%%%%%

\subsection{The Finite Prior Case with KL Uncertainty Set}\label{fulldroexp}
In this section, we use Sharpe ratio and value function (expected utility) for Problem (\ref{eq:merton-bayes}) as evaluation metrics to compare DRBC method with different baselines in different settings. In Setting 1, market paths are generated from Equation (\ref{eq:merton-stock}) multiple times with a groundtruth distribution of drift. We choose an incorrect prior for both Bayesian and DRBC methods. We compare the performance of the Bayesian approach with the incorrect prior (BIP), the correct prior (BCP) (using grondtruth) to DRBC method and report the results in Table \ref{result6.4}. The results indicate the effectiveness of DRBC over prior misspecification.

Setting 2 gives comparisons between DRBC and DRC methods under another market setting where drift $B$ in Equation (\ref{eq:merton-stock}) degenerates to a single point. Again, we choose an incorrect prior for both DRBC and DRC to compute the optimal policies, and report evaluation metrics of them together with the BCP in this case (BCPD) in Table \ref{result6.4-2}. The results clearly show that DRBC reduces the overpessimism and is relatively stable in terms of $\delta$ compared with DRC.

\begin{table}[t]
\caption{Setting 1: comparison of average Sharpe ratios and expected utilities for Bayesian with correct prior ($\text{BCP}$), Bayesian with incorrect prior ($\text{BIP}$), and DRBC ( $\delta = 10^{-3}/10^{-2}$).}
\label{result6.4}
\vskip 0.15in
\begin{center}
\begin{small}
\begin{sc}
\begin{tabular}{lccr}%{p{2.65cm}p{2cm}p{2.7cm}}
\toprule
Method & Sharpe ratio & Expected Utility\\
\midrule
$\text{BIP}$ & 0.68057  & $3.41674$\\% \pm 0.63021$ \\
$\text{BCP}$ & 1.05472 & $4.37966$\\% \pm 1.00335$  \\
% *DRBC ($\delta = 0.00001$)  & 2.9274 $\pm$ 0.9729\\
% *DRBC ($\delta = 0.0001$)  & 3.0381 \\
DRBC ($10^{-3}$)  & 0.93987 & $3.89788$\\% \pm 0.80991$\\
DRBC ($10^{-2}$)  & 0.93989 & $3.89793$\\% \pm 0.80993$\\
% DRBC ($\delta = 0.01$)  & 1.0964 \\
\bottomrule
\end{tabular}
\end{sc}
\end{small}
\end{center}
\vskip -0.1in
\end{table}

\begin{table}[t]
\caption{Setting 2: comparison of average Sharpe ratios and expected utilities for Bayesian with correct degenerate prior ($\text{BCPD}$); DRC and DRBC ($\delta = 10^{-3}/10^{-2}$).}
\label{result6.4-2}
\vskip 0.15in
\begin{center}
\begin{small}
\begin{sc}
\begin{tabular}{lccr}%{p{2.65cm}p{2cm}p{2.7cm}}
\toprule
Method & Sharpe ratio & Expected Utility\\
\midrule
$\text{BCPD}$ & 2.44595  & $6.07575$\\% \pm 0.72128$ \\
$\text{DRC}$($10^{-3}$) & 1.04900 & $3.46948$\\% \pm 0.41257$  \\
$\text{DRC}$($10^{-2}$) & 0.91977 & $3.38634 $\\%\pm 0.40271$  \\

DRBC ($10^{-3}$)  & 1.48139 & $4.43138$\\% \pm 0.70654$\\
DRBC ($10^{-2}$)  & 1.48136 & $4.43145$\\% \pm 0.70654$\\
% DRBC ($\delta = 0.01$)  & 1.0964 \\
\bottomrule
\end{tabular}
\end{sc}
\end{small}
\end{center}
\vskip -0.1in
\end{table}

\subsection{High Dimensional Synthetic Experiments with KL Uncertainty Set}\label{highdimres}

In this section, we scale the dimension of SDE up from one to one hundred to show the performance of our method with Sharpe Ratio and also show the necessity practicality of our method since it is designed beyond low dimensional cases. We modify Algorithm \ref{DROpo} to high dimensional formulas to get optimal fractions. The results confirm DRBC can reduce over-pessimism, and certify our method in high dimensional settings. Details of implementation of DRBC and DRC are in section \ref{highdimdetail}.

\begin{table}[H]
\caption{Comparison of average Sharpe ratios for Bayesian with correct degenerate prior ($\text{BCPD}$); DRC and DRBC for 100 assets.}
\label{result6.4-3}
\vskip 0.15in
\begin{center}
\begin{small}
\begin{sc}
\begin{tabular}{cc}%{p{2.65cm}p{2cm}p{2.7cm}}
\toprule
Method & Sharpe ratio\\
\midrule
$\text{BCPD}$ & 0.954 \\% \pm 0.72128$ \\
$\text{DRC}$ & 0.397\\% \pm 0.41257$  \\
$\text{DRBC}$ & 0.591\\%\pm 0.40271$  \\

\bottomrule
\end{tabular}
\end{sc}
\end{small}
\end{center}
\vskip -0.1in
\end{table}

\section{Experiment Details}\label{ED}
\subsection{Plug-in Method for the LQ Experiment}
\label{app:lq-plugin}

We implement a certainty-equivalent \emph{plug-in} baseline for the DRBC--LQ experiment. The plug-in Merton method used later will be similar so we omit it.
The plug-in method first estimates the unknown drift parameter $\theta$ from observed state--control trajectories,
then solves the classical LQ regulator with the estimated drift matrix $A(\widehat\theta)$ and applies the
resulting linear feedback control.

Recall the controlled diffusion
\begin{equation}
\label{eq:lq-plugin-sde}
dX_t=\big(A(\theta)X_t+Gu_t\big)\,dt+\Sigma\,dW_t,\qquad
A(\theta)=A_0+\sum_{i=1}^{m}\theta_iA_i,
\end{equation}
where $X_t\in\RR^d$ is fully observed, $u_t\in\RR^k$ is the applied control, and $\Sigma$ is known and
nonsingular.

Let $\Delta>0$ be the simulation time step and $t_k=k\Delta$ for $k=0,\dots,N$ with $T=N\Delta$.
Define the corrected increments
\begin{equation}
\label{eq:lq-plugin-increments}
\Delta Y_k := X_{t_{k+1}}-X_{t_k}-Gu_{t_k}\Delta.
\end{equation}
Under \eqref{eq:lq-plugin-sde} and an Euler--Maruyama discretization,
\begin{equation}
\label{eq:lq-plugin-regression}
\Delta Y_k \approx A(\theta)X_{t_k}\Delta + \Sigma\sqrt{\Delta}\,\varepsilon_k,
\qquad \varepsilon_k\sim\mathcal N(0,I_d)\ \text{i.i.d.}
\end{equation}
Because $\Sigma$ is known, the Gaussian likelihood implied by \eqref{eq:lq-plugin-regression} yields an MLE that
coincides with a generalized least-squares (GLS) regression.
Let
\[
r_k := \frac{\Delta Y_k}{\Delta}-A_0X_{t_k}\in\RR^d,
\qquad
\Phi_k := \big[A_1X_{t_k},\,A_2X_{t_k},\,\ldots,\,A_mX_{t_k}\big]\in\RR^{d\times m},
\]
so that $r_k \approx \Phi_k\theta + \text{noise}$.
Whitening by $\Sigma^{-1}$ gives
\[
\tilde r_k := \Sigma^{-1}r_k,\qquad \tilde\Phi_k := \Sigma^{-1}\Phi_k,
\]
and the MLE/GLS estimator is
\begin{equation}
\label{eq:lq-plugin-mle}
\widehat\theta
=
\arg\min_{\theta\in\RR^m}\sum_{k=0}^{N-1}\big\|\tilde r_k-\tilde\Phi_k\theta\big\|_2^2
=
\Big(\sum_{k=0}^{N-1}\tilde\Phi_k^\top\tilde\Phi_k\Big)^{-1}
\Big(\sum_{k=0}^{N-1}\tilde\Phi_k^\top\tilde r_k\Big).
\end{equation}
In practice, when $\sum_k\tilde\Phi_k^\top\tilde\Phi_k$ is ill-conditioned we use a small ridge regularization,
i.e., replace it by $\sum_k\tilde\Phi_k^\top\tilde\Phi_k+\lambda_{\rm rid}I_m$ with $\lambda_{\rm rid}>0$ fixed.

Given $\widehat\theta$, we form the estimated drift matrix
\[
\widehat A := A(\widehat\theta)=A_0+\sum_{i=1}^m \widehat\theta_iA_i.
\]
The plug-in controller then solves the standard finite-horizon continuous-time LQ regulator for the system
\[
dX_t=(\widehat A X_t+Gu_t)\,dt+\Sigma\,dW_t,
\]
with quadratic running and terminal costs as in \eqref{eq:lq-J}.
Let $P(t)\in\RR^{d\times d}$ denote the solution to the Riccati differential equation
\begin{equation}
\label{eq:lq-plugin-riccati}
-\dot P(t)=Q + \widehat A^\top P(t) + P(t)\widehat A - P(t)G R^{-1}G^\top P(t),
\qquad P(T)=Q_T,
\end{equation}
and define the feedback gain
\begin{equation}
\label{eq:lq-plugin-gain}
K_{\rm PI}(t):=R^{-1}G^\top P(t).
\end{equation}
The resulting certainty-equivalent (plug-in) policy is the linear state feedback
\begin{equation}
\label{eq:lq-plugin-policy}
u_t^{\rm PI}=-K_{\rm PI}(t)X_t.
\end{equation}
We evaluate $u^{\rm PI}$ under the same misspecified/ground-truth settings used for DRBC--LQ and report its
performance as a baseline in Section~\ref{sec:exp-lq}.

\subsection{Details of DRBC LQ Experiments}
\label{app:lq-details}
We simulate \eqref{eq:lq-classical} using Euler--Maruyama on a uniform grid with horizon $T=2.0,\Delta t = 0.02,$ and $K=T/\Delta t = 100 \text{ steps.}$
The dimensions are $d_x=10$,$d_u=5,$ and $p=\dim(\theta)=10.$
We use the affine parameterization $A(\theta) \;=\; A_0 + \sum_{j=1}^p \theta_j A_j.$
with the tridiagonal base drift $A_0\in\mathbb{R}^{d_x\times d_x}$: $(A_0)_{ii}=-a_{\text{diag}}, (A_0)_{i,i+1}=a_{\text{upper}},(A_0)_{i+1,i}=a_{\text{lower}}, i=1,\dots,d_x-1,
$
with $a_{\text{diag}}=0.6, a_{\text{upper}}=0.15, a_{\text{lower}}=-0.1.$
For each $j\in\{1,\dots,p\}$, we construct a \emph{dense} (full) matrix $A_j$ as $A_j \;=\; \texttt{Aj\_diag}\cdot D_j \;+\; \texttt{Aj\_upper}\cdot O_j,$
where $D_j$ is a random diagonal matrix and $O_j$ is a random off-diagonal matrix (with zero diagonal). Both components are normalized to have Frobenius norm one ($\|D_j\|_F=\|O_j\|_F=1$) to make the scale knobs interpretable across dimensions. We set $\texttt{Aj\_diag}=0.2$ and $\texttt{Aj\_upper}=0.05,$
and fix the random seed for the $A_j$ construction to \texttt{Aj\_seed}=12345 for reproducibility.
We take $G\in\mathbb{R}^{d_x\times d_u}$ to actuate the first $d_u$ coordinates:
\[
G=\begin{bmatrix} I_{d_u}\\ 0\end{bmatrix}.
\]
The diffusion matrix is diagonal with $\Sigma = \sigma_{\text{scale}}I_{d_x}$ and $\sigma_{\text{scale}}=0.7.$
The quadratic costs are diagonal: $Q = q_{\text{scale}} I_{d_x},$ $Q_f=qf_{\text{scale}} I_{d_x},$ and $R=r_{\text{scale}} I_{d_u}$ 
with $q_{\text{scale}}=3.0,$ $qf_{\text{scale}}=3.0,$ $r_{\text{scale}}=1.0.$ 
In each run we sample the ground-truth drift parameter from a centered Gaussian prior $\theta^\star \sim \mu^\star = \mathcal N(0,\mu_{\star}^2 I_p),$ with $\mu_{\star}=0.5.$
To test robustness to prior misspecification, the DRBC training uses a different nominal prior $\theta \sim \mu = \mathcal N(0,\mu_{\text{nom}}^2 I_p)$ with $\mu_{\text{nom}}=1.0,$
so $\mu\neq\mu^\star$ (heavier tails under $\mu$).
For evaluation we draw $x_0=1.0$, $X_0 \sim \mathcal N(0, x_{0}^2 I_{d_x}),$
and estimate each method's performance under $\theta^\star$ via Monte Carlo with $B_{\text{eval}}=512$
independent rollouts per run. Within each run, we use \emph{common random numbers} (the same Brownian noise and the same $X_0$ draws) across methods to reduce Monte Carlo variance in pairwise differences. We report mean $\pm$ standard deviation across $N_{\text{runs}}=100$
independent runs (seed $0$).
The plug-in method proceeds as described in Appendix~\ref{app:lq-plugin}. Concretely, in each run we generate a single identification trajectory under $\theta^\star$ using a random linear feedback exploration policy $u_t = -K_{\text{rand}} X_t,$ $(K_{\text{rand}})_{ab}\stackrel{i.i.d.}{\sim}\mathcal N(0,\texttt{scale}^2),$ $\texttt{scale}=0.3,$
and then compute a single-trajectory GLS (whitened least-squares) estimate $\widehat\theta$ together with an information/precision proxy $S^b$ (denoted \texttt{S\_prec} in code). Given $\widehat\theta$, we solve the finite-horizon Riccati equation for $A(\widehat\theta)$ and deploy the resulting time-varying LQ feedback controller (certainty-equivalent control).

We also report an oracle controller that knows $\theta^\star$ and solves the LQ problem with drift $A(\theta^\star)$ exactly via the same Riccati solver. We use the gap to the oracle as an interpretable performance measure.

We implement the policy-learning step of Algorithm~2 using a fixed robustness radius $\delta\in\{0.01,0.05,0.1\}$ and the simplification $\lambda \;=\; \frac{C}{\sqrt{\delta}},\qquad C>0,$
with default $C=\texttt{C\_lam}=1.0 (\,\lambda\in\{10,\;4.472,\;3.162\}\text{ for the three }\delta\text{s})$.
This is consistent with duality result in \citet{faury2020distributionally}.
DRBC uses the same identification data as plug-in to compute $(\widehat\theta,S^b)$, and then \emph{freezes} these belief features during policy learning.
We parameterize a time-varying linear feedback controller $u_k = -K_\psi(k;\widehat\theta,S^b)\,x_k,$ for $k=0,\dots,K-1,$
where $K_\psi(\cdot)$ is produced by a small neural network that takes as input a normalized time feature $t_k=k/(K-1)$ together with the concatenated belief features $(\widehat\theta,\mathrm{vec}(S^b/\bar s))$ (with $S^b$ normalized by its mean absolute entry $\bar s$ for scale robustness). The network is a 3-layer MLP with two hidden layers: $\texttt{Linear} \rightarrow \tanh \rightarrow \texttt{Linear} \rightarrow \tanh \rightarrow \texttt{Linear},$
hidden width \texttt{hidden}=128, and outputs a matrix $K_\psi(k)\in\mathbb{R}^{d_u\times d_x}$. We apply an output $\tanh$ nonlinearity and scale by \texttt{K\_scale}=1.0 to control gain magnitude.

At each gradient step we sample $\theta^{(1)},\dots,\theta^{(N_\theta)}\sim \mu$ with $N_\theta=\texttt{N\_theta}=32$. For each $\theta^{(i)}$ we simulate \texttt{B\_traj}=64 rollouts (with freshly sampled $X_0$ and Brownian noise) and average their costs to estimate $Z_i \approx J(u_\psi;\theta^{(i)})$. We then maximize the fixed-$\lambda$ KL-dual objective $\text{obj}(\psi) \;=\; -\lambda\delta \;-\; \lambda \log\!\left(\frac{1}{N_\theta}\sum_{i=1}^{N_\theta} \exp\!\left(\frac{Z_i}{\lambda}\right)\right),$
implemented stably via a \texttt{logsumexp} computation. We optimize \texttt{loss}$(\psi)=-\text{obj}(\psi)$ by Adam with learning rate \texttt{eta}=0.01 for \texttt{Sin}=200 steps. We clip the control actions elementwise to $[-u_{\max},u_{\max}]$ with \texttt{u\_clip}=50 and clip gradients to norm $10$.
% With the default settings above, running the full three-way comparison (plug-in, DRBC for three $\delta$ values, and oracle) for $N_{\text{runs}}=100$ and $B_{\text{eval}}=512$ took approximately 3.5 hours
% on a single CUDA-enabled Colab GPU. 
For completeness, Table~\ref{tab:lq_utility_appendix} reports the raw achieved utility values corresponding
to the gap-to-oracle results in Table~\ref{tab:lq_main}.

\begin{table}[t]
\centering
\caption{Synthetic LQ experiment: achieved utility (mean $\pm$ std.\ over 100 runs).}
\label{tab:lq_utility_appendix}
\small
\setlength{\tabcolsep}{4pt}
\begin{tabular}{lccc}
\toprule
Method & $\delta=0.01$ & $\delta=0.05$ & $\delta=0.10$ \\
\midrule
Plug-in (MLE $\widehat\theta$) 
& $-43.7\pm13.0$ & $-43.7\pm13.0$ & $-43.7\pm13.0$ \\
DRBC 
& $-40.9\pm1.3$ & $-41.3\pm1.3$ & $-41.4\pm1.3$ \\
Oracle (knows $\theta^\star$) 
& $-40.1\pm1.3$ & $-40.1\pm1.3$ & $-40.1\pm1.3$ \\
\bottomrule
\end{tabular}
\vspace{-0.05in}
\end{table}

Empirically, the plug-in baseline has much larger across-run variance than DRBC and the oracle. This behavior is expected in our setting for two reasons:
(i) \emph{single-trajectory identification}: $\widehat\theta$ is estimated from a single finite-horizon trajectory ($T=2$, $K=100$ steps), so occasional unfavorable noise realizations lead to large estimation errors;
(ii) \emph{sensitivity of certainty-equivalent control}: the Riccati-based controller computed from $A(\widehat\theta)$ can be overly aggressive or poorly matched when $A(\widehat\theta)$ deviates from $A(\theta^\star)$, causing transient state growth and rare but catastrophic cost outliers. These heavy-tailed failures inflate the sample standard deviation of plug-in performance. In contrast, DRBC optimizes an entropic (risk-sensitive) objective under the misspecified prior and uses action clipping, which discourages policies that incur large costs under plausible parameter draws, leading to substantially more stable performance across runs.

\subsection{Solving DRBC Merton with Finite Support Prior}\label{toy}
This section introduces the policy learning method that utilizes the swap structure for the Merton problem.
For simplicity, we focus on a commonly used example when the prior distribution takes value on a finite set of points $\left\{b_1, \ldots, b_d\right\}$ with probability mass function $P\left(B = b_i\right) = p_i \in (0,1)$, for $i = 1, 2, \ldots, d,$ where $d \geq 1$. We will make use of the swap argument to faciliate the computation here. For the fixed prior $\mu$, the optimal value function becomes 
$\tilde{V}(\boldsymbol{p}) := \frac{\left(x_0e^{rT}\right)^{\alpha}}{\alpha}\left(\int_{\mathbb{R}}\left(\tilde{F}_{\boldsymbol{p}}(T,z)\right)^{\frac{1}{1-\alpha}}\varphi_T(z)dz\right)^{1-\alpha},$
where $\tilde{F}_{\boldsymbol{p}}(t,z) = \sum_{i=1}^dp_iL_t(b_i,z)$. If we want to find the optimal probability measure $\tilde{Q}^*(\lambda)$, then it suffices to solve the convex problem (justified by a use of Sion's minimax theorem and the closed form solution from \citet{KZ98})
\begin{align}\label{toyopti}
&\inf_{\substack{\boldsymbol{q} = (q_1, \ldots, q_d)\\\sum_{i=1}^d q_i = 1,\text{and }q_i \geq 0}}V(\boldsymbol{q})=
    \inf_{\substack{\boldsymbol{q} = (q_1, \ldots, q_d)\\\sum_{i=1}^d q_i = 1,\text{and }q_i \geq 0}}\tilde{V}(\boldsymbol{q}) + \lambda\sum_{i=1}^d q_i \log \frac{q_i}{p_i}.
\end{align}
 % The precise algorithm for updating $\pi$ is given in Algorithm \ref{DROpo} in Appendix \ref{algo}. We remark that the DRBC learning steps are done via simulated (training) samples $\{S_t\}_{t \in [0,T]}$ before we observe the real market data ($\{\tilde{S}_t\}_{t \in [0,T]}$). In other words, we derive optimal policies of the form in Theorem \ref{Karatzassolution} with a worst-case probability, and $\hat{\pi}_{\text{DRBCKL}}$ is computed by this worst-case probability and $\{\tilde{S}_t\}_{t \in [0,T]}$. 
Note that the value function can be seen as an expectation with respect to the Gaussian distribution $\mathcal{N}(0,T)$, thus it can be computed via simulation methods (plain Monte Carlo).

On the other hand, the gradient information of $\tilde{V}(\boldsymbol{p})$ is hard to compute. Thus, we will focus on the zero-order methods \citep{nesterov2017random, shamir2017optimal, duchi2015optimal}. We will mimic first-order optimization strategies by replacing the
gradient of the objective function with an approximation built through finite differences \citep{leveque2007finite}: for a fixed $i = 1, 2, \ldots, d,$ let $h >0$ small, then $$\frac{\partial \tilde{V}(\boldsymbol{p})}{\partial p_i} \approx \frac{1}{2h}\left[\tilde{V}(\boldsymbol{p}+h\boldsymbol{e}_i) - \tilde{V}(\boldsymbol{p}-h\boldsymbol{e}_i)\right],$$
where $\boldsymbol{e}_i$ is the unit vector in the $i$th coordinate direction. Note that here we use the central difference method, which has the $\mathcal{O}(h^2)$ rate of convergence, instead of the backward or forward difference method, which both have the $\mathcal{O}(h)$ rate of convergence. The exact steps to solve Problem (\ref{toyopti}) are summarized in Algorithm \ref{DROpo}.

% Then we can use the empirical risk minimization technique \citep{Bach} to find the optimal $\boldsymbol{q}^*$ (Algorithm \ref{DROpo}).
\begin{algorithm}
\caption{rMLMC DRBC KL Policy Learning Step}\label{DROpo}
\textbf{Input:} Prior distribution $\boldsymbol{p}$, step-size sequence $\{\alpha_k\}_{k \in \mathbb{Z}_{\geq 0}}$, parameter $h >0$ and $\lambda >0$, observations $\{S_t\}_{t \in [0,T]}$ (see remark in Section \ref{toy} for which process to plug in),   initializations of $\boldsymbol{q}$ and $k = 0$. \\
\textbf{Output:} DRBC optimal policy $\{\pi^*_{\lambda,t}\}_{t \in [0,T]}$.\\ 
\textbf{repeat}
\begin{itemize}
\item For each $i = 1,2,\ldots, d$, approximate $\tilde{V}(\boldsymbol{q}+h\boldsymbol{e}_i)$ and $\tilde{V}(\boldsymbol{q}-h\boldsymbol{e}_i))$ by simulations. 
\item For each $i = 1,2,\ldots, d$, compute $\text{GF}_i = \frac{1}{2h}\left[\tilde{V}(\boldsymbol{q}+h\boldsymbol{e}_i) - \tilde{V}(\boldsymbol{q}-h\boldsymbol{e}_i)\right] + \lambda \left(\log\left(\frac{q_i}{p_i}\right)+1\right)$. 
\item Update $\boldsymbol{q} = \boldsymbol{q} - \alpha_k \text{GF}$, update $\boldsymbol{q} = \text{softmax}(\boldsymbol{q})$, and update $k = k+1$.
\end{itemize}
\textbf{until} $\boldsymbol{q}$ converges to $\boldsymbol{q}^*$.\\
\textbf{Return }$\pi^*_{\lambda,t} = \frac{\int_{\mathbb{R}}\nabla F_{\boldsymbol{q}^*}\left(T,z+Y_t\right)\left(F_{\boldsymbol{q}^*}\left(T,z+Y_t\right)\right)^{\frac{\alpha}{1-\alpha}}\varphi_{T-t}(z)dz}{(1-\alpha)\sigma \int_{\mathbb{R}}\left(F_{\boldsymbol{q}^*}\left(T,z+Y_t\right)\right)^{\frac{1}{1-\alpha}}\varphi_{T-t}(z)dz}$, for each $t \in [0,T]$, where $Y_t $ is given by Equation (\ref{YS}).
\end{algorithm}

\subsection{Details of Experiments in Section \ref{smps}}\label{KMM}
\subsubsection{Choice of Performance Measures}
From \citet{GLS22}, if \[
\Phi(x) =
\begin{cases}
\frac{x^{\gamma}}{\gamma}, & \text{if } x \geq 0, \\
\frac{-(-x)^{\gamma}}{\gamma}, & \text{if } x < 0,
\end{cases}
\quad \text{where } \gamma \in (0,1) \text{ is a fixed constant,}
\] then the optimal terminal wealth has the closed-form
\begin{align}\label{closedform}
  X^*_T
  &= \exp \Bigg( \frac{1}{\alpha} \Bigg( \frac{p}{2T} 
  \left(W_T + \frac{B-b_0}{\sigma}\right)^2 + q \left(W_T + \frac{B-b_0}{\sigma}\right) + c \Bigg) \Bigg),
\end{align}where
$$
\left\{
\begin{aligned}
    p &=  \frac{\frac{1}{1 - \alpha} + \sigma_0^2 -\sqrt{\sigma_0^4 + \frac{2 - 4\alpha}{1 - \alpha} \sigma_0^2 + \frac{1}{(1 - \alpha)^2} - \frac{4\alpha}{1 - \alpha} \gamma}}{2(\sigma_0^2 + \gamma)}, \\
    q &= \frac{\alpha(1 - p)}{(1 - \alpha)(1 - \gamma p)} \nu, \\
    c &= \alpha \left[ \log(x_0) + rT + \left( \nu^2 - \frac{\left( \nu - \frac{q}{\alpha} \right)^2}{1 - \frac{p}{\alpha}} \right) \frac{T}{2} \right].
\end{aligned}
\right.
$$
From \citet{GLS22}, for a fixed $b \in \mathbb{R}$ (more precisely in the support of the prior)
\[
E^{Q^b}[u(X^*_T)] = \frac{1}{\alpha \sqrt{1-p}} \exp \Bigg(
\frac{pT}{2(1-p)} \nu^2_b - \frac{qT}{1-p} \nu_b + \frac{q^2T}{2(1-p)} + c
\Bigg)
\]
\[
= \frac{1}{\alpha \sqrt{1-p}} \exp \Bigg(
\frac{pT}{2(1-p)} \frac{(b - b_0)^2}{\sigma^2} + \frac{qT}{1-p} \frac{b - b_0}{\sigma}
+ \frac{q^2T}{2(1-p)} + c
\Bigg).
\]
% thus 
% \begin{align*}
%     V(x_0) &= \sup_{\alpha \in \mathcal{A}(x_0)} \int_{\mathbb{R}}\Phi\left(E^{Q^b}\left[u(X_T)\right]\right)d\mu(b)\\
%     &= \int_{\mathbb{R}} \frac{1}{\alpha \sqrt{1-p}} \exp \Bigg(
% \frac{pT}{2(1-p)} \frac{(b - b_0)^2}{\sigma^2} + \frac{qT}{1-p} \frac{b - b_0}{\sigma}
% + \frac{q^2T}{2(1-p)} + c
% \Bigg) d\mu(b)\\
% &= 
% \frac{1}{\alpha \sqrt{1-p} \sqrt{1 - \frac{pT \sigma_0^2}{(1-p) \sigma^2}}}
% \exp \Bigg(
% \frac{q^2 T}{2(1-p)} + c + \frac{pT (\mu_0 - b_0)^2}{2(1-p) \sigma^2}
% + \frac{qT (\mu_0 - b_0)}{(1-p) \sigma}
% + \frac{q^2 T^2 \sigma_0^2}{2(1-p) \sigma^2 \big( \sigma^2 - pT \sigma_0^2 \big)}
% \Bigg).
% \end{align*}since $\mu \sim \mathcal{N}\left(\mu_0, \sigma_0^2\right)$.

Following the discussion in Appendix \ref{algo4}, in the alternate optimization steps for the DRBC algorithm, the essential property is how accurate that this optimal terminal wealth can be used to compute $E^{Q_{b_1}}\left[u(X_T)\right]$ for a fixed $b_1 \in \mathbb{R}$, thus we choose it to compare with the closed form solution. 
The performance measures that we taken into considerations are their equivalence properties are used in the alternative optimization procedures. 
% We could also directly test the distribution of the learned and theoretical optimal terminal wealth. For example, we may compare the Wasserstein distance between them by computing the empirical probability density functions after sampling. However, as indicated by Equation (\ref{closedform}), to sample from the theoretical optimal terminal wealth, and exponential term with $\frac{1}{\sigma^2}$ will be computed, which will lead to numerical instability with high probability. On the other hand, the theoretical value $V(x_0)$ can be computed safely, and so as $u\left(h_{\boldsymbol{\theta}}(W_T,B)\right)$ and $\int_{\mathbb{R}}\Phi\left(E^{P^b}\left[u\left(h_{\boldsymbol{\theta}}(W_T,B)\right)\right]\right)d\mu(b)$, since no exponential terms are involved.
\subsubsection{Details of DRBC Merton Problem Experiments}
\label{H.1}
We start with the hyperparameter settings for the synthetic data, then introduce our neural network settings. 
% The network is trained on 1 Nvidia A100 GPU for about one GPU hours.

Sample size is set to be 2000. Second level sample size for rMLMC method to be 100, and geometric distribution parameter $R=0.65$. To make the synthetic data close to real financial market observations, we let $\sigma=0.4$, $\sigma_0 = 2$, $r=0.05/0.1$, $b_0=0.1$, $T=1$, $\mu_0=0.1$, $b=0.1/0.3$, and $x_0=1$. For the function $\Phi'$, to get numerical stability, we use a truncation to approximate at 0.01 and 2, when $x$ is smaller than 0.001, $\Phi'$ gives a constant; when $x$ is larger than 2, $\Phi'$ is $x^{-\frac{1}{2}}$; between 0.001 and 2, $\Phi'$ is the linear interpolation of above two functions.

We use a modified multi-layer perceptron (MLP) $h_{\theta}$ to estimate $X^*_T$. The MLP has four layers in total, first layer takes 2-dim input $(W_T,B)$ and maps to 128 hidden nodes; second layer maps 128 nodes to 256 hidden nodes; third layer maps 256 nodes to 256 nodes, and final layer maps 256 nodes to one output. We use LeakyReLU \citep{maas2013rectifier} as activation function with parameter 0.01. Since our parameter settings mimic real financial data, we need to modify the output of the MLP to satisfy non-negative constraint and match the real terminal wealth distribution easier. We impose a partial linear structure on top of the MLP with constant 1 and learnable parameter $b$. The constant is from financial practices that under optimal portfolio strategy, investor earns excess return. $b$ is set here for an easier learning process and better gradient flow. We also set $\kappa$ a learnable parameter and use gradient descent alternatively for MLP parameters, $b$ and $\kappa$ in the training pipeline. 

For learning rate schedule, we adopt a Warmup-Stable-Decay \citep{wen2024understandingwarmupstabledecaylearningrates} learning rate schedule which achieves great success in large language models. We use 1000 epochs in total, with 10 epochs to linearly warmup the learning rate to 0.001, then steadily train 400 epochs, and finally decay to 0.0003 for the rest 590 epochs. For $b$, learning rate is fixed at $10^{-4}$. For $\kappa$, learning rate linearly decays from 0.01 to $10^{-4}$, then stays until training is done. We use Adam \citep{DBLP:journals/corr/KingmaB14} as the optimizer and saved models can be found at our repository. Training losses for different $b$ and $r$ values are shown in Figure \ref{fig:training_loss}. In all three hyperparameter settings, our network shows stable loss curves and achieve good performances comparing to theory results, as stated in Section \ref{smps}. 

\begin{figure}[ht]
\vskip 0.2in
\begin{center}
\centerline{\includegraphics[width=\columnwidth]{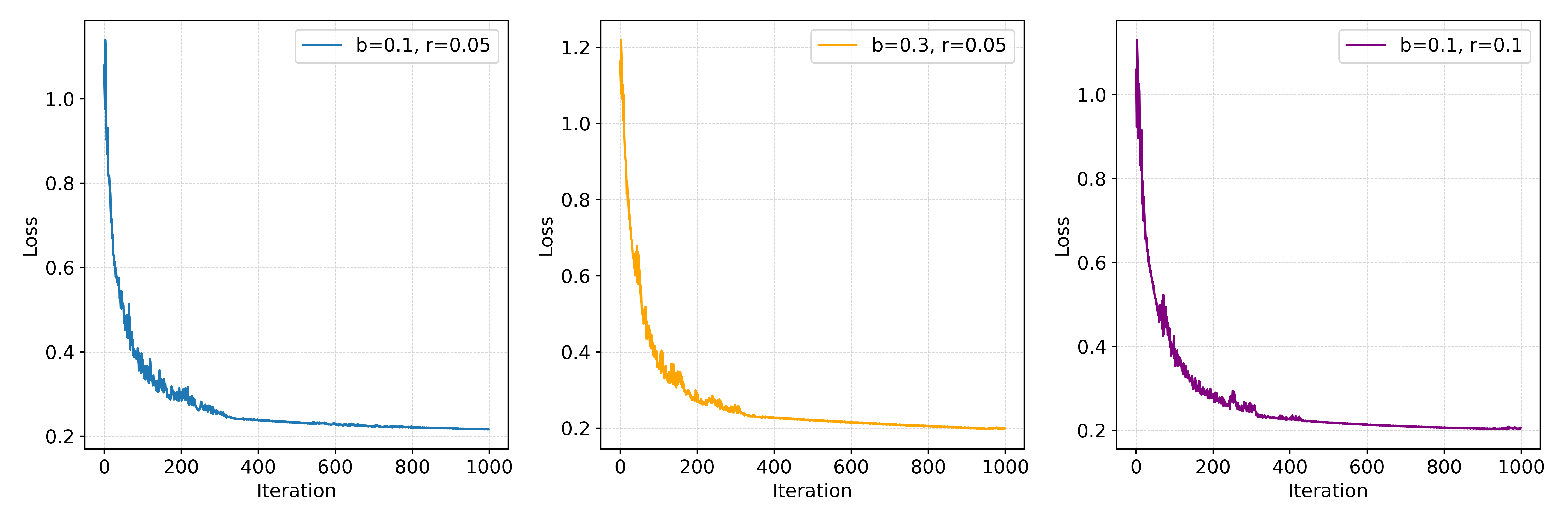}}
\caption{Training losses for different $b$ and $r$ values}
\label{fig:training_loss}
\end{center}
\vskip -0.2in
\end{figure}

We denote the trained optimal parameter as $\theta^*$. We evaluate the model by randomly generating new input pairs $(W_T,B)$ of 2000 samples with different seeds to get $E^{Q^{b}}\left[u(h_{\theta^*} (W_T,B)) \right]$. And we run such experiment 100 times to get the mean and standard deviations shown in Table \ref{result6.1}.

\subsection{Details of Rate of Convergence for Policy Evaluation}\label{CLTnum}

To make sure the validity of our comparison, we first run the policy learning step to get the $\pi$ initialization. We run policy learning step with prior values $[0.01,0.46,0.30,0.21,0.27]$ and probability mass function of prior random variable $B$ $[0.05,0.35,0.35,0.15,0.1]$. We equally divide $[0,T]$ into 1000 intervals, and initialize $\lambda=100$. When calculating the gradient of $\tilde{V}$, we set $h=10^{-6}$ and learning rate $\alpha_k=10^{-5}$ a same number across all loops $k$. We set the convergence condition to be the sum of squared errors of $\boldsymbol{q}_k-\boldsymbol{q}_{k-1}$ less than $10^{-5}$. After we get the converged $\boldsymbol{q}^*$, we get $\pi$ from Equation (\ref{YS}), as in Algorithm \ref{DROpo}.

Then we run Algorithm \ref{alg:kl-eval}. We set $n_0=3$ and $\alpha_k=0.01$. Other hyperparameters are the same as Section \ref{H.1} and prior is the same as above. Here we set the convergence condition to be $ \frac{\partial}{\partial \lambda}\mathcal{E}_{\text{KL}} < 0.3\delta$. For using inner samples to estimate $E^{P^b}\left[u(X_T) \right]$, we simulate Equation (\ref{SDE2}) with $B = b$ 100 times to get $X_T$ and then get the average. Finally we use the converged $\lambda^*$ to calculate $\mathcal{E}_{\text{KL}}$ with different sizes of $n$ 100 times to get the results in Table \ref{result6.3}.
% We document that using 4 Intel Skylake 6148, 20-core, 2.4GHz, 150W processors.
% the whole process including all three $n$ takes about 40 hours.

\subsection{Details of Experiments in Section \ref{realdata}}

The daily $S\&P$ 500 constituents data from 2015-01-01 to 2024-12-31 is from Wharton Research Data Services. For the ease of data cleaning and to avoid stock inclusion and exclusion to the index, we only keep stocks that are $S\&P$ 500 constituents in the whole time window, resulting in 326 stocks left. Interest rate data is from Federal Reserve Bank of St.Louis. We use a rolling window of 1 year for getting the interest rate $r$ and $\sigma$, and use them in DRBC and baseline methods for the following month's investment allocation.   

We remark the choice of $\delta$ is more a managerial decision rather than a scientific choice, and too large $\delta$ will not give meaningful solutions. Here we follow \citet{NianSi} to use the existing data to estimate the distributional shift, which implies a choice of $\delta$. Over the 326 stocks, the mean of $\delta \approx 0.15$, thus we choose it for the experiments. For the prior,  we choose two fixed finite priors with supports inspired by \citet{wang2020continuous}. Prior 1 is [-0.08,0.16,-0.02,0.04,0.10], with probability [0.35,0.08,0.25,0.22,0.10]. Prior 2 is [-0.05,0.15,0.00,0.05,0.10], with probability [0.45,0.05,0.25,0.15,0.1]. The histograms of sharpe ratios for two priors are shown in figure \ref{fig:hist_sr}. 
% We documented that using a single Intel Skylake 6148, 20-core, 2.4GHz, 150W processor, looping over all stocks for a single prior takes about 30 hours. 

\begin{figure}[ht]
  \centering
  % first subfigure
  \begin{subfigure}[b]{0.45\textwidth}
    \centering
    \includegraphics[width=\linewidth]{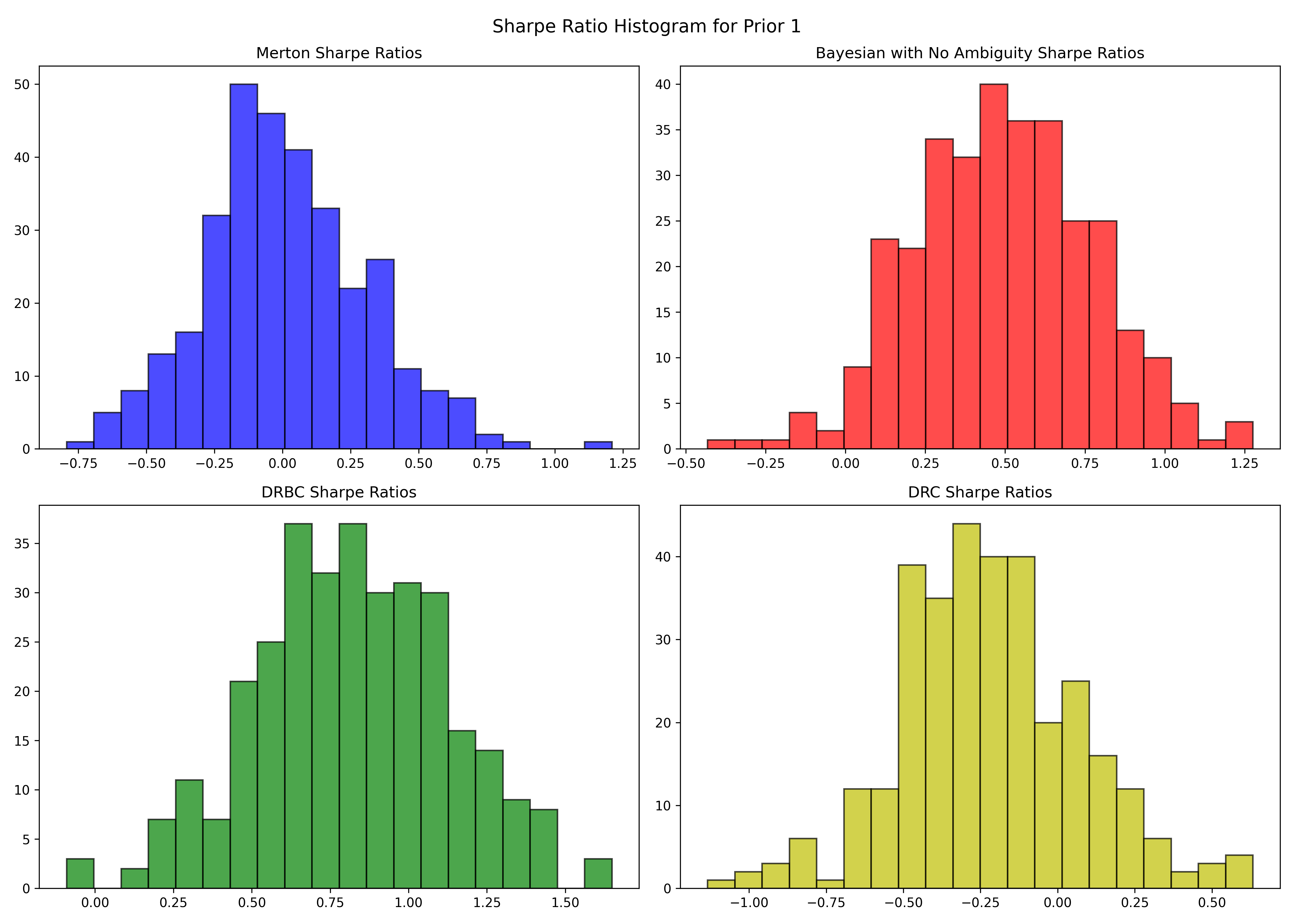}
    \caption{Prior 1}
    \label{fig:1a}
  \end{subfigure}
  \hfill
  % second subfigure
  \begin{subfigure}[b]{0.45\textwidth}
    \centering
    \includegraphics[width=\linewidth]{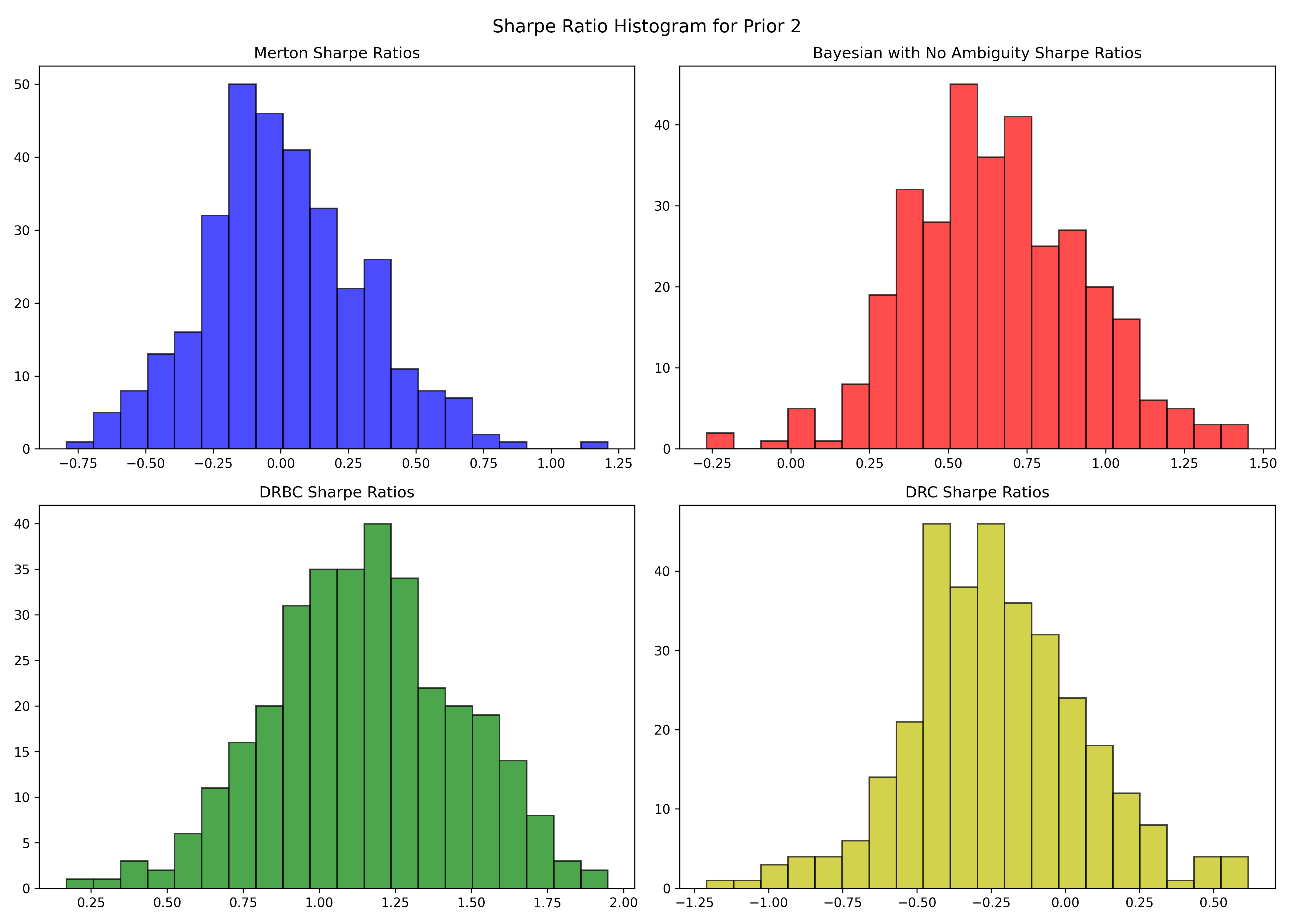}
    \caption{Prior 2}
    \label{fig:1b}
  \end{subfigure}
  \caption{Histogram of Sharpe Ratios under different priors.}
  \label{fig:hist_sr}
\end{figure}

\subsection{Details of Experiments in Section \ref{fulldroexp}}
\subsubsection{Details of Setting 1}
\label{H3.1}
For experiments in Section \ref{fulldroexp}, we run full DRBC algorithm to get $\hat{\pi}_{\text{DRBCKL}}$. Most hyperparameters are the same as in previous settings. One change here is the initialization of $\lambda$. We give a finer pre-condition of $\lambda$ based on \citet{faury2020distributionally}, which proves $\lambda=\mathcal{O}(\frac{1}{\sqrt{\delta}})$. We make our $\lambda=\frac{0.33}{\sqrt{\delta}}$. Another change here is we shift the prior distribution, which is also noted as incorrect prior in Table \ref{result6.4} to $[0.5, 0.05, 0.2, 0.15, 0.1]$, with support of the prior random variable $B$ unchanged. The correct prior distribution is $[0.05,0.5,0.1,0.15,0.2]$. The convergence condition is set to be $\left |\lambda_k - \lambda_{k-1} \right | <  10^{-3}$. 

We regard the full DRBC process as the pre-training phase to get the optimal policy. Then we run the evaluation process as follows: First, we generate market data $\{S_t\}_{t \in [0,T]}$ 200 times by Equation (\ref{eq:merton-stock}) to get 200 paths. Then use Equation (\ref{YS}) to transform $\{S_t\}_{t \in [0,T]}$ to $\{Y_t\}_{t \in [0,T]}$. Thirdly, calculate optimal fractions $\pi$ for BIP, BCP and DRBC cases using Theorem \ref{Karatzassolution} with transformed $Y_t$. Fourthly, simulate Equation (\ref{eq:merton-wealth}) or equivalently Equation (\ref{SDE2}) with $\pi_{\text{BIP}}$, $\pi_{\text{BCP}}$ and $\pi_{\text{DRBC}}$ with the same Brownian terms $W_t$ as the first step to get the wealth paths $X_t$. Finally, we use two metrics to evaluate the performances of BIP, BCP and DRBC. Since DRBC only focuses on terminal wealth, to remove other randomness, we choose the Sharpe ratio definition as in \citet{wang2020continuous}. Another metric is expected terminal utility, as defined in Equation (\ref{eq:merton-bayes}), which is also the value function. We collect terminal utility from all 200 paths to get the mean and we report them in Table \ref{result6.4}. 
% We document that using 2 Intel Skylake 6148, 20-core, 2.4GHz, 150W processors, the whole process takes about half an hour. %and standard deviation of terminal utility. The mean can be regarded as the value function and the standard deviation here is not the standard deviation of value function, but standard deviation of terminal utility.

\subsubsection{Details of Setting 2}\label{Shengbo}
Recall that when $B$ degenerates to a constant, then the Bayesian problem (\ref{eq:merton-bayes}) becomes the Merton's problem. By applying the dynamic programming principle, it suffices to consider the terminal value problem with $V(T) = 1$ and
$$\frac{dV}{dt} + \alpha V \sup_{\pi}\left\{\frac{1}{2}\sigma^2\pi^2\left(\alpha - 1\right) + \left(B - r\right)\pi + r\right\} = 0.$$
An verification argument shows that it suffices to solve
the supremum problem in the ordinary differential equation and the optimal fraction invested in the stock is a constant over time. Based on the theory from \citet{HansenSargent2001}, the Hamilton-Jacobi-Bellman-Isaacs (HJBI) equation for the distributionally robust control (DRC) is
$$\frac{dV}{dt} + \alpha V \sup_{\pi}\inf_{\nu \in \mathcal{U}_{\delta}}\left\{\frac{1}{2}\sigma^2\pi^2\left(\alpha - 1\right) + \left(E_{\nu}[B] - r\right)\pi + r\right\} = 0,$$
where we denote the distribution of $B$ as $\mu$, the uncertainty set is $\mathcal{U}_{\delta} = \{\nu: D_{\text{KL}}(\nu \parallel \mu) \leq \delta\}$, and the notation $E_{\nu}[B]$ denotes the mean of random variable $B$ if its distribution is $\nu$. Similarly as the non-robust Merton's problem, it suffices to solve the sup-inf problem and get the optimal fraction invested in the stock.

For the inner infimum problem, we formulate it as below. Distributions $\nu$ and $\mu$ share same finite support $\{b_i\}_{i=1}^d$. We denote the probability mass of two distributions $\{q_i\}_{i=1}^d$ and $\{p_i\}_{i=1}^d$ respectively.

\begin{align}
\begin{array}{ll}
\text{minimize} & \sum\limits_{i=1}^{d} q_i b_i, \\[10pt]
\text{subject to} & \sum\limits_{i=1}^{d} q_i \ln \left( \frac{q_i}{p_i} \right) \leq \delta, \\[10pt]
& \sum\limits_{i=1}^{d} q_i = 1, \quad q_i \geq 0 \quad \forall i.
\end{array}
\end{align}

After solving the infimum problem, we can directly get the optimal $\pi_{\text{DRC}}$ as in Merton problem, which is a constant across time. For the degenerate prior cases, we can get $\pi_{\text{BCPD}}$ as above. 

The experiment here is similar to Section \ref{H3.1}. In experiment, we choose the prior with $d=5$. We choose the same incorrect prior. First and second steps are the same, except for the prior now is a point mass at 0.46. Then we follow the third step to get $\pi_{\text{DRBC}}$. Finally, we run the final step with $\pi_{\text{DRC}}$, $\pi_{\text{BCPD}}$ and $\pi_{\text{DRBC}}$ to get performance metrics. 
% Using the same hardware as setting 1, we document similar time for the whole process.

\subsection{Details of High Dimensional Experiment Results} \label{highdimdetail}

We first randomly generate high dimensional SDE, then for DRBC, we calculate the empirical centers of $B$ with formula in Appendix B by optimizing $V(x_0)$. For DRC, we slightly modify the one dimensional implementation in section \ref{Shengbo} with below. 
\begin{align*}
&\text{Given: } P = (p_1, \dots, p_m), \quad z_1, \dots, z_m \in \mathbb{R}^d, \quad
B(z) = \sum_{k=1}^d z_k, \quad \delta > 0. \\[6pt]
&\text{Define the tilted distribution: }q_i(\alpha) \;=\; \frac{p_i \, e^{\alpha B(z_i)}}{\sum_{j=1}^m p_j e^{\alpha B(z_j)}}. \\[6pt]
&\text{Then: } \mu_Q(\alpha) \;=\; \sum_{i=1}^m q_i(\alpha) \, z_i 
\;=\; \frac{\sum_{i=1}^m p_i \, z_i \, e^{\alpha B(z_i)}}{\sum_{j=1}^m p_j e^{\alpha B(z_j)}}. \\[6pt]
&D_{\mathrm{KL}}(Q(\alpha) \,\|\, P)
\;=\; \alpha \, \mathbb{E}_{Q(\alpha)}[B(Z)]
\;-\; \log \!\Bigg( \sum_{j=1}^m p_j e^{\alpha B(z_j)} \Bigg). \\[6pt]
&\text{Find } \alpha^* \text{ such that } 
D_{\mathrm{KL}}(Q(\alpha^*) \,\|\, P) = \delta. \\[6pt]
&\text{Get: } \alpha^*, \quad 
\mu_{Q^*} \;=\; \mu_Q(\alpha^*).\\[6pt]
&\text{Using Merton Style Formula to get Optimal Fraction: } \pi^*=\frac{1}{1-\alpha}\Sigma^{-1}(\mu_{Q^*}-r)
\end{align*}

In both DRBC and DRC case, we choose $\delta=0.4$. All other settings are the same as section \ref{Shengbo}.

\begin{remark}
    \citet{wang2023foundation} discusses the distributionally robust control (choose the worst case in every step) formulation in the discrete state space case, and derive conditions to apply the dynamic programming approaches similar to \citet{HansenSargent2001}. We remark that these conditions are assumed in \citet{HansenSargent2001} rather than derived. A takeaway of this is that we may also do the similar theoretical foundation as in \citet{wang2023foundation} and then do the similar steps as in \citet{HansenSargent2001}: derive the HJB equation for the Bayesian problem and then get the HJBI equation for the DRBC formulation, which will have super complicated form and will be hard to solve. This is why we say the DRBC formulation looses the dynamic programming principle and another efficient method is needed to get the optimal policy.
\end{remark}

\end{document}